\documentclass[11pt,a4paper,oneside, english, abstracton]{scrartcl}

\usepackage{amsmath}		%Mathe
\usepackage{amsthm}			%Mathe
\usepackage{amscd}
\usepackage{amsfonts}
\usepackage{amssymb}		%Mathe
\usepackage{mathtools}		%Mathe
\usepackage[a4paper,left=2cm,right=2cm,top=2.7cm,bottom=3cm,bindingoffset=5mm]{geometry}	%Seite
\usepackage{rotating}
\usepackage{graphicx}
\usepackage{float}			%Bilder geraten nicht ins falsche Kapitel

\usepackage[english]{babel}		%Sprache, Silbentrennung
\usepackage[utf8]{inputenc}	%Buchstaben
\usepackage{calc}
\usepackage{thumbpdf}			%pdf
\usepackage{latexsym}
\usepackage{xcolor}
\usepackage{enumerate}		% Aufzaehlung mit (i)
\usepackage{url}				%Links in Lit
\usepackage{cite}				% sortiert citations richtig

\usepackage{algorithm} %zur Darstellung von Algorithmen
\usepackage[noend]{algpseudocode}

\definecolor{darkgreen}{rgb}{0,0.55,0}

\usepackage[bookmarks,bookmarksnumbered,colorlinks]{hyperref}
\hypersetup{linkcolor = black,anchorcolor = black,citecolor =
	black,filecolor = black,urlcolor = black}

\newtheoremstyle{break}
{\topsep}{\topsep}%
{\itshape}{}%
{\bfseries}{}%
{\newline}{}%
\newtheorem{thm}{Theorem}[section]

\newtheorem{lem}[thm]{Lemma}
\newtheorem{bsp}[thm]{Example}
\newtheorem{defn}[thm]{Definition}
\newtheorem{cor}[thm]{Corollary}

\newtheorem{rem}[thm]{Remark}
\newtheorem{ass}[thm]{Assumption}

\makeatletter
\renewenvironment{proof}[1][\proofname]{\par
	\pushQED{\qed}%
	\normalfont\topsep6\p@\@plus6\p@\relax
	\trivlist
	\item[\hskip\labelsep
	\bfseries
	#1\@addpunct{.}]\mbox{}%\par\noindent\ignorespaces
}{%
	\popQED\endtrivlist\@endpefalse
}
\makeatother

\numberwithin{equation}{section}

\algrenewcommand\algorithmicrequire{\textbf{Input:}}
\algrenewcommand\algorithmicensure{\textbf{Output:}}

\def\eps{\varepsilon}
\DeclareMathOperator*{\argmin}{arg\,min}

\newcommand{\R}{\mathbb{R}}
\newcommand{\Rp}{\mathbb{R}_{\geq 0}}
\newcommand{\N}{\mathbb{N}}		%für Buchstaben mit doppel-Strich

\newcommand{\norm}[1]{\left\lVert#1\right\rVert}

\newcommand{\X}{\mathbb{X}}
\newcommand{\U}{\mathbb{U}}

\newcommand{\K}{\mathcal{K}}
\newcommand{\B}{\mathcal{B}}
\newcommand{\LL}{\mathcal{L}}

\newcommand{\UNP}{\mathbb{U}_\mathcal{P}^N}
\newcommand{\JN}{\mathcal{J}^N}

\newcommand{\Pa}{\mathcal{P}}
\newcommand{\equ}{(x^{e},u^{e})}

\newcommand{\eq}{x^{e}}
\newcommand{\tell}{\widetilde{\ell}}
\newcommand{\tF}{\widetilde{F}}
\newcommand{\tJ}{\widetilde{J}}

\newcommand{\ub}{\textbf{u}}
\newcommand{\ubs}{\textbf{u}^\star}
\newcommand{\ubsN}{\textbf{u}^{\star,N}}
\newcommand{\posN}{\textbf{u}_{x_0}^{\star}}

\newcommand{\posNxk}{\textbf{u}_{x(k)}^{\star}}
\newcommand{\posNxn}{\textbf{u}_{x(0)}^{\star}}
\newcommand{\posNxK}{\textbf{u}_{x(K)}^{\star}}
\newcommand{\posNxkk}{\textbf{u}_{x(k+1)}^{\star}}
\newcommand{\Jni}{J_i^N}

\newcommand{\li}{\ell_i}

\begin{document}
	
	\title{Relaxed dissipativity assumptions and a simplified algorithm
		for multiobjective MPC}
	%	\author{Gabriele Eichfelder$^1$, Lars Gr\"une$^2$, Lisa Kr\"ugel$^2$, and Jonas Schie\ss{}l$^2$}
	%	\publishers{$^1$Institute of Mathematics\\Technische Universit\"at Ilmenau, Germany\\
		%	and\\
		%	$^2$Chair of Applied Mathematics, Mathematical Institute\\
		%		Universit\"at Bayreuth, Germany}
	\author{
		Gabriele Eichfelder\thanks{Institute of Mathematics, Technische Universit\"{a}t Ilmenau, Po 10 05 65, D-98684 Ilmenau, Germany, {\texttt gabriele.eichfelder@tu-ilmenau.de}, Orcid 0000-0002-1938-6316} \and 
		Lars Gr\"une\thanks{Chair of Applied Mathematics, Mathematical Institute, 	Universit\"at Bayreuth, Germany, {\texttt  lars.gruene@uni-bayreuth.de}, Orcid 0000-0002-9331-4097 } \and 
		Lisa Kr\"ugel\thanks{Chair of Applied Mathematics, Mathematical Institute, 	Universit\"at Bayreuth, Germany, {\texttt lisa.kruegel@uni-bayreuth.de}, Orcid 0000-0002-2336-4099 }
		\and Jonas Schie\ss{}l\thanks{Chair of Applied Mathematics, Mathematical Institute, 	Universit\"at Bayreuth, Germany, {\texttt jonas.schiessl@uni-bayreuth.de}}
	}
	
	\date{\today}
	\maketitle

	\begin{abstract}
		We consider nonlinear model predictive control (MPC) with multiple competing cost functions. In each step of the scheme, a multiobjective optimal control problem with a nonlinear system and terminal conditions is solved. We propose an algorithm and give performance guarantees for the resulting MPC closed loop system. Thereby, we significantly simplify the assumptions made in the literature so far by assuming strict dissipativity and the existence of a compatible terminal cost for one of the competing objective functions only. We give conditions which ensure asymptotic stability of the closed loop and, what is more, obtain performance estimates for all cost criteria. Numerical simulations on various instances illustrate our findings. The proposed algorithm requires the selection of an efficient solution in each iteration, thus we examine several selection rules and their impact on the results. \\%and we also  examine numerically how different selection rules impact the results. \\
	\end{abstract}
	
	\section{Introduction}
	{\let\thefootnote\relax\footnotetext{The authors are supported by DFG Grant Gr 1569/13-2.}}
	
	Model predictive control (MPC) is a control method in which in each sampling instant an optimal control problem is solved in order to determine the control input for the next sampling interval. Besides important system theoretic properties such as stability and constraint satisfaction, the optimization-based nature of the method also allows to conclude performance estimates of the closed loop, measured in terms of the optimization objective used for computing the control. In this paper we consider general (often also called economic) MPC formulations, for which such performance estimates have been obtained in \cite{AnAR09,AmRA11,AnAR12,Grue13,Gruene2014,GruP15a} (see also Chapter 7 in \cite{Gruene2017a} for a concise presentation). These references cover MPC schemes both with and without terminal constraints and costs. In both cases, strict dissipativity of the underlying optimal control problem is a crucial assumption.
	
	In many practical applications, as in \cite{Logist2010, Schmitt2020, Kajgaard2013}, it is desirable to consider not only a single but several cost criteria. For instance, in chemical process control, it may be desirable to stabilize a chemical process at a certain set point but at the same time approach the set point in such a way that the yield of the reaction is maximized (this is also the situation that we will illustrate in the numerical example in this paper). As these criteria might be conflicting, the resulting optimization problem is a multicriteria or multiobjective optimization problem \cite{Ehrgott2005}. In an MPC context, multiobjective optimization has been investigated, e.g., in \cite{Zavala2012,Zavala2015,Stieler2018,Gruene2019a,Flasskamp2020} in different contexts. Particularly, \cite{Stieler2018,Gruene2019a} present a multiobjective MPC algorithm for which some of the performance estimates known from classical MPC can be carried over, again using strict dissipativity as a main theoretical ingredient in the proofs. See also \cite{BFGV20} for an application of this algorithm to a PDE-governed control problem.
	
	Whenever multiple objective functions, i.e., cost criteria, are considered, one has to agree on an optimality notion used for such problems. A widely used concept is that of efficiency \cite{Ehrgott2005}: a feasible solution is called efficient in case none of the objective functions can be improved by choosing another feasible solution without a worsening of the values of at least one of the remaining functions. The vector of the corresponding values of all the objective functions under consideration is  called a nondominated point. These nondominated points form the nondominated set, which corresponds to the optimal value in case of just one objective function. In general, as the objective functions are competing, there is not a single efficient point which optimizes all objectives at the same time. Thus, the set of optimal solutions, called efficient set, is in general not unique. In single-objective MPC in each sampling instant typically an optimal solution is chosen to determine the control for the next step, and at least the optimal value of this optimal control problem is unique. In case of several competing criteria, a typically infinite number of efficient solutions as potential choices exist with non-unique values. Still, a certain efficient solution has to be chosen for the next step. This additional degree of freedom makes the extension of MPC to the multiobjective setting more challenging and requires new iterative approaches which guide this choice for guaranteeing convergence results and for allowing performance estimates. Within this work we will reach this aim by introducing upper bounds on the possible choices steered by the selected efficient solution in the first step.  Note that we do not rely on scalarization approaches like  a weighted sum of the objective functions as done in \cite{Sauerteig}, as this would require pre-knowledge for instance about the preferred weights, which is in general not available. For a discussion on the advantage of direct methods compared to scalarization based approaches we refer to \cite{Eichfelder20}. 
	
	The present paper builds on \cite{Gruene2019a} and is also closely related to \cite{Zavala2012}. More precisely, we significantly simplify the assumptions made in \cite{Gruene2019a}, by requiring strict dissipativity and the existence of a terminal cost satisfying the usual conditions from the MPC literature only for one of the optimization objectives (by default, always for the first), rather than for all of them. Despite the relaxed conditions, we obtain essentially the same results as in \cite{Gruene2019a} concerning the qualitative behavior and the averaged and non-averaged performance of the closed loop. For some of the results we can even simplify the proposed algorithm. Moreover, we develop conditions under which asymptotic stability of the closed loop (as opposed to mere convergence to a steady state) can be ensured. The close relation to \cite{Zavala2012} stems from the fact that, as in this reference, the first optimization objective plays a particular role and in particular determines the steady state to which the closed-loop solutions converge. However, the main difference of this paper to \cite{Zavala2012} is that we obtain performance estimates for all objectives.
	
	As in \cite{Gruene2019a}, our theoretical performance estimates depend on the choice of the efficient solution in the first step of the MPC iteration, i.e., at initial time. This implies that these estimates do not consider the choices of the efficient solutions in the subsequent iterations. So far, the effects of these choices on the performance of the MPC closed-loop solution were not investigated. In this paper, in addition to the theoretical results, we use numerical examples to shed some light on these effects.
	
	The paper is organized as follows: In Section \ref{sec: settings} we introduce the problems we are considering along with basic definitions and properties from multiobjective optimization as well as stability results for the single-objective case. In Section \ref{sec: algorithm} we introduce a new version of a multiobjective MPC algorithm for multiobjective optimal control problems with terminal constraints. We move on, in Section \ref{sec: performance} and \ref{sec: stability}, showing performance results for all cost criteria and a stability theorem. In Section \ref{sec: numerics} we illustrate our theoretical findings by numerical examples. Further, in Section \ref{sec:numerics2} we investigate the influence of the subsequent iterations of the algorithm by means of numerical examples. Section \ref{sec: conclusion} concludes this paper.
	
	\section{Problem formulation and Preliminaries}\label{sec: settings}
	In the following we  introduce the multiobjective optimal control problem we are considering. We give the 
	basic definitions and properties from multiobjective optimization and we recall stability
	results for the single-objective case. 
	\subsection{Multiobjective optimal control problems}
	We consider discrete time nonlinear systems of the form
	\begin{equation}\label{eq: system}
		x(k+1) = f(x(k),u(k)),\quad x(0)=x_0
	\end{equation}
	with $f: \R^n\times \R^m\to \R^n$ continuous. We denote the solution of system \eqref{eq: system} for a control sequence $\ub = (u(0),\dots, u(N-1))\in (\R^m)^{N}$ and initial value $x_0\in\R^n$ by $x_{\ub}(\cdot, x_0)$, or short by $x(\cdot)$ if there is no ambiguity about the respective control sequence and the initial value.
	
We impose nonempty state and input constraint sets $\X\subseteq \R^n$ and $\U\subseteq \R^m$, respectively, $\U^N$ as the set for control sequences of length $N$, as well as a nonempty terminal constraint set $\X_0\subseteq \R^n$, and the set of admissible control sequences for $x_0\in\X$ up to time $N\in\N$ by $\U^N(x_0):= \{ \ub\in \U^{N}\,|\, x_{\ub}(k, x_0)\in\X\; \forall\; k=1,\ldots,N-1 \text{ and}\; x_{\ub}(N, x_0)\in\X_0\}$. The terminal constraint $x_{\ub}(N,x_0)\in\X_0$ can generally not be satisfied for all initial values $x_0\in \X$, such that we define the feasible set 
	\begin{equation}\label{eq: X_N}
		\X_N:=\{x_0\in\X\mid  \exists\ \ub\in\U^N: x_\ub(k,x_0)\in\X,\ \forall\;k=1,\dots,N-1, \text{ and}\;  x_{\ub}(N,x_0)\in\X_0\},
	\end{equation}
	noting that $\U^N(x_0)\ne \emptyset$ if and only if $x_0\in\X_N$. Assumption \ref{ass: modTerminal} (iv), below, will guarantee that $\X_N\ne \emptyset$ for all $N\ge 1$.
	Further, a pair $\equ\in\X\times\U$ is called equilibrium if $\eq=f\equ$ holds.
	
	For given continuous stage cost $\ell_1:\X\times\U\to\R$ and continuous terminal cost $F_1:\X_0\to \R_{\geq 0}$, we define the cost functional $J_1^N:\X\times \U^N\to \R$ by 
	\begin{equation}\label{eq: cost_functional 1}
		J^N_1(x_0,\ub):=\sum_{k=0}^{N-1}\ell_1(x_{\ub}(k,x_0),u(k))+F_1(x_{\ub}(N,x_0)),
	\end{equation}
	and for $i\in\{2,\dots,s\}, s\geq 2$ we define continuous stage costs $\ell_i:\X\times\U\to\R$ and the corresponding cost functionals $\Jni:\X\times\U^N\to \R$ by
	\begin{equation}
		\Jni(x_0,\ub):=\sum_{k=0}^{N-1}\li(x_{\ub}(k,x_0),u(k))
	\end{equation}
	for horizon $N\in\N$ with $N\geq 2$. Here, $F_1$ is defined on the terminal constraint set $\X_0\subseteq \X$ and we need to ensure that $x_{\ub}(N,x_0)\in\X_0$, i.e.,   $J^N_1(x_0,\ub)$ is well defined for $x_0\in\X_N$ and $\ub\in \U^N(x_0)$ and we will only use it for such arguments in the remainder of this paper.
	
	We remark that we do not need any terminal costs for $i\in\{2,\dots,s\}$, which significantly simplifies the design compared to \cite{Gruene2019a,Stieler2018}. We aim on minimizing all cost functionals $J_1^N,\dots, J_s^N$ at the same time for given $x_0$ w.r.t.\ $\ub$ and along a solution of \eqref{eq: system}. Hence, we obtain a multiobjective optimal control problem with terminal constraints and costs
	\begin{equation}\tag{MO OCP}\label{MO MPC terminal}
		\begin{split}
			\min_{\ub\in\U^N(x_0)}J^N(x_0,\ub)&:= \left(J_1^N(x_0,\ub),\dots,J_s^N(x_0,\ub)\right)\\
			x(k+1)&=f(x(k),u(k)), \quad k=0,\dots,N-1\\
			x(0)&=x_0,\quad x(k)\in\X\\
			x_{\ub}(N,x_0)&\in \X_0.
		\end{split}
	\end{equation}
	Throughout this paper, we will only consider multiobjective optimal control problems with terminal constraints of the form \eqref{MO MPC terminal} with $s\geq 2$ and $N\geq2$.
	
	\subsection{Basics of multiobjective optimization}
	
	In the presence of multiple competing objectives we need an appropriate notion of optimality. In general, there will not be one optimal $\ub$ such that all cost functionals are minimized simultaneously. The formalization we will use in this paper is based on the componentwise ordering in the image space $\R^s$ and is summarized in the next definition. For this definition and an introduction to multiobjective optimization we refer, for instance, to \cite{Ehrgott2005} or the recent survey \cite{Eichfelder20}.
	\begin{defn}
		Let $N$ be the horizon length. A sequence $\ubs\in\U^N(x_0)$ is called efficient for \eqref{MO MPC terminal} with $x_0\in\X$ if there is no $\ub\in\U^N(x_0)$ such that
		\begin{align*}
			\forall\ i\in\{1,\dots,s\}:\; \Jni(x_0,\ub)&\leq \Jni(x_0,\ubs) \text{ and}\\
			\exists\ i\in\{1,\dots,s\}:\; \Jni(x_0,\ub)&< \Jni(x_0,\ubs).
		\end{align*}
		The objective value $J^N(x_0,\ubs)=(J_1^N(x_0,\ubs),\dots,J_s^N(x_0,\ubs))$ is called nondominated.
	\end{defn}
	Usually, there is not only one (unique) efficient solution of \eqref{MO MPC terminal} but there exists a set of such solutions and such nondominated values. Therefore, the set of all efficient solutions of length $N$ for initial value  $x_0\in\X_N$ will be denoted by $\UNP(x_0)$. We define the set of attainable values by
	\[\JN(x_0):=\{J^N(x_0,\ub)=(J^N_1(x_0,\ub),\dots,J^N_s(x_0,\ub))\mid \ub\in\U^N(x_0)\},\]
	and the nondominated set by
	\[\JN_\Pa(x_0):=\{J^N(x_0,\ub)\mid \ub\in\UNP(x_0)\}.\]
	This set is often referred as the Pareto front. In this paper, the $\min$-operator is defined as
	\[\min_{\ub\in\U^N(x_0)}J^N(x_0,\ub)= \JN_\Pa(x_0)\]
	and, accordingly
	\[\argmin_{\ub\in\U^N(x_0)}J^N(x_0,\ub)=\UNP(x_0).\]
	We now provide basic definitions and results from the theory of multiobjective optimization, adapted from \cite{Ehrgott2005, Sawaragi1985} to our setting.
	\begin{defn}[External stability]
		The set $\JN_\Pa(x_0)$ is called externally stable for $\JN(x_0)$ if for all $y\in\JN(x_0)$ there is $y_\Pa\in\JN_\Pa(x_0)$ such that $y\geq y_\Pa$ holds componentwise.
	\end{defn}
	\begin{defn}[Cone-Compactness]
		The set $\JN(x_0)$ is called $\Rp^s$-compact if for all $y\in\JN(x_0)$ the set $(y-\Rp^s)\cap\JN(x_0)$ is compact.
	\end{defn}
	Here, we write $z-\Rp^s$ for the difference of the sets $\{z\}$ and $\Rp^s:=\{y\in\R^s\mid y_i\geq 0\; \forall i=1,\dots,s\}$ in the Minkowski sense. The next theorem states a condition for external stability of the set $\JN_\Pa(x_0)$ and a proof can be found in \cite{Ehrgott2005, Sawaragi1985}.
	\begin{thm}\label{thm: ext_stab}
		Given a horizon $N\in\N$ and an initial value $x_0\in\X_N$. If $\JN(x_0)\neq \emptyset$ and $\JN(x_0)$ is $\Rp^s$-compact, then the set $\JN_\Pa(x_0)$ is externally stable for $\JN(x_0)$.
	\end{thm}
	Since Theorem \ref{thm: ext_stab} is in practice difficult to verify, the next lemma provides easily checkable conditions for external stability which we will need for feasibility below.
	\begin{lem}\label{lem: externStab}
		Let $\U$ be compact, $\X$ and $\X_0$ be closed and let  $x\in\X_N$ and $N\in\N$. Then the set $\JN_\Pa(x)$ is externally stable for $\JN(x)$.
	\end{lem}
	\begin{proof}
		Analogous to the proof of Lemma 4.8 in \cite{Stieler2018} or Lemma 2.5 in \cite{Gruene2019a}.
	\end{proof}
	In the single-objective case an immediate consequence of the dynamic programming principle (DPP) is that tails of optimal control sequences are again optimal control sequences. The same result holds for efficient solutions and can be found in \cite[Lemma 4.1]{Stieler2018}.
	\begin{lem}[Tails of efficient solutions are efficient solutions]
		Let $K<N$. If $\ubsN\in\UNP(x_0)$, then $\ub^{\star,K}\in\U_\Pa^{N-K}(x_{\ubsN}(K,x_0))$ with $\ub^{\star,K}:=\ubsN(\cdot+K)$ for all $K<N$, where $\ubsN(\cdot+K):=(u^{\star,N}(K),u^{\star,N}(K+1),\dots,u^{\star,N}(N-1))$.
	\end{lem}
	
	\subsection{Auxiliary results on stability and dissipativity}
	We will make use of comparison-functions defined by
	\begin{align*}
		\mathcal{K} :=\{\alpha:\Rp\to\Rp&\mid \alpha \text{ is continuous and}\\
		&\quad\text{strictly increasing with }\alpha(0)=0\}\\
		\mathcal{K}_\infty :=\{\alpha:\Rp\to\Rp&\mid\alpha\in\mathcal{K},\  \alpha \text{ is unbounded}\}\\
		\mathcal{L}:=\{\delta:\Rp\to\Rp&\mid\delta\text{ is continuous and}\\
		&\quad\text{strictly decreasing with} \lim_{t\to\infty}\delta(t)=0\}
		%\\
		%\mathcal{KL}:=\{\beta:\Rp\times\Rp\to\Rp&\mid\beta \text{ is continuous, }  \beta(\cdot,t)\in\mathcal{K}\;\forall t\in\Rp,\; \beta(r,\cdot)\in\mathcal{L}\;\forall t\in\R_{>0}\}.
	\end{align*}
	Moreover, with $\B_\eps(x_0)\subseteq \R^n$ we denote the open ball with radius $\eps>0$ around $x_0$.
	
	The MPC Algorithms 1 and 2, below, generate a solution trajectory that is referred to as the MPC closed loop. 
	For analyzing its stability, we recall the definition of a uniform time-varying Lyapunov function, which can be found in \cite[Definition 2.21]{Gruene2017a}. To this end, we consider a general time-varying discrete-time dynamical system given by
	\begin{equation} \label{eq:tvsys} x(k+1) = g(k,x(k))
	\end{equation}
	with $g:\N_0\times\X\to\X$, initial condition $x(0)=x_0$, and state space $\X\subseteq \R^n$.
	\begin{defn}[Uniform time-varying Lyapunov function]\label{def: lyap}
		Consider system \eqref{eq:tvsys}, an equilibrium $\eq\in\X$, i.e., $\eq=g(k, \eq)$, $k\in\N$, subsets of the state space $S(k)\subseteq\X$, $k\in\N_0$, and define $\mathcal{S}=\{(k,x)\mid k\in\N_0, x\in S(k)\}$. A function $V:\mathcal{S}\to\Rp$ is called uniform time-varying Lyapunov function on $\mathcal{S}$ if the following conditions are satisfied:
		\begin{enumerate}[(i)]
			\item There exist functions $\alpha_1, \alpha_2\in\K_\infty$ such that
			\[\alpha_1(\norm{x-\eq})\leq V(k,x)\leq \alpha_2(\norm{x-\eq})\]
			holds for all $(k,x)\in S$.
			\item There exists a function $\alpha_V\in\K$ such that
			\[V(k+1, g(k,x))\leq V(k,x)-\alpha_V(\norm{x-\eq})\]
			holds for all $k\in\N_0$ and $x\in S(k)$ with $g(k,x)\in S(k+1)$.
		\end{enumerate}
	\end{defn}
	The following theorem shows that the existence of such a Lyapunov function ensures asymptotic stability. For a proof we refer to \cite[Theorem 2.22]{Gruene2017a} and for the definition of the stability conditions used in its formulation to \cite[Definitions 2.14 and 2.16]{Gruene2017a}.
	\begin{thm}\label{thm: stab with lyap}
		Let $\eq$ be an equilibrium of system \eqref{eq:tvsys}, i.e., $\eq=g(k, \eq)$, $k\in\N$, and assume there exists a uniform time-varying Lyapunov function $V$ on a set $\mathcal{S}\subset\N_0\times \R^n$ as defined in Definition \ref{def: lyap}. If each $S(k)$ contains a ball $\B_\nu(\eq)$ with radius $\nu>0$ with $g(k, x)\in S(k+1)$ for all $x\in\B_\nu(\eq)$, then $\eq$ is locally asymptotically stable with $\eta=\alpha_2^{-1}\circ\alpha_1(\nu)$. If the family of sets $S(k)$ is forward invariant (i.e., if $g(k,x)\in S(k+1)$ for all $(k,x)\in \mathcal{S}$) then $\eq$ is asymptotically stable on $S(k)$. If $S(k)=\R^n$ holds for all $k\in\N_0$ then $\eq$ is globally asymptotically stable.
	\end{thm}
	Recent research has established close connections between strict dissipativity and both the stability and near-optimality of closed-loop solutions of model predictive control schemes. We will also use this notion here. To this end, we introduce strict dissipativity for a single-objective optimal control problem, meaning we consider only one cost functional, see, for instance, \cite[Section 8.2]{Gruene2017a}.
	\begin{defn}[Strict dissipativity]\label{def: strDiss}
		A single-objective optimal control problem with stage cost $\li$ is strictly dissipative at an equilibrium $\equ$ if there exists a storage function $\lambda_i:\X \to\R$ bounded from below and satisfying $\lambda_i(\eq)=0$, and a function $\rho\in\K_\infty$ such that for all $(x,u)\in\X\times\U$ the inequality
		\[\li(x,u)-\li\equ+\lambda_i(x)-\lambda_i(f(x,u))\geq \rho(\norm{x-\eq})\]
		holds.
	\end{defn}
	
	\section{A new multiobjective MPC algorithm} \label{sec: algorithm}
	
	In this section, we introduce a multiobjective  model predictive control (MO MPC) scheme that relies on solving multiobjective optimal control problems of the type \eqref{MO MPC terminal}. %Thus, we use the MPC theory to solve a multiobjective optimal control problem. 
	Since in the multiobjective case there are several "optimal" (efficient) solutions we have to adapt the "standard" MPC, e.g. see \cite{Gruene2017a}. % To this end, we analyze the case of \eqref{MO MPC terminal},
	We build on the results in \cite{Stieler2018, Gruene2019a} which we recall at the appropriate places for completeness. We introduce a simplified version of the algorithm presented in \cite{Gruene2019a}, in which we allow for more general problems of the type \eqref{MO MPC terminal} than in \cite{Gruene2019a}, and we get rid of the restrictive assumption of the existence of stabilizing stage and terminal costs in all objective functions. Rather, we require strict dissipativity of the optimal control problem and the existence of a compatible terminal cost for only one stage cost. In particular, the existence of terminal costs that are jointly 
	compatible with all the stage costs is no longer required. For two objectives, the resulting algorithm bears similarities with the one in \cite{Zavala2015}, where one stabilizing and one economic objective were considered. In this sense, we merge the ideas presented in \cite{Gruene2019a,Zavala2015}, but at the same time we also extend them and, in addition to stability results, we will also provide performance estimates for all cost criteria $\Jni$, $i\in\{1,\dots,s\}$, which are not present in \cite{Zavala2015}. Throughout, we make the convention that the optimal control problem is strictly dissipative for the first stage cost $\ell_1$.
	
	%We consider MO MPC using terminal conditions meaning that there is a terminal constraint set $\X_0\subseteq \X$ and terminal costs $F_i:\X_0\to \R_{\geq 0}$, $i\in\{1, \dots,s\}$, such that the MO OCP now reads
	%\begin{equation}\label{MO MPC terminal}
	%	\begin{split}
		%		\min_{u\in\U^N(x)}J^N(x_0,\ub),\quad &\text{with } J_i^N(x_0,\ub) :=\sum_{k=0}^{N-1}\ell_i(\xk)+F_i(x(N,x_0))\\
		%		x(k+1)&=f(x(k),u(k))\\
		%		x(0)&=x_0\\
		%		x(N,x_0)&\in\X_0.
		%	\end{split}
	%\end{equation}
	%The terminal constraint $x(N,x_0)\in\X_0$ can generally not be satisfied by all initial values $x\in \X$ such that we define the feasible set $\X_N:=\{x\in\X\mid  \exists u\in\U^N: x(k)\in\X, k=1,\dots,N-1, x(N,x_0)\in\X_0\}\neq \emptyset$.
	
	\begin{ass}\label{ass: modTerminal}
		We assume that
		\begin{enumerate}[(i)]
			\item there is an equilibrium $\equ\in\X \times\U$ such that $f\equ=\eq$
			\item there is a storage function $\lambda_1:\X\to \R$ bounded from below with $\lambda_1(\eq)=0$ and a function  $\alpha_{\ell,1}\in\mathcal{K}_\infty$ such that
			\begin{equation}\label{eq: 1dissi}
				\ell_1(x,u)-\ell_1\equ+\lambda_1(x)-\lambda_1(f(x,u))\geq \alpha_{\ell,1}(\norm{x-\eq} + \norm{u-u^e})\quad \forall (x,u)\in\X\times\U
			\end{equation}
			\item $\li$ is continuous for all $i=1,\dots,s$
			\item $\eq\in\X_0$ and there exists a local feedback $\kappa:\X_0\to\U$ satisfying
			\begin{enumerate}
				\item $f(x,\kappa(x))\in\X_0$ for all $x\in\X_0$
				\item $\forall x\in \X_0$: $F_1(f(x,\kappa(x)))+\ell_1(x,\kappa(x))\leq F_1(x) +\ell_1\equ$
			\end{enumerate}
			\item the sets $\JN_\Pa(x)$ are externally stable for  $\JN(x)$  for each $x\in\X_N$, with $\X_N$ from \eqref{eq: X_N}.
		\end{enumerate}
	\end{ass}
	We note that item (iv) of this assumption ensures $\X_0\subseteq\X_N$ for all $N\ge 1$ and thus $\X_N\ne\emptyset$. Item (iv b) is usually referred to as compatibility of the terminal cost $F_1$.
	
	Assumption \ref{ass: modTerminal} states that we only require strict dissipativity for the optimal control problem with the first stage cost, while for the remaining $s-1$ stage costs we do not impose any conditions nor the existence of terminal costs. We remark that stabilizing, i.e., positive definite stage costs are a special case of stage costs for which strict dissipativity holds with $\lambda\equiv 0$. In the following, under Assumption \ref{ass: modTerminal}, we provide performance estimates for all cost criteria $\Jni$ and show asymptotic stability of the closed loop.
	
	Algorithm \ref{alg: modMOMPC terminal}, below, gives a variant of Algorithm 2 in \cite{Gruene2019a}, in which the constraint \eqref{ineq: mpcend1} is only imposed for the first cost criterion $J_1^N$, instead of for all cost criteria $\Jni$, $i=1,\ldots,s$, as in \cite{Gruene2019a}. By using the local feedback $\kappa$ as part of the comparison control sequence in Step (2), we enforce the trajectory to converge by means of the first objective function.
	\begin{algorithm}[H]
		\caption{MO MPC with terminal conditions and constraints on $J_1$} \label{alg: modMOMPC terminal}
		\begin{algorithmic}
			\vspace*{7pt}
			\Require{MPC horizon $N$, number of iterations $K$, $x_0\in\X$ and $\kappa$ from Assumption \ref{ass: modTerminal} (iv).}
			\For $k = 0,\dots,K$:
			\begin{itemize}
				\item[(0)] If $k=0$, set $x(0)=x_0$ and choose an efficient solution $\posNxn\in\U_\mathcal{P}^N(x(0))$ of \eqref{MO MPC terminal}. Go to (2).
				\item[(1)] If $k\geq 1$, choose an efficient solution $\posNxk$ of \eqref{MO MPC terminal} with $x_0 = x(k)$ so that the inequality
				\begin{equation}\label{ineq: mpcend1}
					J_1^N(x(k),\posNxk)\leq J_1^N(x(k),\ub_{x(k)})
				\end{equation}
				holds.
				\item[(2)] For $x:=x_{u^{\star}_{x(k)}}(N,x(k))$ set
				\[\ub_{x(k+1)}:=\left(u_{x(k)}^{\star}(1),\dots,u_{x(k)}^{\star}(N-1),\kappa(x)\right).\]
				\item[(3)] Apply the feedback $\mu^N(k,x(k)):=u_{x(k)}^{\star}(0)$, i.e., evaluate $x(k+1) = f(x(k),\mu^N(k,x(k)))$, set $k=k+1$ and go to (1).
			\end{itemize}
			\EndFor
			\Ensure{MPC closed-loop trajectory $x_\mu(k,x_0) := x(k)$, $k\in\N_0$}
		\end{algorithmic}
	\end{algorithm}
	\begin{figure}[h]
		\begin{center}
			\includegraphics[scale = 0.8]{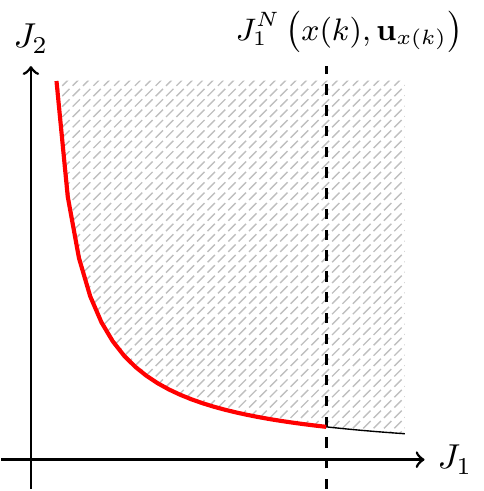}
			\caption{Visualization of step (1)}\label{fig: paretoJ1}
		\end{center}
	\end{figure}
	We have visualized the bound in \eqref{ineq: mpcend1} in Figure \ref{fig: paretoJ1}, where the dashed line represents the bound resulting from $\ub_{x(k)}^N$ and the red line the set of nondominated points of \eqref{MO MPC terminal} satisfying this bound. We like to stress that we do not constrain the second objective (or, for that matter, any other objectives that may appear in the formulation). We remark that for the subsequent theoretical results it is not important which nondominated point on the red part of the nondominated set in Figure \ref{fig: paretoJ1} we choose in step (1), as we provide performance bounds which hold for any feasible choice of $\posNxk$ and only depend on the choice of $\posNxn$ in step (0). However, this does not mean that the choices for $k\ge 1$ do not affect the MPC closed-loop solutions. In Section \ref{sec:numerics2}, we will thus investigate the development of the nondominated sets and the effect of different selection rules for the efficient point via numerical simulations. Before we do this, we show in the subsequent sections that this MO MPC algorithm has the certain desirable properties: feasibility, convergence and performance results. We start with feasibility and convergence of the closed-loop trajectory and with performance results for the cost criterion $J_1$. We note that the convergence of the closed-loop and performance results were already shown in \cite{Stieler2018, Gruene2019a} for the algorithm with stronger constraints.
	
	\section{Performance Results}\label{sec: performance}
	
	Within this section we examine the properties of the closed loop trajectory which results from Algorithm \ref{alg: modMOMPC terminal}. We obtain convergence for the trajectory as well as performance estimates for all cost criteria. 
	
	\subsection{Performance estimate of $J_1^N$}
	In this subsection we state a performance theorem for the first cost criterion $J_1^N$, which guarantees a bounded performance of the feedback $\mu^N$ defined in Algorithm \ref{alg: modMOMPC terminal}. To this end, we first show a performance estimate for the so-called rotated cost function $\tJ_1^N$ and conclude then the trajectory convergence as well as the main performance theorem of this section.
	
	First, we recall the classical definitions of rotated costs.
	\begin{defn}
		For $x\in\X$ and $u\in\U$ we define the rotated stage cost
		\begin{equation}
			\tell_1(x,u):= \ell_1(x,u)-\ell_1\equ+\lambda_1(x)-\lambda_1(f(x,u))
		\end{equation}
		with equilibrium $\eq$ and storage cost $\lambda_1$ from Assumption \ref{ass: modTerminal}, and rotated terminal cost
		\begin{equation}
			\tF_1(x):=F_1(x)+\lambda_1(x).
		\end{equation}
		The corresponding cost functional is given by
		\begin{equation}
			\tJ_1^N(x_0,\ub):=\sum_{k=0}^{N-1}\tell_1(x_{\ub}(k,x_0),u(k))+\tF_1(x_{\ub}(N,x_0)).
		\end{equation}
	\end{defn}
	We remark that, under Assumption \ref{ass: modTerminal}, for all $x\in\X$, $u\in\U$ the rotated stage cost $\tell_1$ is bounded from below by $\alpha_{\ell, 1}$ with $\alpha_{\ell, 1}$  from Assumption \ref{ass: modTerminal}(ii) and, thus, $\tell(x,u)\geq 0$. We emphasize that for implementing the MPC algorithms from this paper the rotated stage cost $\tell_1$ and the storage function $\lambda_1$ do not need to be known, as we always optimize the original stage cost $\ell_1$, whereas $\tell_1$ is only needed as an auxiliary cost for the subsequent analysis.
	
	Next, we use the definitions above to derive some relations between rotated and classical cost. It is easy to check that the relation 
	\begin{equation}\label{eq: rel_rot_norm}
		\tJ_1^N(x_0,\ub)=J_1^N(x_0,\ub)-N\ell_1\equ+\lambda_1(x_0),\quad \forall\,\ub\in\U^N(x_0),
	\end{equation}
	and the equalities $\tell_1\equ=0$ and $\tF_1(\eq)=0$ hold. Moreover, the inequality
	\begin{equation}\label{eq: F1decay}
		\tF_1(f(x,\kappa(x))\leq \tF_1(x)-\tell_1(x,\kappa(x))
	\end{equation}
	holds for all $x\in\X_N$ and $\kappa$ from Assumption \ref{ass: modTerminal} (iv).
	
	Using the definitions and relations above we can state a performance estimate for the rotated stage cost $\tell_1$. We will use this result to conclude the convergence of the closed-loop trajectory. We refer again to \cite{Stieler2018}, especially to Section 5.1.2. where similar results and proofs are provided.
	\begin{lem}[Non-averaged rotated performance for $\ell_1$]\label{lem: rotPerf}
		Let Assumptions \ref{ass: modTerminal} hold and $x_0\in\X_N$. Then, it holds
		\[\tJ_1^\infty(x_0,\mu^N):=\lim_{K\to\infty}\sum_{k=0}^{K-1}\tell_1(x_\mu(k,x_0),\mu^N(k,x_\mu(k,x_0)))\leq \tJ_1^N(x_0,\posN)\]
		with $\mu^N$ the MPC feedback defined in Algorithm \ref{alg: modMOMPC terminal}.
	\end{lem}
	\begin{proof}
		The existence of the efficient solutions in step (0) and (1) is concluded from Assumption \ref{ass: modTerminal} (v) -- the external stability of $\JN_\Pa(x)$. Feasibility of $\ub_{x(k+1,x_0)}$ in (2) follows from Assumption \ref{ass: modTerminal} (iv). Recursive feasibility of $\X$, see \cite[Theorem 3.5]{Gruene2017a} for a definition, is an immediate consequence.
		
		Further, for each $K\in\N$ it holds, with $\posNxk$ denoting the control from Algorithm \ref{alg: modMOMPC terminal},
		\begin{align*}
			\sum_{k=0}^{K-1}&\tell_1(x_\mu(k,x_0),\mu^N(k,x_\mu(k,x_0)))\\
			=\sum_{k=0}^{K-1}&\ell_1(x_\mu(k,x_0),\mu^N(k,x_\mu(k,x_0)))-\ell_1\equ+\lambda_1(x_\mu(k,x_0))-\lambda_1(x_\mu(k+1,x_0))\\
			=\sum_{k=0}^{K-1}&\left(J_1^N(x_{\posNxk}(k,x_0),\posNxk)-J_1^{N-1}(x_{\posNxk}(k+1,x_0),\posNxk(\cdot +1))\right)\\
			&-K\ell_1\equ+\lambda_1(x_0)-\lambda_1(x_\mu(K,x_0)),
		\end{align*}
		in which the first equality follows from the definition of $\tell_1$ and the second equality holds by means of the terminal condition.
		Note that we use the notation 
		$\posNxk(\cdot +1):=(u^{\star}(1),u^{\star}(2),\dots,u^{\star}(N-1))$
		Further, because of step (1) in Algorithm \ref{alg: modMOMPC terminal} and dissipativity we can estimate
		\begin{align*}
			&
			\!\begin{alignedat}{2}
				&\sum_{k=0}^{K-1}&&\left(J_1^N(x_{\posNxk}(k,x_0),\posNxk)-J_1^{N-1}(x_{\posNxk}(k+1,x_0),\posNxk(\cdot +1))\right)\\
				&	 &&-K\ell_1\equ+\lambda_1(x_0)-\lambda_1(x_\mu(K,x_0))\\
				\leq&\sum_{k=0}^{K-1}&&\left(J_1^N(x_{\posNxk}(k,x_0),\posNxk)-J_1^N(x_{\ub_{x(k)}}(k+1,x_0),\ub_{x(k+1,x_0)})+\ell_1\equ\right)\\
				&	 &&-K\ell_1\equ+\lambda_1(x_0)-\lambda_1(x_\mu(K,x_0))
			\end{alignedat}\\
			&	\leq  J_1^N(x_0,\posN)-J_1^N(x_{\ub_{x(K)}}(K,x_0),\ub_{x(K,x_0)})+\lambda_1(x_0)-\lambda_1(x(K,x_0))\\
			&	=\tJ_1^N(x_0,\posN)-\tJ_1^N(x_{\ub_{x(K)}}(K,x_0),\ub_{x(K,x_0)})\\
			&	\leq \tJ_1^N(x_0,\posN).
		\end{align*}
		Finally, letting $K$ tend to infinity and using that $\tell_1(x,u) \ge 0$, for all $x\in\X$ and $u\in\U$ yields the statement.
	\end{proof}
	
	\begin{cor}[Trajectory convergence]\label{cor: conv, with, mod}
		Consider \eqref{MO MPC terminal}. Let Assumption \ref{ass: modTerminal} hold. Then the closed-loop trajectory $x(\cdot)=x_\mu(\cdot,x_0)$ driven by the feedback $\mu^N$ from Algorithm \ref{alg: modMOMPC terminal} converges to the equilibrium $\eq$ and $\tell_1(x(k),\mu^N(k,x(k)))$ converges to $0$ as $k\to\infty$.
	\end{cor}
	\begin{proof}
		We follow the proof of Corollary 4.9 in \cite{Stieler2018}:\\
		From Theorem \ref{lem: rotPerf} it follows that the sum $\sum_{k=0}^\infty \tell_1(x(k),\mu^N(k,x(k)))$ converges and, thus, the sequence satisfies $\tell_1(x(k),\mu^N(k,x(k)))\to 0$ as $k\to\infty$. Hence, since the optimal control problem with stage cost $\tell_1$ is strictly dissipative and $\alpha_{\ell,1}\in\K$, we get
		\[0=\lim_{k\to\infty}\alpha_{\ell,1}(\norm{x(k)-\eq})=\alpha_{\ell,1}\left(\lim_{k\to\infty}\norm{x(k)-\eq}\right),\]
		%Hence, we get that for all $\eps>0$ there exists $K\in\N_0$ such that for all $k\geq K$ it holds
		%	\[\eps> |\tell_1(x(k),\mu^N(x(k)))| =\tell_1(x(k),\mu^N(x(k)))\geq \min_{u\in\U}\tell_1(x(k),u)\geq \alpha_{\ell,1}(\norm{x(k)-\eq}).\]
		%	This directly implies that $\tell_1(x(k),\mu^N(x(k)))\to 0$. Since $\alpha_{\ell,1}\in\K$ it follows that
		%	\[\alpha_{\ell,1}\left(\lim_{k\to\infty}\norm{x(k)-\eq}\right)=\lim_{k\to\infty}\alpha_{\ell,1}(\norm{x(k)-\eq})=0,\]
		which is equivalent to $\lim_{k\to\infty}\norm{x(k)-\eq}=0$.
	\end{proof}
We now carry over the estimates for $\tJ_1^\infty$ to $J_1^\infty$. To this end and for the subsequent stability analysis in Section \ref{sec: stability}, we need an additional assumption.

\begin{ass}\label{ass: cont1}
	There exist $\gamma_{F_1}\in\K_\infty$ and $\gamma_{\lambda_1}\in\K_\infty$ such that the following holds.
	\begin{enumerate}[(i)]
		\item For all $x\in\X_0$ it holds that 
		\[|F_1(x)-F_1(\eq)|\leq \gamma_{F_1}(\norm{x-\eq})\]
		and it yields that $F_1(\eq)=0$.
		\item For all $x\in\X$ it holds that 
		\[|\lambda_1(x)-\lambda_1(\eq)|\leq \gamma_{\lambda_1}(\norm{x-\eq})\]
		with $\lambda_1$ from Assumption \ref{ass: modTerminal}.
		%		\LG{Teil (iii) wird in den Abschnitten 4.2 und 5.1 nicht mehr gebraucht. Wenn das nach der \"Uberarbeitung von Abschnitt 4.3 so bleibt, kann er ganz wegfallen.}
		%		\item For each $N\in\N$ and for all $x\in\X_N$ and $\ub^{*,N}\in\UNP(x)$ it holds that 
		%		\[|J_1^N(x,\ub^{*,N})-J_1^N\equ|\leq \gamma_{J_1^N}(\norm{x-\eq}).\]
	\end{enumerate}
\end{ass}
%\LK{Assumption \ref{ass: cont1} (iii) states uniform continuity of $J_1^N$ in $\equ$. This assumption can be fulfilled if the stage cost $\ell_1$ is strictly $(x,u)-$dissipative, i.e. there exists $\alpha_{\ell,1}\in\K$ such that
	%	\[\tell_1(x,u)\geq \alpha_{\ell,1}(\norm{x-\eq}+\norm{u-u^{e}}),\quad \forall (x,u)\in\X\times\U.\]
	%	Then, the control $u(k)$ converges to the equilibrium $u^{e}$ for $k\to\infty$ with the same argumentation as in Corollary \ref{cor: conv, with, mod}. Due to the continuity of the stage cost $\ell_1$ we can conclude that $\lim_{k\to\infty}\ell_1(x(k,x_0),u(k)) = \ell_1\equ$ and, thus, that the uniform continuity of $J_1^N$ holds.} Together with this continuity and the results above we can show a performance result on $J_1$ as stated in \cite{Gruene2019a, Stieler2018}.

Using Part (ii)	of this assumption, we can show an infinite horizon performance result on $J_1$ similar to the one in \cite{Gruene2019a, Stieler2018}.

\begin{thm}[Performance estimate for $J_1$]\label{thm: perf, with, 1}
	Consider the multiobjective optimal control problem  with terminal conditions \eqref{MO MPC terminal}. Let Assumptions \ref{ass: modTerminal} and \ref{ass: cont1} (ii) hold and assume $\ell_1\equ =0$. Then, the MPC feedback $\mu^N:\N_0\times\X\to \U$ defined in Algorithm \ref{alg: modMOMPC terminal} renders the set $\X$ forward invariant and has the infinite-horizon closed-loop performance
	\begin{equation}
		J_1^\infty(x_0,\mu^N):=\sum_{k=0}^\infty \ell_1(x_\mu(k,x_0),\mu^N(k,x_\mu(k,x_0)))\leq J_1^N(x_0,\posN)
	\end{equation}
	in which $\posN$ denotes the efficient solution of step (0) in Algorithm \ref{alg: modMOMPC terminal}.
\end{thm}

\begin{proof}
	As in the proof of Lemma \ref{lem: rotPerf} the existence of the efficient solutions in step (0) and (1) in Algorithm \ref{alg: modMOMPC terminal} is, again, concluded from the external stability of $\JN_\Pa(x)$. Feasibility of $\ub_{x(k+1,x_0)}$ in (2) follows from Assumption \ref{ass: modTerminal} (iv). Forward invariance of $\X$ is an immediate consequence.
	
	Using the definition of $\tell_1$, the estimate from the proof of Lemma \ref{lem: rotPerf}, the relation \eqref{eq: rel_rot_norm}, and $\ell(x^e,u^e)=0$, it holds that 
	\begin{eqnarray*} 
		\sum_{k=0}^{K-1}\ell_1(x_\mu(k,x_0),\mu^N(k,x_\mu(k,x_0))) & = & - \lambda_1(x_0) + \sum_{k=0}^{K-1}\tell_1(x_\mu(k,x_0),\mu^N(k,x_\mu(k,x_0))) + \lambda_1(x_\mu(K,x_0))\\
		& \le &  -\lambda_1(x_0) +\widetilde J_{1}^N(x_0,\posN) + \lambda_1(x_\mu(K,x_0))\\[2ex]
		& = & J_1^N(x_0,\posN) + \lambda_1(x_\mu(K,x_0)).
	\end{eqnarray*}
	Now Assumption \ref{ass: cont1} (ii) together with the storage function $\lambda_1$ from Definition \ref{def: strDiss} with $\lambda_1(x^e)=0$ and the fact that by Corollary \ref{cor: conv, with, mod} we have $x_\mu(K,x_0)\to x^e$ implies that $\lambda_1(x_\mu(K,x_0))\to 0$ as $K\to\infty$. This shows the assertion.
	%	As the proof of Theorem 5.11 in \cite{Stieler2018} for fixed $i=1$.\\
	%	\begin{align*}
		%		&\sum_{k=0}^{K-1}\ell_1(x(k,x_0),\mu^N(x(k,x_0)))\\
		%		=&\sum_{k=0}^{K-1}J_i^N(x(k,x_0),\posNxk)-J_i^{N-1}(x(k+1,x_0),\posNxk(\cdot +1))\\
		%		\leq &J_i^N(x_0,\posN)-J_i^N(x(K,x_0),\ub_{x(K,x_0)}^{*,N})+ K\underbrace{\ell_1\equ}_{=0}.
		%	\end{align*}
	%Since we assume $\ell_1\equ=0$ we can conclude that $J_1^N\equ = \sum_{k=0}^{N-1}\ell_1\equ = 0$. Now, we combine Assumption \ref{ass: cont1}(iii), Corollary \ref{cor: conv, with, mod} and the fact that $\gamma_{J_i^N}$ is a $\K_\infty-$ function ad thus continuous. This way, we get
	%\begin{align*}
	%	\lim_{K\to\infty}|J_1^N(x(K,x_0),\ub^{*,N}_{x(K,x_0)})|&\leq \lim_{K\to\infty}\gamma_{J_1^N}(\norm{x(K,x_0)-\eq})\\
	%	&=\gamma_{J_1^N}(\lim_{K\to\infty}\norm{x(K,x_0)-\eq})=0.
	%\end{align*}
	%This implies the assertion.
\end{proof}

\begin{rem} \label{rem: averJ1}
	The proof of Theorem \ref{thm: perf, with, 1} also implies the averaged performance estimate
	\[ \limsup_{K\to\infty} \frac 1K \sum_{k=0}^{K-1} \ell_1(x_\mu(k,x_0),\mu^N(k, x_\mu(k,x_0)))= 0.\]
	In case $\ell_1(x^e,u^e) \ne 0$, this inequality holds for the shifted cost $\hat \ell_1(x,u) = \ell_1(x,u) - \ell_1(x^e,u^e)$. This implies 
	\begin{eqnarray*}&&\limsup_{K\to\infty} \frac 1K \sum_{k=0}^{K-1} \ell_1(x_\mu(k,x_0),\mu^N(k,x_\mu(k,x_0)))\\ && \qquad =  
		\limsup_{K\to\infty} \frac 1K \sum_{k=0}^{K-1} \hat \ell_1(x_\mu(k,x_0),\mu^N(k,x_\mu(k,x_0))) + \ell_1(x^e,u^e) \\
		&& \qquad \le  \ell_1(x^e,u^e).\end{eqnarray*}
	Thus, we obtain an averaged performance estimate also in the case $\ell_1(x^e,u^e)\ne 0$.
\end{rem}
In the proof above we argue that we get feasibility because of the external stability. We remark that Lemma \ref{lem: externStab} provides conditions such that external stability can be guaranteed.

\subsection{Averaged Performance Estimates for $J_i$}

Besides the performance of $J_1$ we are also interested in performance estimates for $J_i$, $i\in\{2,\dots,s\}$. Thus, we first consider the averaged performance of $J_i$, $i\in\{2,\dots,s\}$, by using the results of the previous section. Hence, the continuity of the stage costs $\li$, $i\in\{2,\dots,s\}$, and the trajectory convergence deliver the averaged performance estimate.

\begin{lem} \label{lem: lbound} For each $i=2,\ldots,s$ there is $\omega_i\in\K_\infty$ such that $|\ell_i(x,u)-\ell_i(x^e,u^e)| \le \omega_i(\tell_1(x,u))$.
\end{lem}
\begin{proof} Since $\ell_i$ is continuous, there is $\tilde\omega_i\in\K_\infty$ such that $|\ell_i(x,u)-\ell_i(x^e,u^e)| \le \tilde\omega_i(\norm{x-x^e} + \norm{u-u^e})$. Since $\tell_1(x,u) \ge \alpha_{\ell,1}(\norm{x-x^e} + \norm{u-u^e})$, the assertion follows with $\omega_i = \tilde\omega_i\circ \alpha_{\ell,1}^{-1}$.
\end{proof}

\begin{thm}\label{thm: avPerfJi}
	Consider the multiobjective optimal control problem with terminal\\ conditions~\eqref{MO MPC terminal}. Let Assumptions \ref{ass: modTerminal} and \ref{ass: cont1} hold. %, and assume that \LK{$\li(x,\ub^*)$ is bounded from above on $\X_N$ for all $\ub^*\in\UNP(x)$ and $i\in\{2,\dots,s\}$}. 
	% Das brauchen wir nicht, weil es f�r jedes feste \tilde k sowieso gilt.
	
	Then, the MPC-feedback $\mu^N:\N_0\times\X\to\U$ defined in Algorithm \ref{alg: modMOMPC terminal} has the infinite-horizon averaged closed-loop performance
	\[\bar{J}_i^\infty(x_0,\mu^N):= \limsup_{K\to\infty}\dfrac 1 K\sum_{k=0}^{K-1}\li(x_\mu(k,x_0),\mu^N(k,x_\mu(k,x_0)))\leq \li\equ\]
	for all objectives $i\in\{2,\dots,s\}$.
\end{thm}

\begin{proof}
	The existence of the efficient solutions and feasibility is ensured by Lemma \ref{lem: rotPerf} and Theorem \ref{thm: perf, with, 1}. Further, from Corollary \ref{cor: conv, with, mod} and Lemma \ref{lem: lbound} it follows that there exists $M\in\N_0$ such that for all $k\geq M$ the relation $\li(x_{\posNxk}(k,x_0),\posNxk(0))=\li\equ +\eps(k)$, $i\in\{2,\dots,s\}$, with $\eps(k)\to 0$ as $k\to\infty$, holds. Thus, given any arbitrary $\tilde\eps>0$, there exists $\tilde k\in\N_0$, $\tilde K\ge M$, such that for $k\geq \tilde k$ the error term satisfies $\eps(k)<\tilde\eps$. Thus, the sequence $(\eps(k))_{k\in\N}$ tends to $0$ for $k\to\infty$.
	
	Then, for each fixed, but arbitrary $K\in\N$ with $K>\tilde k$
	\begin{align*}
		&\dfrac 1 K \sum_{k=0}^{K-1}\li(x_\mu(k,x_0),\mu^N(k,x_\mu(k,x_0)))\\
		=&\dfrac 1 K \left(\sum_{k=0}^{\tilde k-1}\li(x_{\posNxk}(k,x_0),\posNxk(0))+\sum_{k=\tilde k}^{K-1}\li(x_{\posNxk}(k,x_0),\posNxk(0))\right)\\
		\leq & \dfrac C K + \dfrac 1 K \sum_{k=\tilde k}^{K-1}\li\equ +\underbrace{\eps(k)}_{\le \tilde \eps}\\
		\leq & \dfrac C K + \left(1-\dfrac{\tilde k}K\right)\li\equ +\left(1-\frac {\tilde k} K \right)\tilde \eps,
	\end{align*}
	where $C := \sum_{k=0}^{\widetilde k-1}\li(x_{\posNxk}(k,x_0),\posNxk(0))$ is independent of $K$ and $\posNxk$ denotes the control from Algorithm \ref{alg: modMOMPC terminal}. Letting $K\to\infty$, this implies
	\[  \bar{J}_i^\infty(x_0,\mu^N) \le \ell_i(x^e,u^e) + \tilde\eps\]
	and since $\tilde\eps>0$ was arbitrary, this shows the assertion.
\end{proof}

\section{Stability}\label{sec: stability}
In this section, we show that Algorithm \ref{alg: modMOMPC terminal} has a stability property. Therefore, we use the assumptions, results and calculations from the previous sections to formulate the following theorem. To this end, we adapt the classical stability result from the single-objective case.
\begin{thm}[Asymptotic Stability]\label{thm: as_stab}
	Consider the multiobjective optimal control problem with terminal conditions \eqref{MO MPC terminal}, which we assume to be strictly dissipative for stage cost $\ell_1$ at the equilibrium $\equ$. Let Assumption \ref{ass: modTerminal} (iv) and (v) and Assumption \ref{ass: cont1} (i) and (ii) be satisfied.
	Let $\gamma_J\in\K_\infty$ and choose the efficient solutions $\posN$ in step (0) of Algorithm~\ref{alg: modMOMPC terminal} such that they satisfy the inequality
	\begin{equation} 
		J_1^N(x_0,\posN) \le \gamma_J(\|x_0-x^e\|)  + N\ell_1(x^e,u^e). \label{eq: Jbound}
	\end{equation}%	
	Then the (optimal) equilibrium $\eq$ is asymptotically stable on $\X_N$ for the MPC closed loop defined in Algorithm \ref{alg: modMOMPC terminal}.
\end{thm}

\begin{proof}
	We follow the proof of the single-objective case and show that the modified cost functional $\tJ_1^N$ is a uniform time-varying Lyapunov function (see Definition \ref{def: lyap}) for the closed-loop system for $\eq$. Then, we can conclude, using Theorem \ref{thm: stab with lyap}%and \cite{Gruene2017a}
	, that the equilibrium $\eq$ is asymptotically stable. Without loss of generality we may assume $\ell_1\equ=0$, because replacing $\ell_1$ by $\ell_1-\ell_1\equ$ does not change the closed-loop solutions and thus not the stability.
	
	To this end we first show an auxiliary inequality. %It yields that
	%	\begin{equation}
		%		\ell_1(x(k),u^{\star,N}_{x(k)}(0))=J_1^N(x(k),\posNxk)-J_1^{N-1}(x(k+1),\posNxk(\cdot + 1)).
		%	\end{equation}
	In order to simplify the notation, we write $x$ instead of $x(k,x_0)$ and $x^+$ instead of $x(k+1,x_0)$ for the states on the MPC closed-loop solution. Then for the control sequences defined in Algorithm \ref{alg: modMOMPC terminal} it holds that
	\[ x^+ = f(x,\ub^{\star,N}_{x}(0)), \qquad x_{\ub_{x^+}}(N,x^+) = f(x_{\ub_{x^+}}(N-1,x^+), \kappa(x_{\ub_{x^+}}(N-1,x^+)))\]
	and
	\[ x_{\ub_{x^+}(\cdot+1)}(j,x^+) = x_{\ub^{\star}_{x}}(j+1,x) \qquad \mbox{for } j=0,\ldots,N-1.\]
	Moreover, we observe the relation
	\[\tJ_1^N(x,\ub^{\star}_{x})=\sum_{j=0}^{N-2} \tell_1(x_{\ub^{\star}_{x}(\cdot+1)}(j,x^+), \ub^{\star}_{x}(j+1))+\tell_1(x,\ub^{\star}_{x}(0))+F_1(x_{\ub_{x^+}(\cdot+1)}(N-1,x^+)).\]
	Using these identities and inequality \eqref{ineq: mpcend1}  it thus follows that
	\begin{eqnarray*} 
		\widetilde J_1^N(x^+,\ub^{\star}_{x^+}) & \le & \widetilde J_1^N(x^+,\ub_{x^+}) \\
		& = & \sum_{j=0}^{N-1} \tell_1(x_{\ub_{x^+}}(j,x^+), \ub_{x^+}(j)) + \widetilde F_1(x_{\ub_{x^+}}(N,x^+)) \\
		&=&\sum_{j=0}^{N-2} \tell_1(x_{\ub^{\star}_{x}(\cdot+1)}(j,x^+), \ub^{\star}_{x}(j+1))
		%& = & \underbrace{\sum_{j=0}^{N-2} \tell_1(x_{\ub^{\star}_{x}(\cdot+1)}(j,x^+), \ub^{\star}_{x}(j+1))}_{\begin{array}{l} \scriptscriptstyle = \; \widetilde J_1^N(x,\ub^{\star}_{x}) \; - \; \tell_1(x,\ub^{\star}_{x}(0))\\
		%\scriptscriptstyle \;\;\; - \; \widetilde F_1(x_{\ub_{x^+}(\cdot+1)}(N-1,x^+))\end{array}}
\\
&& + \;\; \tell_1(x_{\ub^{\star}_{x}(\cdot+1)}(N-1,x^+)), \kappa(x_{\ub^{\star}_{x^+}(\cdot+1)}(N-1,x^+))\\
&& + \;\; \widetilde F_1(x_{\ub_{x^+}}(N,x^+))\\
& = & \widetilde J_1^N(x,\ub^{\star}_{x}) - \tell_1(x,\ub^{\star}_{x}(0)) \\
&& + \;\; \tell_1(x_{\ub^{\star}_{x}(\cdot+1)}(N-1,x^+), \kappa(x_{\ub^{\star}_{x^+}(\cdot+1)}(N-1,x^+))\\
&& + \;\; \widetilde F_1(x_{\ub_{x^+}}(N,x^+)) - \widetilde F_1(x_{\ub^{\star}_{x^+}(\cdot+1)}(N-1,x^+))\\
& \le & \widetilde J_1^N(x,\ub^{\star}_{x}) - \tell_1(x,u^{\star}_{x}(0)). 
\end{eqnarray*}
In the last step,  we used inequality \eqref{eq: F1decay} with $x=x_{\ub^{\star}_{x^+}(\cdot+1)}(N-1,x^+)$ and $\ell\equ=0$.

%
%
%	\LG{In} Algorithm \ref{alg: modMOMPC terminal}, $\posNxk$ is chosen such that $u^{*,N}(N-1):=\kappa(\bar x)$ with $\kappa$ from Assumption \ref{ass: modTerminal} (iv) and $\bar x := x_{u^{*,N}}(N-1,x)$. 
%	Thus,  we can estimate
%	\begin{align*}
%		J_1^N(x,\posNx)&=\sum_{j=0}^{N-1}\ell_1(x(j,x),u_{x}^{*,N}(j))+F_1(x(N,x_0))\\
%		&=\sum_{j=0}^{N-1}\ell_1(x(j,x),u_{x}^{*,N}(j)) +\ell_1(\bar x, \kappa(\bar x))+ F_1(f(\bar x,\kappa(\bar x)))\\
%		&\leq \sum_{j=0}^{N-2}\ell_1(x(j,x),u_{x}^{*,N}(j)) + F_1(\bar x)+\ell_1\equ \\
%		&= J_1^{N-1}(x,\posN)+\ell_1\equ.
%	\end{align*}
%	From 
%	\begin{align*}
%		\ell_1(x,u_{x}^{*,N}(0))&=J_1^N(x,\posNx)-J_1^{N-1}(x^+,\posNx(\cdot + 1))\\
%		&\leq J_1^N(x,\posNx)-J_1^N(x^+,\posNx(\cdot +1))+\ell_1\equ
%	\end{align*}
%	follows that
%	\begin{align}
%		\tJ_1^N(x_0,\posN)&=\sum_{k=0}^{N-1}\ell_1(x(k,x_0),u_{x(k)}^{*,N}(0))-\ell_1\equ+\lambda_1(x(k,x_0))-\lambda_1(x(k+1,x_0))\nonumber\\
%		&\qquad\qquad+F_1(x(N,x_0))+\lambda_1(x(N,x_0))\nonumber \\
%		&=J_1^N(x_0,\posN)-N\ell_1\equ +\lambda_1(x_0)\nonumber\\
%		&\geq J_1^N(f(x_0,u_{x_0}^{*,N}(0)),\posN(\cdot+1))\nonumber\\
%		&\qquad\qquad+\ell_1(x_0,u_{x_0}^{*,N}(0))-\ell_1\ggw+\lambda_1(x_0)-N\ell_1\equ \nonumber\\
%		&=\tJ_1^N(f(x_0,u_{x_0}^{*,N}(0)),\posN(\cdot+1))+\tell_1(x_0,u_{x_0}^{*,N}(0)).
%	\end{align}
We will now check that $V(k,x) := \tJ_1^N(x,\ub^{\star}_{x})$, with $\ub^{\star}_{x}$ denoting the control from the $k$-th step of the MPC iteration, is a uniform time-varying Lyapunov function according to Definition \ref{def: lyap} for $g(k,x) = f(x,\mu^N(k,x))$.
In order to do this, we show  the existence of $\alpha_1,\alpha_2,\alpha_3\in\K_\infty$, such that the inequalities 
\begin{enumerate}[(i)]
\item $\alpha_1(\norm{x_0-\eq})\leq \	\tJ_1^N(x_0,\posN)\leq \alpha_2(\norm{x_0-\eq})$
\item $\tell_1(x,u)\geq \alpha_3(\norm{x-\eq})$
\end{enumerate}
hold for all $x,x_0\in\X$. Condition (ii) is satisfied by our strict dissipativity assumption with $\alpha_3 = \alpha_{\ell,1}$. For the inequalities in condition (i) we first need to establish a lower bound for $\tF_1$. We recall Assumption \ref{ass: modTerminal} (iv) with local feedback $\kappa$ for each $x\in\X_0$. Then, Assumption \ref{ass: modTerminal} and strict dissipativity imply
\[\tF_1(f(x,\kappa(x)))\leq \tF_1(x)-\tell_1(x,\kappa(x))\leq \tF_1(x)-\alpha_{\ell,1}(\norm{x-\eq}).\]
%which is independent from the choice of the control as long as the assumption is fulfilled. 
By induction along the closed-loop solution for the local feedback $\kappa$ we then obtain
\[\tF_1(x_\kappa(K,x))\leq \tF_1(x)-\sum_{k=0}^{K-1}\alpha_{\ell,1}(\norm{x_\kappa(k,x)-\eq}).\]
By Assumption \ref{ass: cont1} (i) and (ii) and Corollary \ref{cor: conv, with, mod} this implies $\tF_1(x_\kappa(K,x))\to \tF(\eq)=0$ as $K\to\infty$ from which we can conclude
\[\tF_1(x)\geq \lim_{K\to \infty}\sum_{k=0}^{K-1}\alpha_{\ell,1}(\norm{x_\kappa(k,x)-\eq})\geq  \alpha_{\ell,1}(\norm{x-\eq})\geq 0.\]
From this, the definition of $\tJ_1^N$ immediately implies $\tJ_1^N(x_0,\posN)\geq \tell_1(x_0,\mu^N)\geq \alpha_{\ell,1}(\norm{x_0-\eq})$ and thus the inequality for $\alpha_1$ with $\alpha_1=\alpha_{\ell,1}$.

Finally, since $\tJ_1^N\equ =0$ and due to Assumption \ref{ass: cont1} (ii), the (in)equalities \eqref{eq: Jbound}, \eqref{eq: rel_rot_norm}, and $\ell\equ=0$ it follows that $\alpha_2 = \gamma_{\lambda_1}+\gamma_{J_1^N}$.
\end{proof}

Observe that in the case of stabilizing stage costs, we obtain $\lambda_1\equiv 0$ and $\ell_1\equ=0$, and thus $\tJ_1^N = J_1^N$. This implies that the objective function itself is a Lyapunov function.

\begin{rem} (i) It is not a priori clear that inequality \eqref{eq: Jbound} can be satisfied. In order to guarantee this, techniques similar to those used, e.g., in \cite[Proposition 5.14]{Gruene2017a} or \cite[Propositions 2.15 or 2.16]{RaMD17} for single-objective MPC could be used.\\[2mm]
(ii) If inequality \eqref{eq: Jbound} can be satisfied, then it will restrict the choice of the efficient point in step~(0) of Algorithm \ref{alg: modMOMPC terminal}. In particular, enforcing \eqref{eq: Jbound} will typically require to put more emphasis on the cost $J_1^N$ at the expense that the performance of $J_i^N$ for $i\ge 2$ may deteriorate.
\end{rem}

\subsection{Non-averaged Performance Estimates on $J_i$}
In the next section we aim to show a non-averaged performance result on $J_i$ for $i=2,\dots,s$. For this purpose we will use the results of the previous sections and combine them with the idea of the performance of single-objective economic MPC without terminal conditions, see \cite{Gruene2017a}. To this end, we consider the trajectories $x$ which are driven by the efficient solution $\posN$. We denote these trajectories by $x(\cdot)=x_{\posN}(\cdot,x_0)$ and name them efficient trajectories.

First, we show that the end points of the efficient trajectories are close to the equilibrium because of stability and strict dissipativity for the stage cost $\ell_1$.

\begin{lem} \label{lem: endbound}Let Assumptions \ref{ass: modTerminal} and \ref{ass: cont1} hold and consider efficient trajectories $x(j) = x_{\posN}(j,x_0)$, $j=0,\ldots,N$, for which there is $\gamma_J\in\K_\infty$ such that 
\begin{equation} 
J_1^{N-j}(x(j),\posN(j+\cdot))- N\ell_1\equ \le \gamma_J(\|x(j)-\eq\|) \qquad \mbox{for all } j=0,\ldots,N. \label{eq: Jboundall}
\end{equation}
Then there are $\rho_1,\rho_2\in\K_\infty$ such that for all $N\in\N$ the final points on the trajectories satisfy
\[\norm{x(N) -\eq}\le \rho_1(\rho_2(\|x_0-x^e\|)/N).\]
\end{lem}
\begin{proof} 
From \eqref{eq: Jboundall} and \eqref{eq: rel_rot_norm} we obtain that 
\[ \widetilde J_1^N(x(j),\posN(j+\cdot)) \le \gamma_J(\|x(j)-x^e\|) + \gamma_{\lambda_1}(\|x(j)-x^e\|) =:\rho_2(\|x(j)-x^e\|). \]
Using this inequality for $j=0$ implies that there exists a time index $j_0\in\{0,\ldots,N\}$ such that $\tell_1(x(j_0),\posN(j_0)) \le \rho_2(\|x(j_0)-x^e\|)/N$ if $j_0<N$ or $\widetilde F_1(x(N))\le \rho_2(\|x(j_0)-x^e\|)/N$. If $j_0<N$, then using the lower bound from the dissipativity $\alpha_{\ell,1}$ on $\tell_1$ and, again, the inequality above it follows that  
\[ \widetilde F_1(x(N)) \le \widetilde J_1^{N-j_0}(x(j_0),\posN(j_0+\cdot)) \le \rho_2\left(\alpha_{\ell,1}^{-1}(\rho_2(\|x(j_0)-x^e\|)/N)\right). \]
Since, as shown in the proof of Theorem \ref{thm: as_stab}, $\alpha_{\ell,1}$ is also a lower bound on $\widetilde F_1$, we obtain
\[ \|x(N)-x^e\| \le \max\{ \alpha_{\ell,1}^{-1}(\rho_2(\|x(j_0)-x^e\|)/N), \alpha_{\ell,1}^{-1}\circ\rho_2\circ\alpha_{\ell,1}^{-1}(\rho_2(\|x(j_0)-x^e\|)/N)\}.\]
This implies the assertion with $\rho_1(r)= \max\{\alpha_{\ell,1}^{-1}(r),\alpha_{\ell,1}^{-1}\circ\rho_2\circ\alpha_{\ell,1}^{-1}(r)\}$.
\end{proof}

Next, in order to establish a performance estimate on $J_i$, $i=2,\dots,s$ we extend the constraint \eqref{ineq: mpcend1} to all $i=1,\ldots,s$. In this way we end up with an algorithm originally proposed in \cite{Gruene2019a}. However, different from \cite{Gruene2019a}, for the subsequent results we still do not require additional properties of the stage cost $\li$ for $i\ge 2$. We can avoid these conditions by exploiting that the feedback from Assumption \ref{ass: modTerminal}(iv) steers a state $x$ to the equilibrium $\eq$.

\begin{ass}\label{ass: kappabound} For each $i=2,\ldots,s$ there is $\gamma_i\in\K_\infty$ such that
\[ \ell_i(x,\kappa(x)) \le \ell_i(x^e,u^e) + \gamma_i(\|x-x^e\|) \]
holds for all $x\in\X_0$ and $\kappa$ from Assumption \ref{ass: modTerminal}(iv). 
\end{ass}

\begin{algorithm}[H] \caption{MO MPC with terminal conditions and constraints on $J_1,\ldots,J_s$}\label{alg: modMOMPC2}
\begin{algorithmic}
\vspace*{7pt}
\Require{MPC Horizon $N$, number of iterations $K$, $x_0\in\X$ and $\kappa$ from Assumption \ref{ass: modTerminal} (iv).}
\For $k = 0,\dots,K$:
\begin{itemize}
	\item[(0)] If $k=0$, set $x(0)=x_0$ and choose an efficient solution $\posNxn\in\U_\mathcal{P}^N(x(0))$ of \eqref{MO MPC terminal}. Go to (2).
	\item[(1)] If $k\geq 1$, choose a efficient solution $\posNxk$ of \eqref{MO MPC terminal} with $x_0 = x(k)$ so that the inequalities
	\begin{equation}\label{ineq: mpc_constr}
		J_i^N(x(k),\posNxk)\leq J_i^N(x(k),\ub_{x(k)}), \quad i=1,\dots,s.
	\end{equation}
	hold.
	\item[(2)] For $x:=x_{u^{\star}_{x(k)}}(N,x(k))$ set
	\[\ub_{x(k+1)}:=\left(u_{x(k)}^{\star}(1),\dots,u_{x(k)}^{\star}(N-1),\kappa(x)\right).\]
	\item[(3)] Apply the feedback $\mu^N(k,x(k)):=u_{x(k)}^{\star}(0)$, i.e., evaluate $x(k+1) = f(x(k),\mu^N(k,x(k)))$, set $k=k+1$ and go to (1).
\end{itemize}
\EndFor
\Ensure{MPC closed-loop trajectory $x_\mu(k,x_0) := x(k)$, $k\in\N_0$}
\end{algorithmic}
\end{algorithm}
\begin{figure}[h]
\begin{center}
\includegraphics[scale = 0.8]{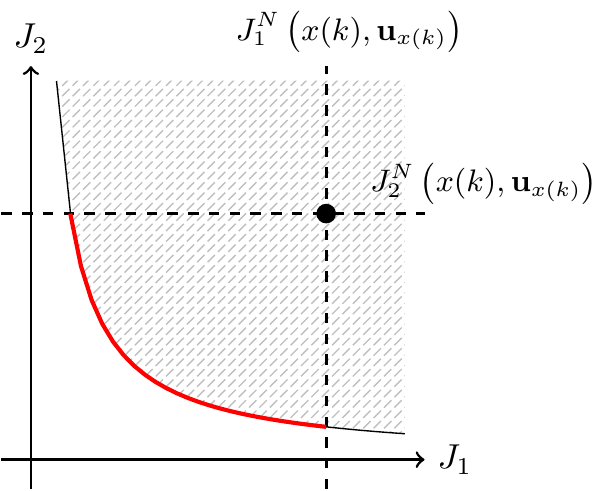}
\caption{Visualization of step (1)}\label{fig: paretoJ}
\end{center}
\end{figure}
In Figure \ref{fig: paretoJ} we have visualized the bounds of step (1) in Algorithm \ref{alg: modMOMPC2}. The dashed lines represent the bounds \eqref{ineq: mpc_constr} and the red line the set of nondominated points satisfying these inequalities. We remark that now all cost criteria $\Jni$ are bounded. For Algorithm \ref{alg: modMOMPC2} we can state the following performance result.
\begin{thm} [Performance Estimate for $J_i$] \label{thm: perf}
Let Assumptions \ref{ass: modTerminal}, \ref{ass: cont1}, and \ref{ass: kappabound} hold and assume that the efficient solutions generated by Algorithm \ref{alg: modMOMPC2} satisfy the inequalities \eqref{eq: Jboundall} for some $\gamma_J\in\K_\infty$. Then for all $i=2,\ldots, s$ and for any $C>0$ there is a function $\delta_i\in\LL$ such that
\[ J_i^K(x_0,\mu^N) \le J_i^N(x_0,\posN) + (K-N)\ell_i(x^e,u^e) + K\delta_i(N)\]
for all $N,K\in\N$ with $K\ge N$ and all $x_0\in\X_N$ with $\|x_0-x^e\| \le C$.
\end{thm}
\begin{proof} 
Consider the controls $\posNxk$ and $\ub_{x(k+1)}$ from Algorithm \ref{alg: modMOMPC2}, where $x(k)$ denotes the closed loop solution generated by the algorithm. Then Lemma \ref{lem: endbound} and Assumption \ref{ass: kappabound} imply
\[ J_i^N(x(k+1),\ub_{x(k+1)}) \le J_i^N(x(k),\posNxk) - \ell_i(x(k),\mu^N(k,x(k))) + \ell_i(x^e,u^e) + \gamma_i(\rho_1(\rho_2(\widetilde C)/N)).\]
Here we used that the asymptotic stability property of the closed loop from Theorem \ref{thm: as_stab} implies that whenever the initial value satisfies $\|x_0-x^e\|\le C$, then there is $\widetilde C>0$ such that $\|x(k)-x^e\|\le \widetilde C$ for all $k\in\N$. This inequality together with inequality \eqref{ineq: mpc_constr} for $i=2,\dots,s$ implies
\begin{eqnarray*} J_i^K(x_0,\mu^N) &  = & \sum_{k=0}^{K-1} \ell_i(x(k),\mu^N(k,x(k)))\\
& \le & \sum_{k=0}^{K-1}  \left(  J_i^N(x(k),\posNxk)  - J_i^N(x(k+1),\ub_{x(k+1)}) + \ell_i(x^e,u^e) +\gamma_i( \rho_1(\rho_2(\widetilde C)/N))\right)\\
& \le & \sum_{k=0}^{K-1}  \left(  J_i^N(x(k),\posNxk)  - J_i^N(x(k+1),\posNxkk) + \ell_i(x^e,u^e) +\gamma_i( \rho_1(\rho_2(\widetilde C)/N))\right)\\
& = & J_i^N(x_0,\posN) - J_i^N(x(K),\posNxK) + K\ell_i(x^e,u^e) + K\gamma_i(\rho_1(\rho_2(\widetilde C)/N)).
\end{eqnarray*}
Now again the asymptotic stability and the boundedness of $\|x_0-x^e\|$ imply the existence of $\chi\in\LL$ such that $\|x(k)-x^e\| \le \chi(k)$. 

This implies $\widetilde J_1^N(x(k),\posNxk)\le \gamma_J(\chi(k))$ and the individual (nonnegative) terms of this sum also satisfy $\tell_1(x_{\posNxk}(j,x(k)), \posNxk(j)) \le \gamma_J(\chi(k))$ for all $j=0,\ldots,N-1$. By Lemma \ref{lem: lbound} this yields $\ell_i(x_{\posNxk}(j,x(k)), \posNxk(j)) \ge \ell_i(x^e,u^e) - \omega_i(\gamma_J(\chi(k)))$ and we can conclude that
\[ J_i^N(x(K),\posNxK) \ge N\ell_i(x^e,u^e) - N\omega_i(\gamma_J(\chi(K))) \ge N\ell_i(x^e,u^e) - K\omega_i(\gamma_J(\chi(N))), \]
where we used $K\ge N$ and the monotonicity of $\chi\in\LL$ in the last step. This yields the assertion with 
\[ \delta_i(N) =  \rho_1(\rho_2(\widetilde C)/N) + \omega_i(\gamma_J(\chi(N))). \]
\end{proof}

\begin{rem} The fact that the error term $K\delta(N)$ grows linearly with $K$ might at the first glance make the estimate useless.  However, unless we are in the very special case that $\ell_i(x^e,u^e)=0$, the term $J_i^N(x_0,\posN) + (K-N)\ell_i(x^e,u^e)$ also grows affine linearly with $K$. Hence, for all sufficiently large $K$ the relative error is proportional to $\delta(N)$ and thus decreases to $0$ as $N$ tends to infinity. Hence, in terms of the relative error the estimate gives a perfectly useful estimate. We note that this estimate is structurally similar to estimates for the closed-loop performance of single-objective economic MPC without terminal conditions, see, e.g., \cite[Theorem 8.39]{Gruene2017a}.
\end{rem}

\begin{rem} The inequalities in \eqref{eq: Jboundall} restrict the efficient solutions, for which the performance statement in Theorem \ref{thm: perf} holds. If $\gamma_J\in\K_\infty$ can be computed or is known, the inequalities could be added as additional constraints in the optimization routine.
\end{rem}

\section{Numerical simulations}\label{sec: numerics}
The aim of this section is to illustrate the theoretical results of the previous sections. In the following we distinguish between the efficient solutions chosen in the different steps of Algorithm \ref{alg: modMOMPC terminal} and \ref{alg: modMOMPC2}. To this end, we introduce the following denominations for the efficient solutions in the algorithms:
\begin{itemize}
\item the efficient solution $\posN$ chosen in step (0), i.e., in the first iteration, we call \textit{the first efficient solution}.
\item the efficient solutions $\posNxk$ chosen in step (1), i.e., from iteration step $k=2$ onwards, we name \textit{the subsequent efficient solutions.} 
\end{itemize} 

For verifying the theoretical results we use the example of a chemical reactor, see \cite{Diehl2011, Zavala2015}.
\begin{bsp}[Reactor Part 1]\label{ex: reactor1}
We consider a single first-order, irreversible chemical reaction in an isothermal continuous stirred-tank reactor (CSTR)
\[A \to B \quad r=k_rc_A\]
in which $k_r = 1.2$ is the rate constant. The material balances and the system data are provided in \cite{Diehl2011} whereas the stage costs -- a tracking type cost forcing the solutions to a desired equilibrium and an economic stage cost maximizing the yield (by minimizing the negative yield) of the reaction -- are introduced in \cite{Zavala2015}. The overall
bi-objective optimal control problem  
is given by
\begin{align}\label{reactor}
\min_{\ub\in\U^N(c_0)} J^N(c_0,\ub)&=\left(\sum_{k=0}^{N-1}\ell_1(c(k,c_0),u(k)), \sum_{k=0}^{N-1}\ell_2(c(k,c_0),u(k))\right),\quad \nonumber\\
\text{s.t.}\quad	c_A(k+1)&=c_A(k)+\frac 1 2 \left(\frac{u(k)}{V}(c_{A_f}-c_A(k))-k_r{c_A(k)}\right),\nonumber\\
c_B(k+1)&=c_B(k)+\frac 1 2 \left(\frac{u(k)}{V}(c_{B_f}-c_B(k))+k_r{c_B(k)}\right),\nonumber\\
c(0)&=c_0=(0.4,0.2)\nonumber\\
c(N,c_0)\in\X_0&=\{(c^{e},u^{e})\}\nonumber\\
\X&=[0,20]\times[0,20],\quad \U=[0,20].
\end{align}
We consider the molar concentrations $c_A(k)\geq 0$ and $c_B(k)\geq 0$, $k\in\N$, of $A$ and $B$ respectively, and $0\leq u(k)\leq 20$(L/min) is the flow through the reactor at time $k$. The feed concentrations of $A$ and $B$ are given by $c_{A_f}=1$ mol/L, and $c_{B_f}=0$ mol/L respectively. The volume of the reactor is given by $V =10$ L. Further, we abbreviate the states by $c = (c_A, c_B)$ and we consider two stage costs given by
\begin{align*}
\ell_1(c,u)&=\frac 1 2(c_A-\frac 1 2)^2+\frac 1 2 (c_B-\frac 1 2)^2+\frac 1 2 (u -12)^2,\\
\ell_2(c,u)&=-2u c_B + \frac 1 2 u,
\end{align*}
where the second stage cost consists of the price of $B$ and a separation cost. These second costs therefore represent the (negative) economic yield of the reaction. Further, we set the terminal cost to zero, i.e. $F_i\equiv 0$ for $i=1,2$. The equilibrium under consideration of the system in \eqref{reactor} is given by $(c_A^{e},c_B^{e},u^{e})=(c^{e},u^{e})=(\frac 1 2, \frac 1 2, 12)$, which we also set as the terminal constraint, i.e. $\X_0=\{(c^{e},u^{e})\}$. This way, Assumption \ref{ass: modTerminal} is fulfilled since stabilizing stage costs always render the optimal control problem strictly dissipative and by setting $\kappa = u^{e}$ there exists a local feedback with the desired properties. Hence, Assumption \ref{ass: cont1} is also satisfied. By imposing box constraints $\U$ and $\X$ we can conclude external stability of $\JN_\Pa(c_0)$ and, thus, the trajectory convergence as well as the averaged and non-averaged performance of the first cost criterion $J_1$ by Corollary \ref{cor: conv, with, mod} and Theorem \ref{thm: perf, with, 1}.
\begin{figure}[h]
\begin{center}
	\includegraphics[scale = 0.3]{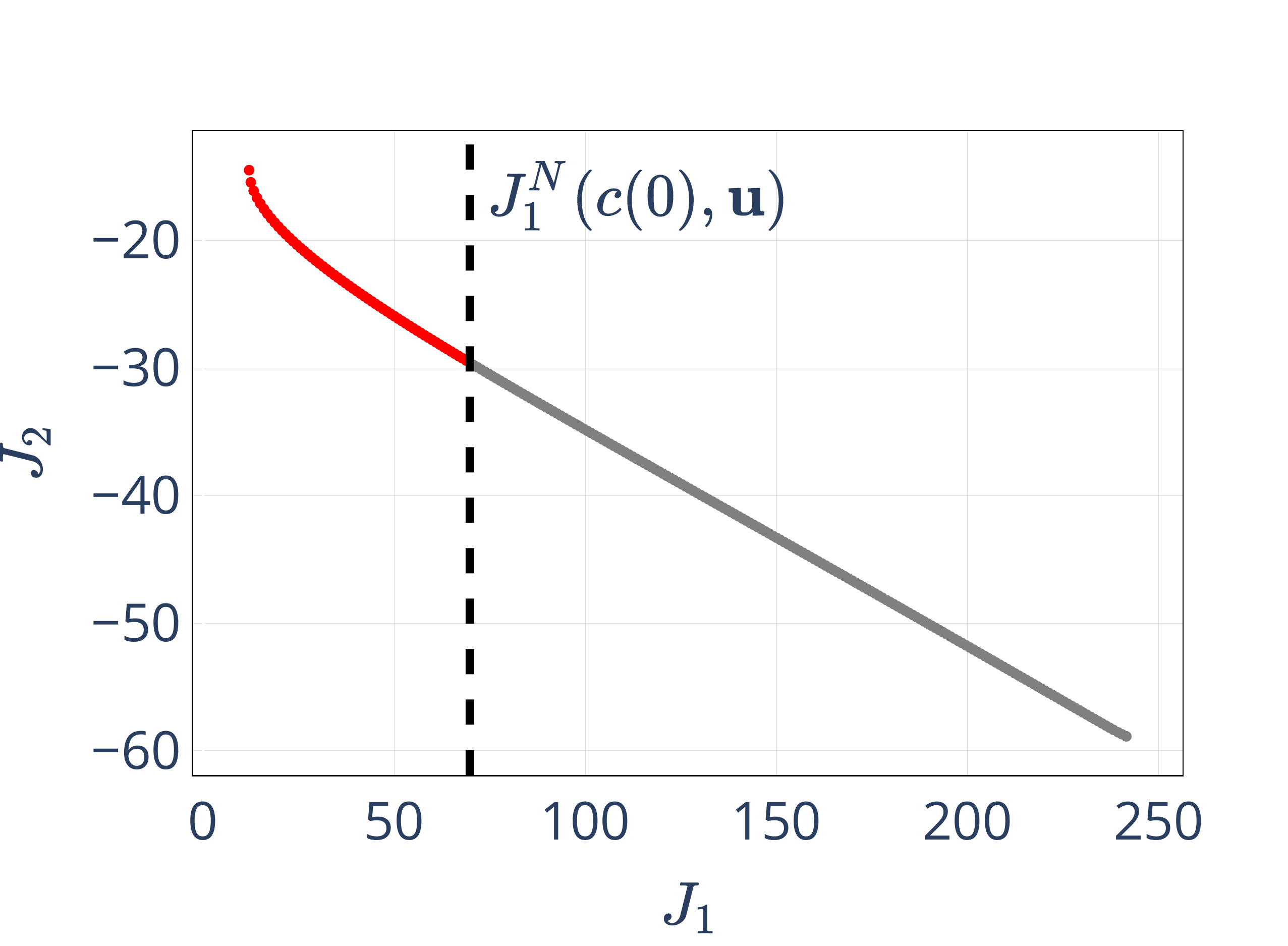}
	\caption{Visualization of step (1)}\label{fig: reactorBoundJ1}
\end{center}
\end{figure}

\begin{figure}[h]
\begin{center}
	\begin{minipage}{0.49\textwidth}
		\includegraphics[width =\textwidth]{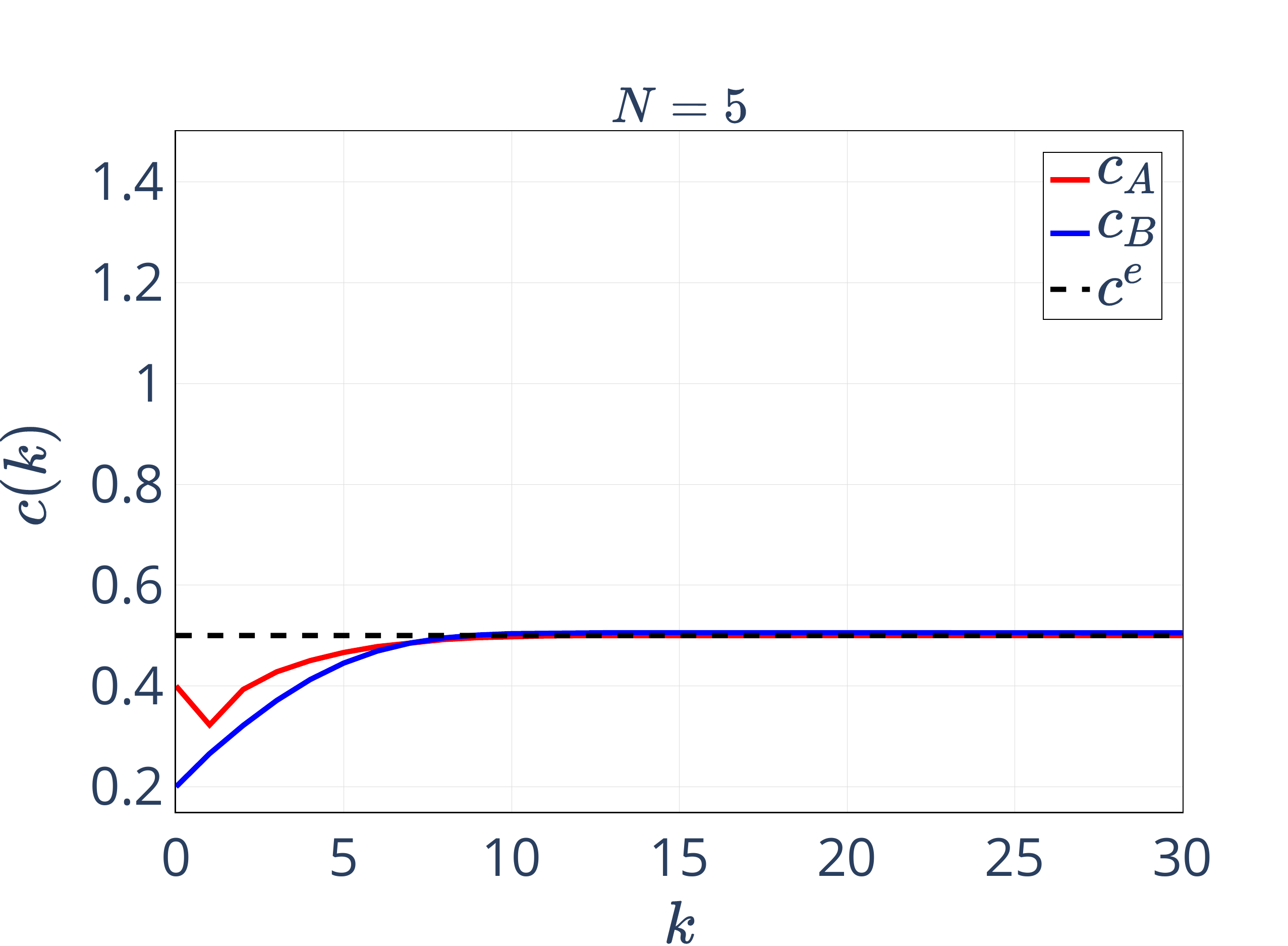}
	\end{minipage}
	\begin{minipage}{0.49\textwidth}
		\includegraphics[width =\textwidth]{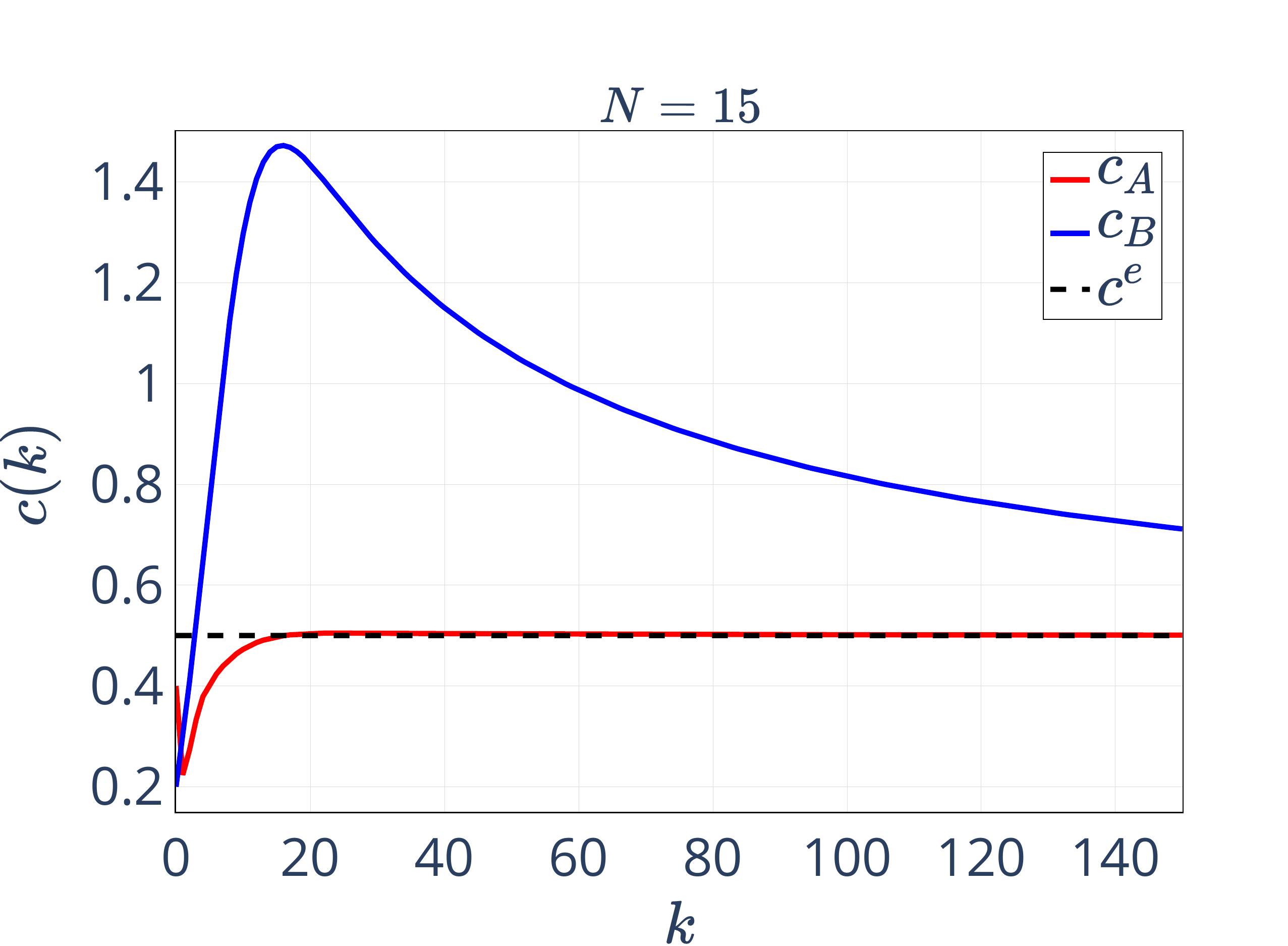}
	\end{minipage}
	\caption{Closed-loop trajectory for $N=5$ (left) and $N = 15$ (right)}\label{fig: tra}
\end{center}
\end{figure}

In this example we use 
\begin{itemize}
\item Algorithm \ref{alg: modMOMPC terminal} to substantiate our theoretical results with numerical simulations. Thus, we restrict only the first objective by the constraint in step (1) of the algorithm. The resulting bound on the nondominated set in the second iteration is visualized in Figure \ref{fig: reactorBoundJ1};
\item the ASMO Solver \cite{GithubASMO}, a solver  for nonlinear  multiobjective optimization, to generate an approximation of the nondominated set in the first iteration for choosing the first efficient solution. ASMO is an implementation of the algorithm presented in \cite{Eichfelder2008,Eichfelder2009} which combines the 
Pascoletti-Serafini scalarization with an adaptive parameter control to achieve an approximation of the nondominated set with well distributed approximation points;  
\item as the first efficient solution the efficient solution with $J^5(x_0,\ub_{c_0}^\star)=(54.034, 9.500)$ for $N=5$, and with $J^{15}(c_0,\ub_{c_0}^\star)=(408.459,-478.459)$ for $N=15$;
\item the global criterion method, also known as compromise programming approach, see, for instance, \cite{Miettinen1998}, to find efficient solutions of the multiobjective optimization problems  in the subsequent iterations as proposed in a multiobjective MPC context in \cite{Zavala2012, Zavala2015}. This means that the subsequent efficient solution $\posNxk$  is chosen in each iteration such that 
\begin{equation}\label{eq:defideal}
	\begin{array}{rcl}
		\posNxk \in\argmin %=\min_{\ub\in\JN_\Pa(x(k,x_0))}\left(
		\bigg\{\left(
		\sum\limits_{i=1}^{s}|\Jni(x(k),\ub)-z_i^\star|^2\right)^{\frac 1 2}\,&\bigg|\,&
		\ub\in\U^N(x_0),\\
		&& J_1^N(x(k),\posNxk)\leq J_1^N(x(k),\ub_{x(k)})\bigg\},
\end{array}	\end{equation} 
where 
\[z_i^\star=\min \left\{\left. \Jni(x(k),\ub)\,\right|\ \ub\in\U^N(x_0),\ J_1^N(x(k),\posNxk)\leq J_1^N(x(k),\ub_{x(k)}) \right\},\]
for all $i=1,\ldots,s$, is set as the so called %utopia point or 
ideal point of the restricted problem, cf.\  \cite{Zavala2012};  i.e., $\posNxk$ is defined as pre-image of  the  nondominated point   which has the smallest Euclidean distance to the ideal point. Whenever applying Algorithm \ref{alg: modMOMPC2}  instead of Algorithm \ref{alg: modMOMPC terminal}, then $J_1^N(x(k),\posNxk)\leq J_1^N(x(k),\ub_{x(k)})$ is replaced by $J_i^N(x(k),\posNxk)\leq J_i^N(x(k),\ub_{x(k)})$, $i=1,\ldots,s$ in the above optimization problems. 
\end{itemize}
We note that for this example the optimization problems contained as subproblems in these algorithms are non-convex, hence we have no theoretical guarantees that the numerical optimization reached a globally optimal solution. However, the numerical results strongly suggest that globally optimal solutions were found in all our numerical simulations.

The behavior of the closed-loop trajectory is visualized in Figure \ref{fig: tra} for MPC-horizons $N=5$ and $N=15$. We observe that the trajectories converge to the equilibrium $c^e$ independently of the choice of MPC-horizon and the initial value. However, the MPC-horizon $N$ influences the convergence rate. On the left side, for $N=5$, the $c_B$-trajectory converges faster to the equilibrium point $c^{e} =\frac 1 2$ than for $N=15$. We note that this is a typical behavior of MPC with equilibrium terminal constraints, see \cite[Discussion after Ex.\ 7.23]{Gruene2017a}.
In addition, the components of the trajectory show different transient behavior.
\begin{figure}[h]
\begin{center}
	\begin{minipage}{0.49\textwidth}
		\includegraphics[width =\textwidth]{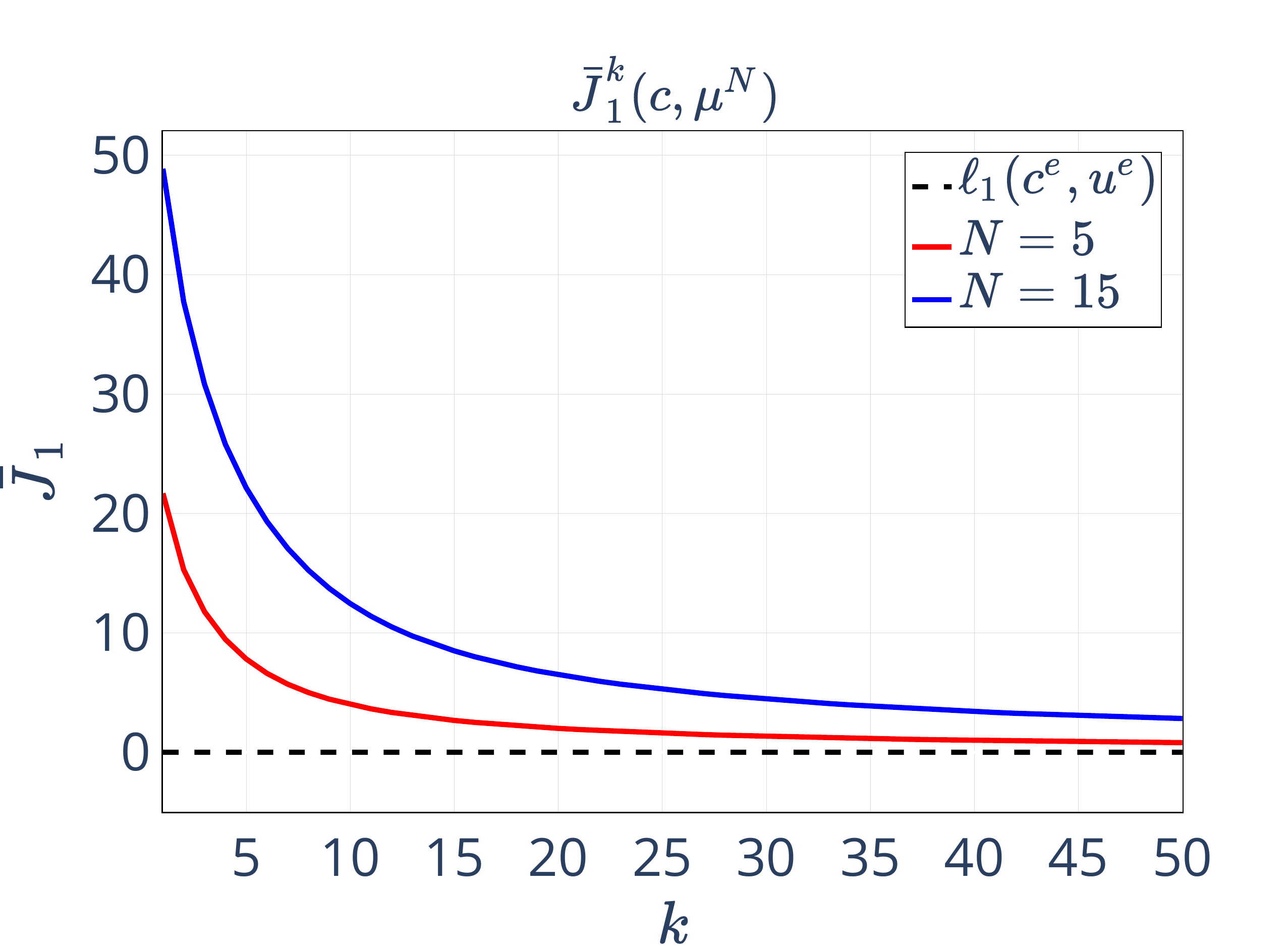}
		\caption{$\bar J_1^k$ for $N=5$ and $N = 15$}\label{fig: avJ1}
	\end{minipage}
	\begin{minipage}{0.49\textwidth}
		\includegraphics[width =\textwidth]{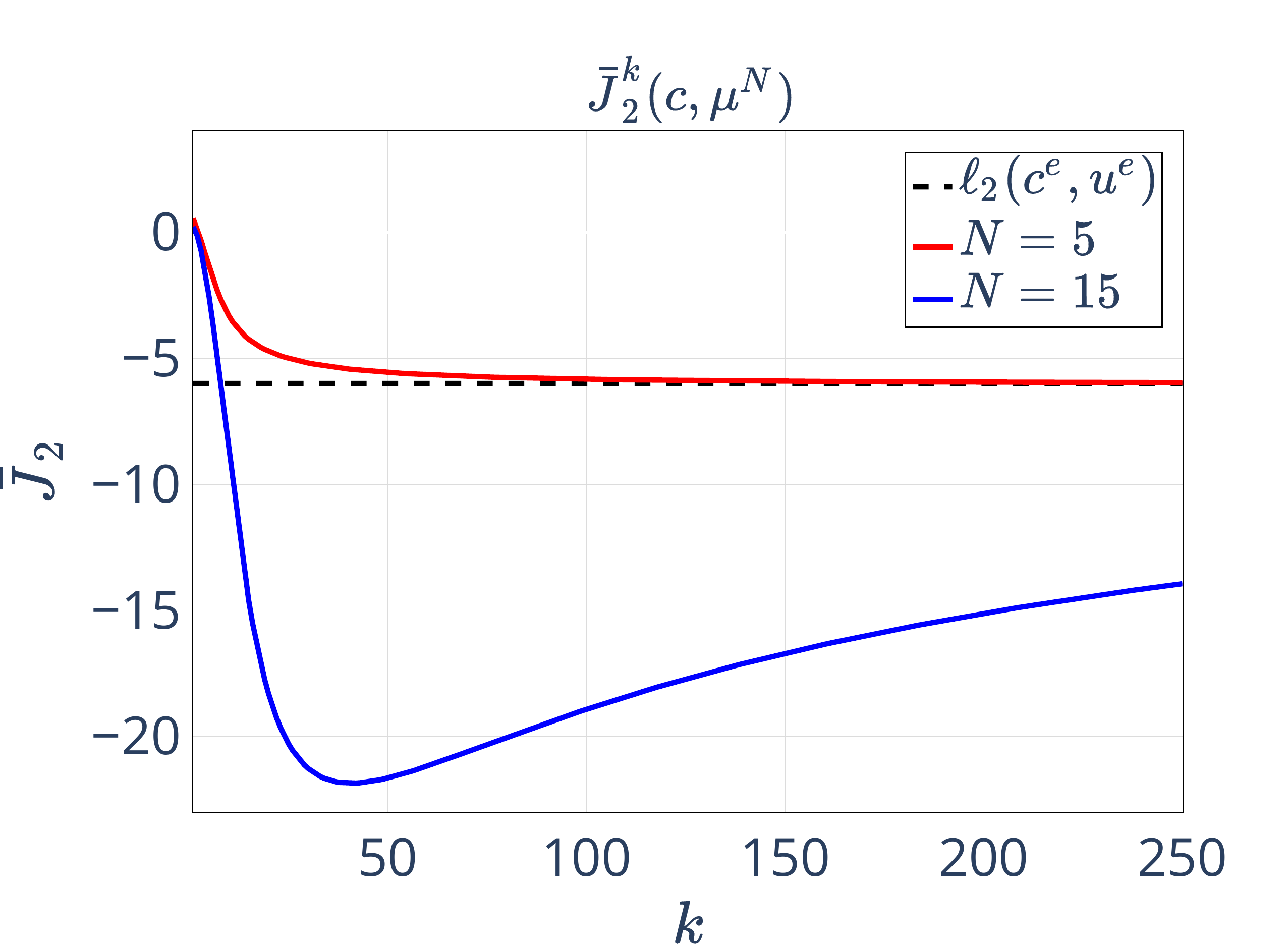}
		\caption{$\bar J_2^k$ for $N=5$ and $N = 15$}\label{fig: avJ2}
	\end{minipage}
\end{center}
\end{figure}

In contrast, the bound of the averaged performance of $J_1$ is independent of $N$, the initial value and the choice of the efficient solutions in each iteration. According to Remark \ref{rem: averJ1} the bound is given by $\ell_1(c^{e},u^{e})=0$. This bound and the averaged costs $\bar J_1^k$ in dependence of the iteration step $k$ are visualized in Figure \ref{fig: avJ1} for MPC-horizons $N=5$ and $N=15$. Additionally, the averaged cost of the cost criterion $J_2$ is visualized in Figure \ref{fig: avJ2} with bound $\ell_2(c^{e},u^{e})=-6$ for MPC-horizon $N=5$ and $N=15$. For both cost criteria the averaged cost $\bar J_i^k$, $i=1,2$, converges for $k\to\infty$ whereas for $N=5$ the convergence is significantly faster. Moreover, we remark that the second averaged cost $\bar J_2^k$ requires considerably more iterations to converge. 
\begin{figure}[h]
\begin{center}
	\begin{minipage}{0.328\textwidth}
		\includegraphics[width =\textwidth]{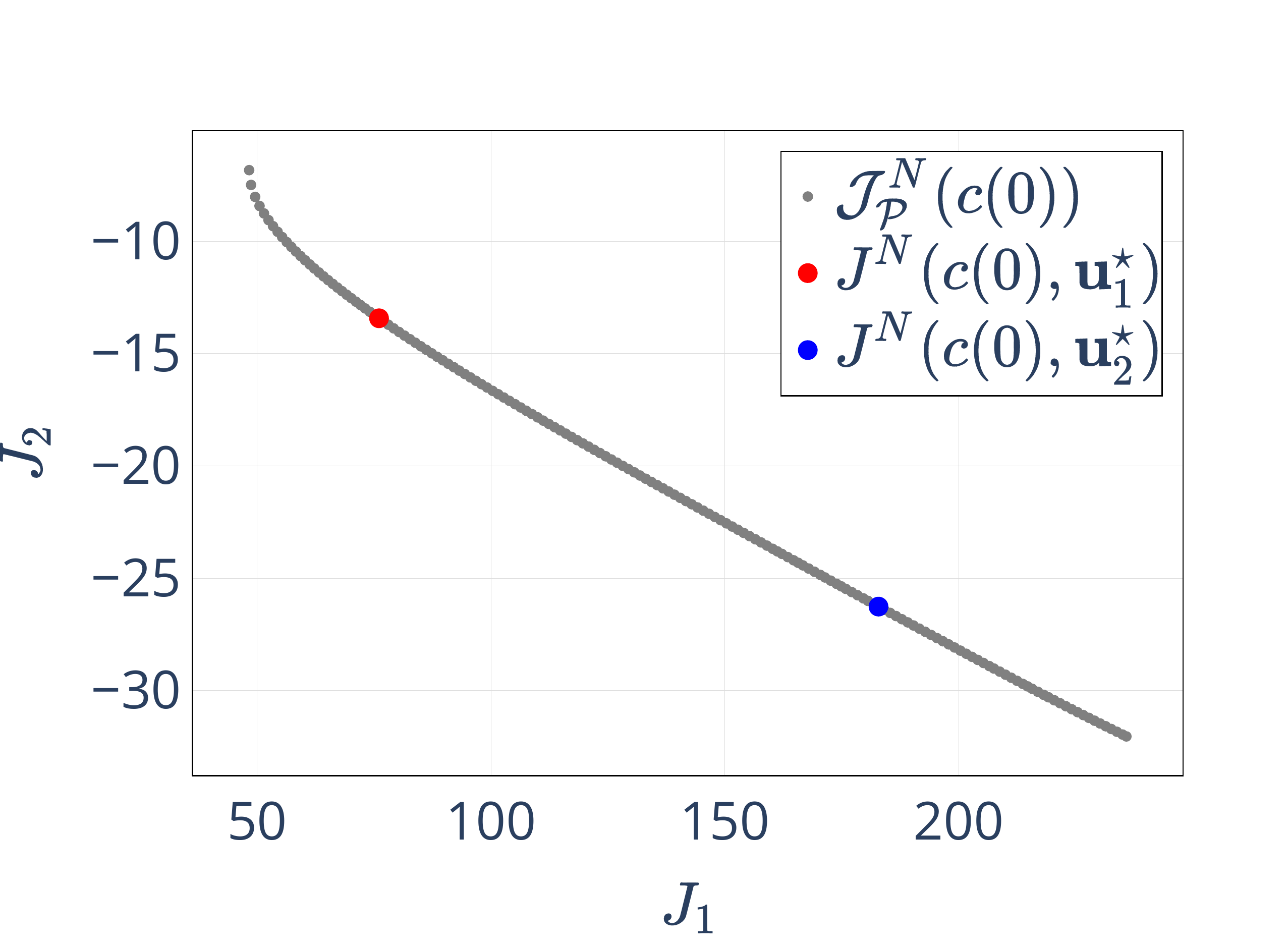}
	\end{minipage}
	\begin{minipage}{0.328\textwidth}
		\includegraphics[width =\textwidth]{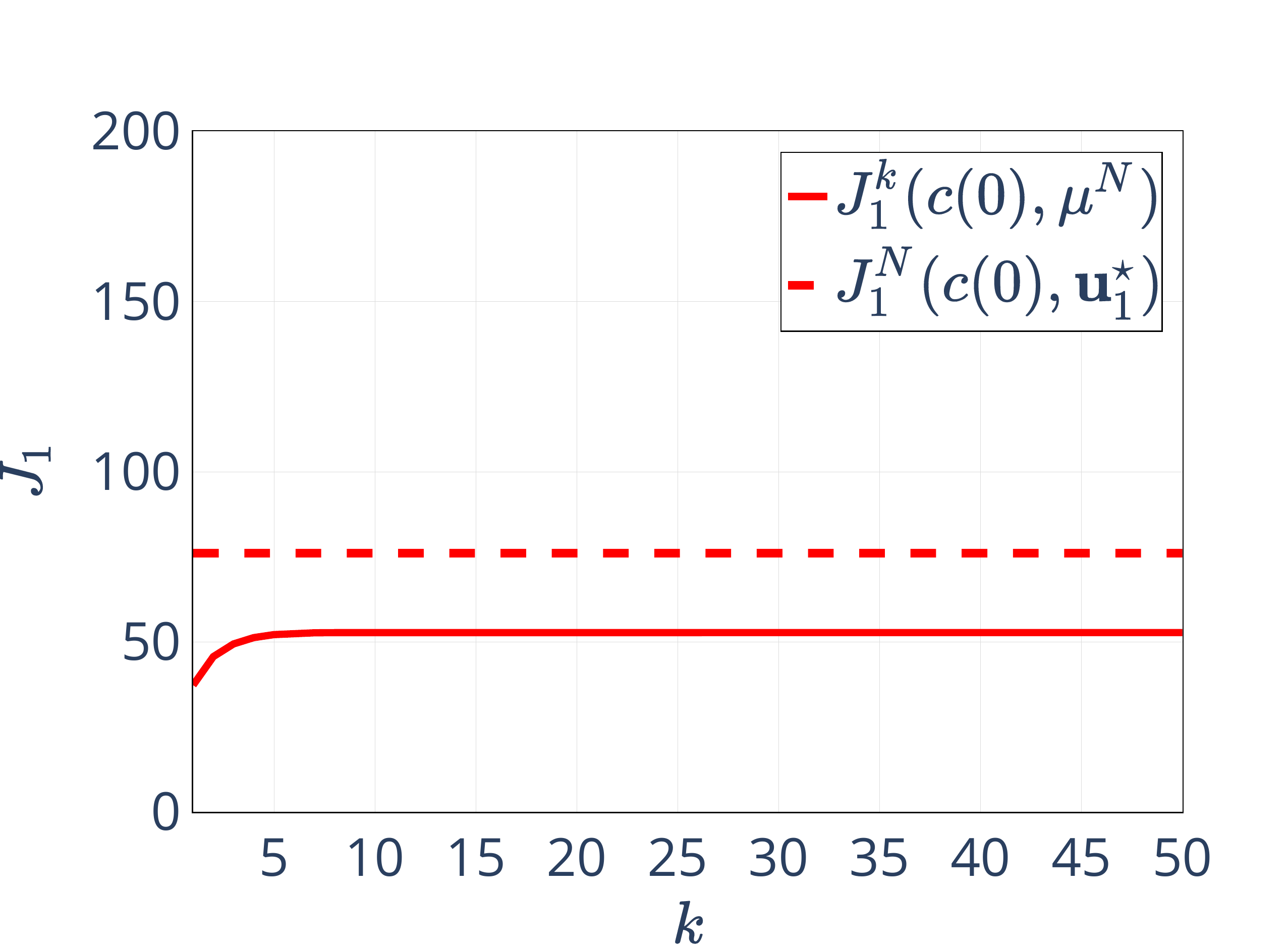}
	\end{minipage}
	\begin{minipage}{0.328\textwidth}
		\includegraphics[width =\textwidth]{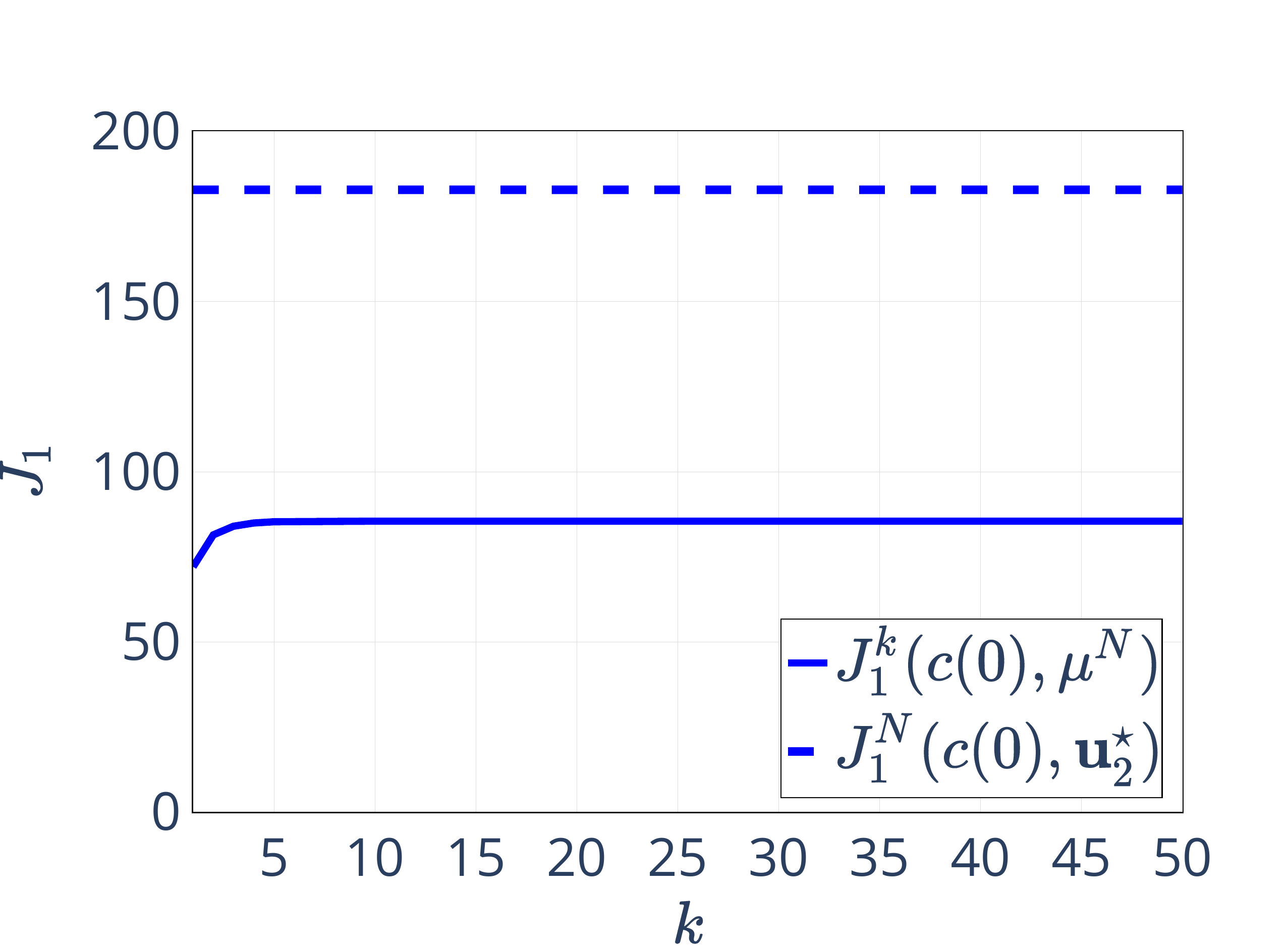}
	\end{minipage}
	\caption{Choice of the efficient solution and the corresponding costs $J_1$}\label{fig: J1pos}
\end{center}
\end{figure}

Since the upper bound on $J_1^\infty(c(0),\mu^N)$ from Theorem \ref{thm: perf, with, 1} depends on the first efficient solution $\ub_{c_0}^{\star}$ in Algorithm \ref{alg: modMOMPC terminal}, we have visualized the performance result for different choices of this efficient solution. In Figure \ref{fig: J1pos} on the left side the first nondominated set $\JN_\Pa(c_0)$ is shown with the different choices of the first efficient solution. The red point corresponds to the efficient solution such that $J^5(c_0,\ub_{c_0}^\star)=(76.064,-13.435)$ and the blue point corresponds to $J^5(c_0,\ub_{c_0}^\star)=(182.852,-26.267)$. Further, the performance of the first cost criterion $J_1$  for $N=5$ is visualized in dependence of the iteration step $k$ and the choice of the first efficient solution (the red line correspond to the red efficient solution and the blue line respectively to the blue one). The dashed lines mark the upper bounds derived by the values of the first objective function for the chosen first efficient solution $J_1^N(c(0),\ub_{c_0}^{\star})$. Hence, we remark that the choice of the first efficient solution has a big impact on the upper bound % $J_1^N(c(0),\ub_{c_0}^{\star})$ 
and %, thus, 
on the performance of $J_1$.

By choosing the efficient solution $\ub_1^\star$ (red point), which has a relatively small value in the first cost functional, we get an upper bound of about $76$. In contrast, the efficient solution $\ub_2^\star$ (blue point) with small value in the second cost delivers an upper bound of approximately $182$. Moreover, we observe that for both efficient solutions the values of the cost functional $J_1$ reach a small neighborhood of their stationary values 53 (red) and 86 (blue), respectively, after less than 10 iteration steps. Additionally, the theoretical upper bound, which depends on the choice of the first efficient solution and is visualized as a dashed line, is adhered as expected. Thus, we can confirm the dependence of the performance of $J_1$ on the choice of the first efficient solution.

The last result shown for Algorithm \ref{alg: modMOMPC terminal} is the asymptotic stability property of the equilibrium $(c^{e},u^{e})$. In contrast to the convergence, the condition for stability depends on the initial value. For this reason, we have to ensure that inequality \eqref{eq: Jbound} is verified for the initial value $c(0)=c_0=(c_A(0),c_B(0))$ and the corresponding first efficient solution $\ub_{c_0}^{\star}$. With a suitable choice of the first efficient solution we can ensure the existence of $\gamma_J\in\K_\infty$ such that inequality \eqref{eq: Jbound} holds, since $\ell_1(c,u)$ is a quadratic function and the system is exponentially controllable to $c^e$.
\begin{figure}[h]
\begin{center}
	\includegraphics[scale=0.3]{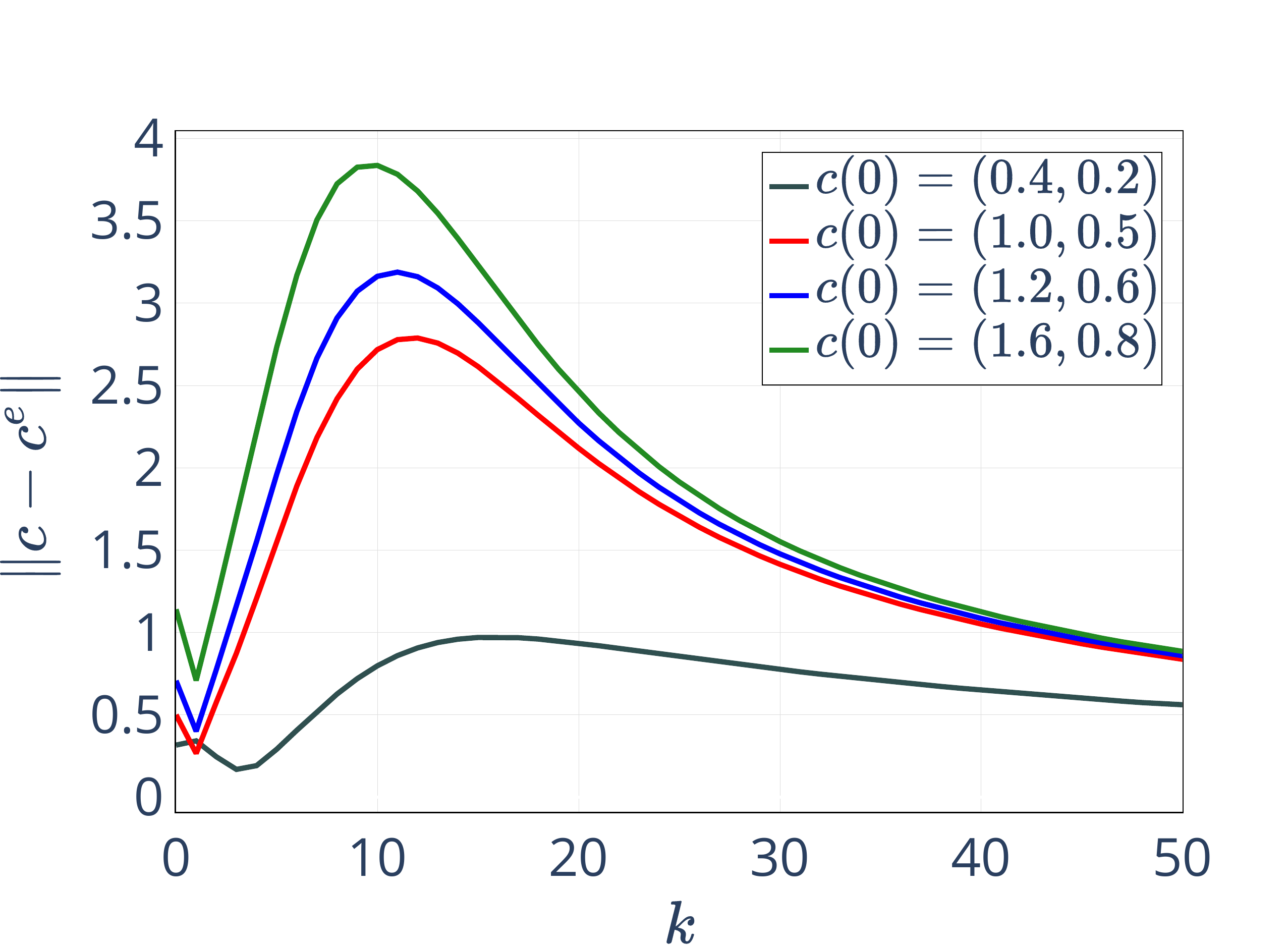}
	\caption{$\norm{c(k)-c^{e}}_2$ for different initial values $c(0)$}\label{fig: trastab}
\end{center}
\end{figure}

In Figure \ref{fig: trastab} the Euclidean norm $\norm{c(k)-c^{e}}_2$ of the closed-loop trajectory is visualized for fixed MPC-horizon $N = 15$ in dependence of the iteration step $k$ and for different initial values $c(0)$. There, we observe that the closer the initial value is to the equilibrium the smaller is the peak of the norm of the trajectory. The numerical results indicate that the stability result from Theorem \ref{thm: as_stab} holds for this example.
\end{bsp}

In the next example we continue with the reactor example, but now we consider Algorithm \ref{alg: modMOMPC2} and check the stronger assumptions that we will need to apply Theorem \ref{thm: perf}.
\begin{bsp}[Reactor Part 2]\label{ex: reactor2}
We consider again the isothermal reactor described in Example \ref{ex: reactor1} with the same constants and constraints. Now, we like to illustrate the results of the performance on the second cost criterion $J_2$. Therefore, we need to consider Algorithm \ref{alg: modMOMPC2} where inequality \eqref{ineq: mpc_constr} holds for all cost criteria. Since we have imposed the special case of an endpoint constraint $\X_0=\{(c^{e},u^{e})\}$ the endpoint is fixed by $c(N) = (c_A(N),c_B(N)) = (c_A^{e},c_B^{e})$ and Assumption \ref{ass: kappabound} is trivially satisfied for $\kappa = u^{e}$. Thus, we can conclude the existence of $\delta \in\LL$ for which the performance estimate on $J_2$ according to Theorem \ref{thm: perf} holds.

For $N=5$, numerical test show that for $\delta(5)=1/5$ the inequality 
\[ J_2^k(c(0),\mu^5) \le J_2^N(c(0),\ub_{c(0)}^{\star}) + (k-5)\ell_2(c^{e},u^{e}) + \frac k 5=:\mathcal{M}(\ub_{c_0}^{\star},5,k)\]
holds for $k\geq 5$ large enough. The second cost $J_2$ and the corresponding bound $\mathcal{M}(\ub_{c_0}^{\star},5,k)$ are visualized in Figure \ref{fig: J2bound}. For other MPC-horizons $N$ and other choices of the first efficient solution it is not that easy to find appropriate values for the $\LL-$function $\delta$. For this reason, we only visualize the bound $\mathcal{M}$ for this special case in Figure \ref{fig: J2bound}.
\begin{figure}[h]
\begin{center}
	\begin{minipage}{0.48\textwidth}
		\includegraphics[width =\textwidth]{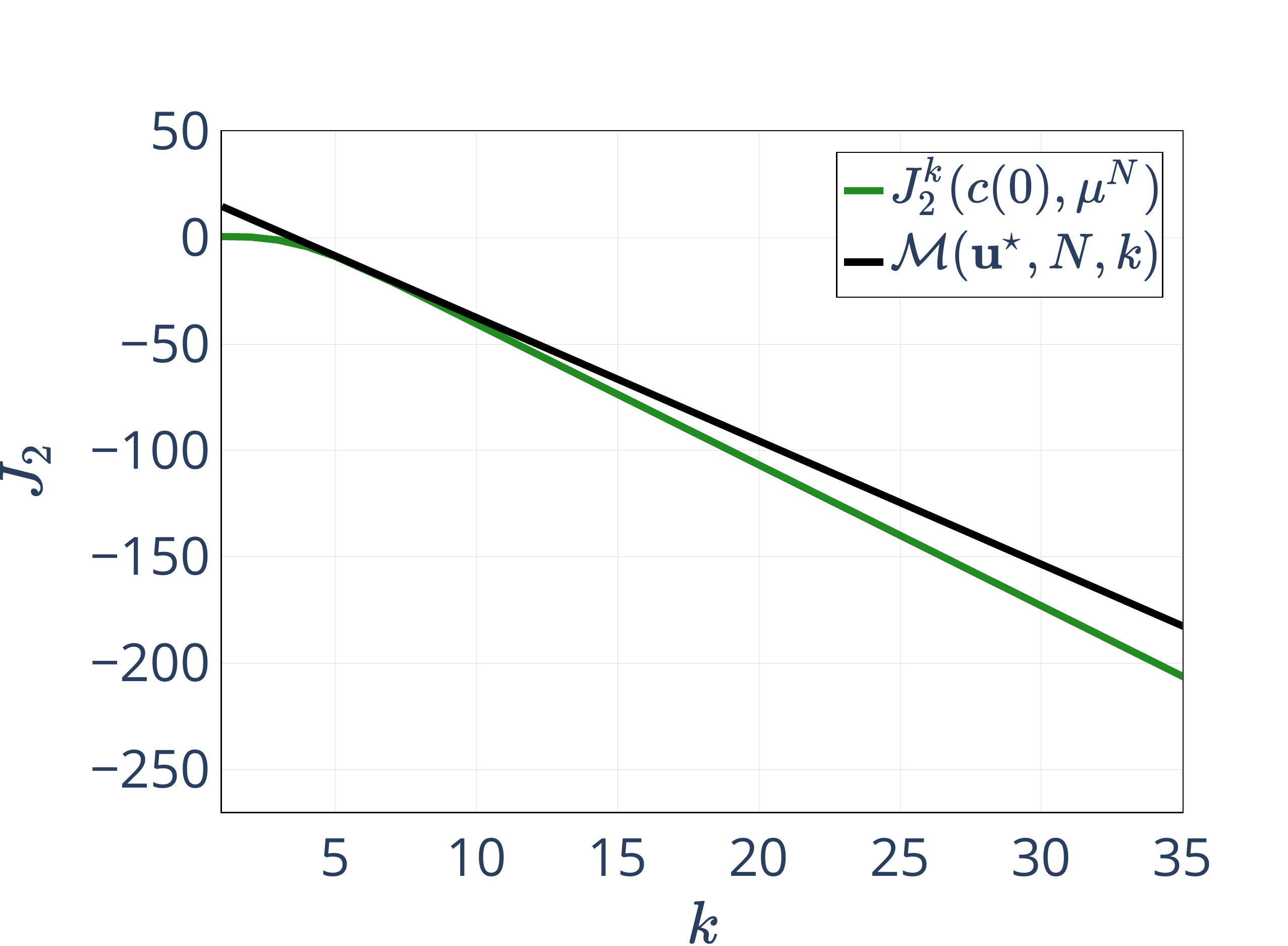}
		\caption{$J_2^k(c(0),\mu^N)$ with corresponding bound $\mathcal{M}(\ub_{c_0}^{\star},5,k)$}\label{fig: J2bound}
	\end{minipage}
	\quad
	\begin{minipage}{0.48\textwidth}
		\includegraphics[width =\textwidth]{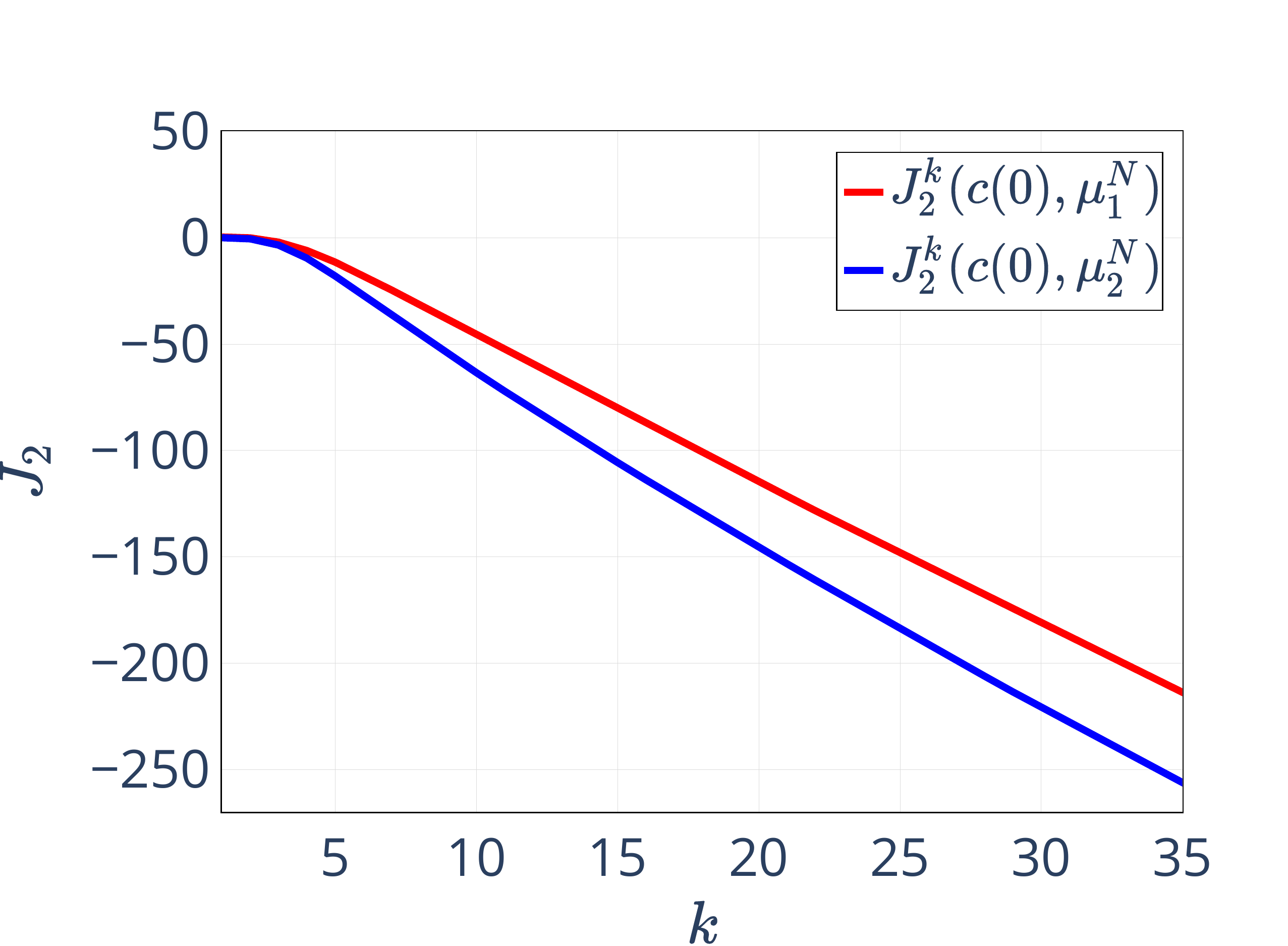}
		\caption{$J_2^k(c(0),\mu^N)$ for different efficient solutions}\label{fig: J2N5}
	\end{minipage}
\end{center}
\end{figure}

In Figure \ref{fig: J2N5} the performance of the cost criterion $J_2$ is visualized for MPC-horizon $N= 5$ and for different choices of the first efficient solution $\ub_{c(0)}^{\star,N}$. The first efficient solutions are chosen as in Example \ref{ex: reactor1} in Figure \ref{fig: J1pos} on the left side. Note that the first nondominated set $\JN_\Pa(c(0))$ is identical for both algorithms. Thus, the efficient solution (and the colors) are the same as in the previous simulations. Again, we remark that the choice of the first efficient solution has an impact on the performance of the second cost criterion $J_2$.
\end{bsp}

After verifying the theoretical results from Section \ref{sec: performance} and \ref{sec: stability} by means of numerical simulations for the isothermal reactor with two cost functions, we now like to illustrate that---as the theoretical results suggest---our approach also works for more than two cost criteria. %In particular, we would like to emphasise that our theoretical results also apply to a general number of objectives. 
To this end, we add another cost criterion to the multiobjective optimal control problem  \eqref{reactor} and present the numerical results in the same manner as in Example \ref{ex: reactor1} and \ref{ex: reactor2}.
\begin{bsp}[Reactor with three objectives]
We consider the isothermal reactor from Example \ref{ex: reactor1} and the corresponding multiobjective optimal control problem   \eqref{reactor}. In order to extend the example we add a third cost function given by
\[J_3^N(c_0,\ub):=\sum_{k=0}^{N-1}\ell_3(c(k,c_0),u(k)),\; \text{with }\ell_3(c,u)=u^2,\]
i.e., we now aim to minimize $J^N(c_0,\ub)=(J_1^N(c_0,\ub),J_2^N(c_0,\ub),J_3^N(c_0,\ub))$. 	Stage cost $\ell_3$ is a continuous function and, thus, Assumption \ref{ass: modTerminal} is satisfied. %Since the additional objective does not change the quantities from problem \ref{reactor} 
Assumptions \ref{ass: cont1} and \ref{ass: kappabound} can be shown exactly as in Example \ref{ex: reactor1} and \ref{ex: reactor2}.  For our numerical simulations we consider the MPC-horizon $N=15$. As in Example \ref{ex: reactor2}, we apply Algorithm \ref{alg: modMOMPC2} to illustrate the trajectory convergence, the averaged performance of all cost criteria $J_i$, $i=1,2,3$, and, especially, the non-averaged performance of $J_i$ for all $i\in\{1,2,3\}$. Further, we chose the first efficient solution such that $J^{15}(c_0,\ub_{c_0}^\star)=(317.827,-380.092,1969.311)$.
\begin{figure}[h]
\begin{center}
	\includegraphics[scale=0.3]{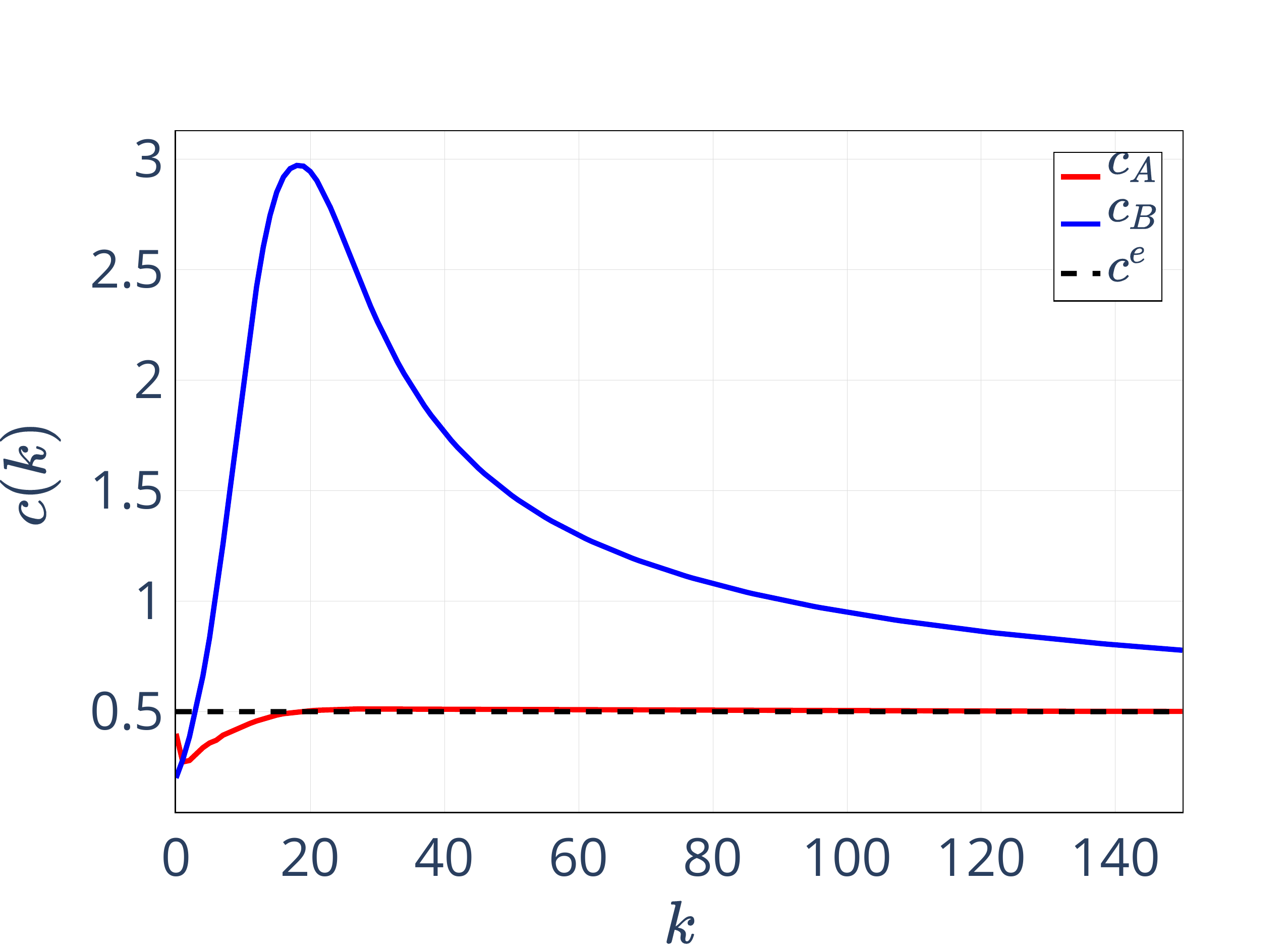}
	\caption{Closed-loop trajectory}\label{fig: traj_3}
\end{center}
\end{figure}

In Figure \ref{fig: traj_3} we observe that the closed-loop trajectory behaves qualitatively as in Figure \ref{fig: tra} but quantitatively there are differences. Especially, the peak of the second component $c_B$ is higher than in Example \ref{ex: reactor1}. Further, the averaged cost $\bar J_1$ has a smaller start value and the amplitude of the second averaged cost $\bar J_2$ is larger than in the previous example. These phenomenons are visualized in Figure \ref{fig: avCost3}. Especially, on the right side in Figure \ref{fig: avCost3} we remark that the third averaged cost $\bar J_3$ also converges from below to the theoretical bound $\ell_3\equ=144$ as stated in Theorem \ref{thm: avPerfJi}.
\begin{figure}[h]
\begin{center}
	\begin{minipage}{0.328\textwidth}
		\centering
		\includegraphics[width =\textwidth]{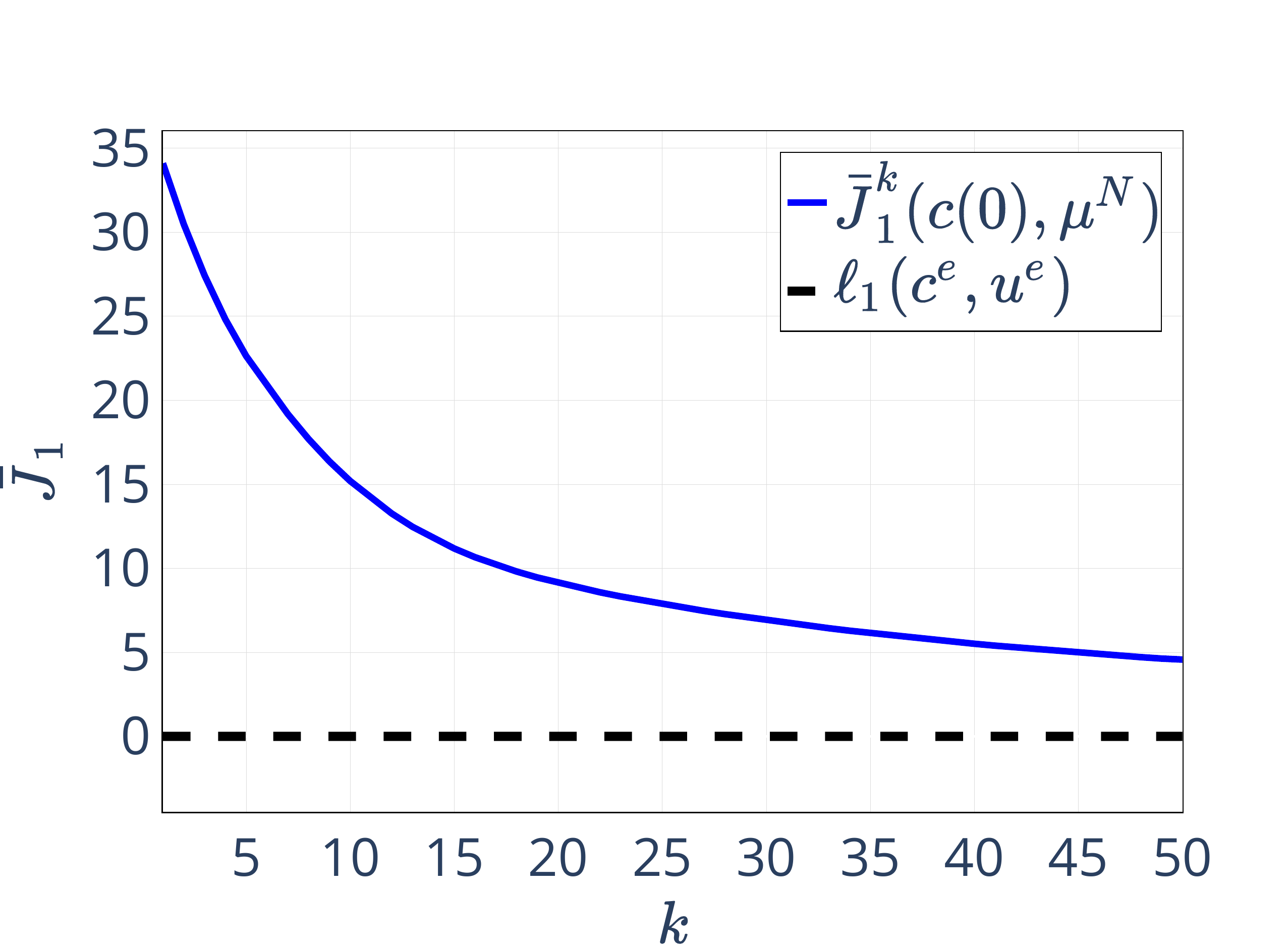}\\
		(a) $\bar J_1^k(c_0,\mu^N)$
	\end{minipage}
	\begin{minipage}{0.328\textwidth}
		\centering
		\includegraphics[width =\textwidth]{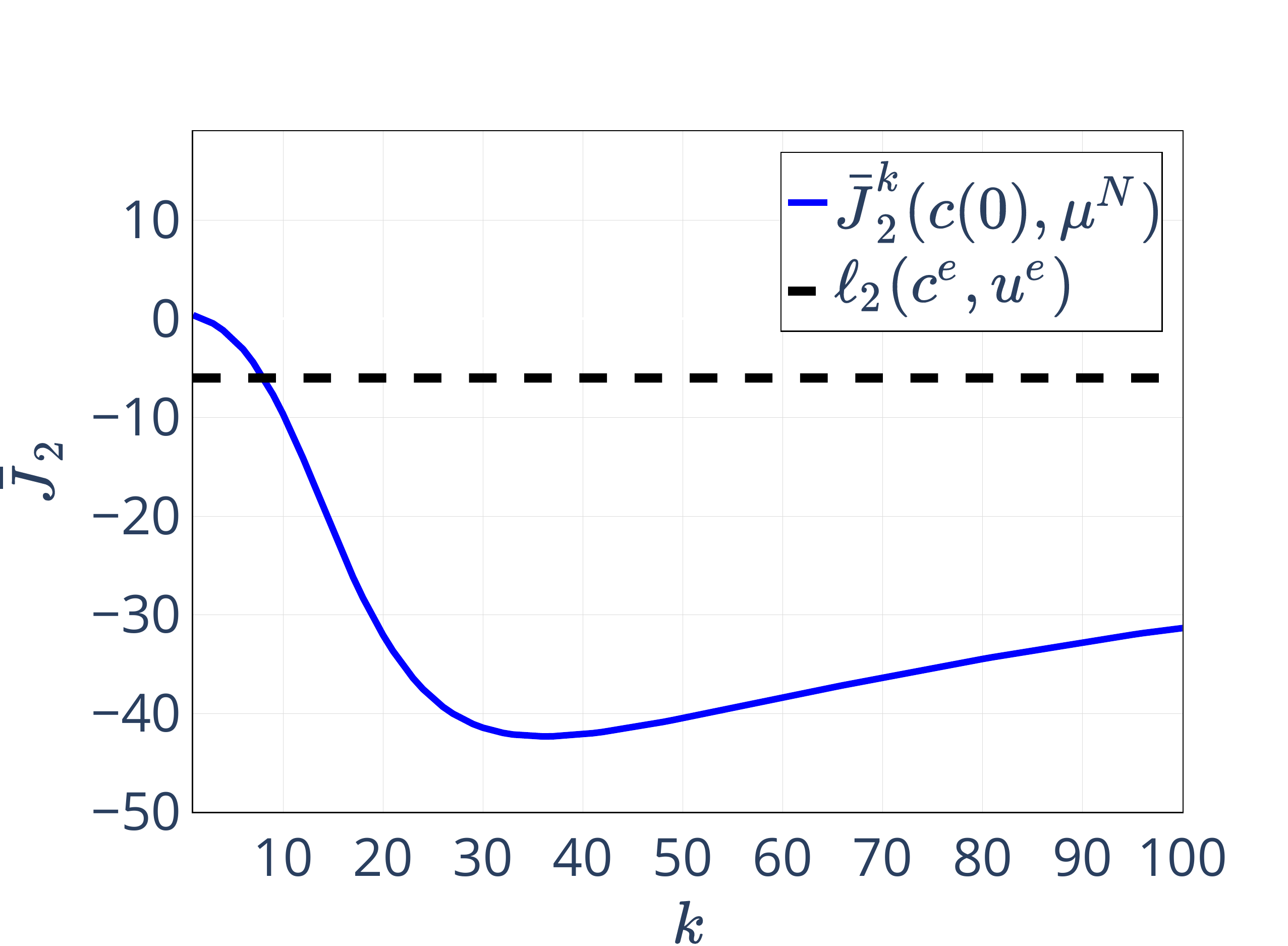}\\
		(b) $\bar J_2^k(c_0,\mu^N)$
	\end{minipage}
	\begin{minipage}{0.328\textwidth}
		\centering
		\includegraphics[width =\textwidth]{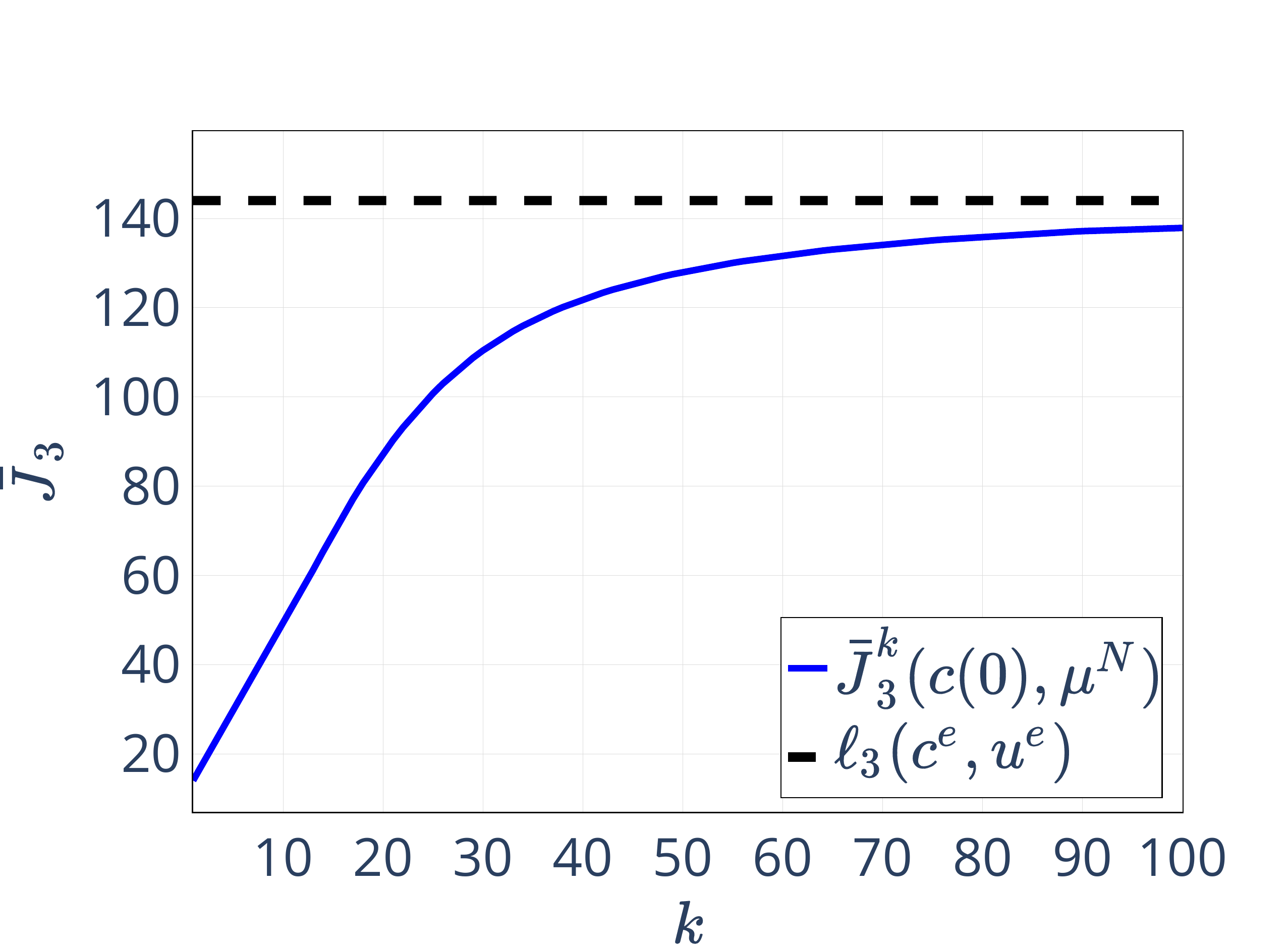}\\
		(c) $\bar J_3^k(c_0,\mu^N)$
	\end{minipage}
	\caption{Averaged Performance of all cost criteria $J_i$}\label{fig: avCost3}
\end{center}
\end{figure}
\begin{figure}[h]
\begin{center}
	\begin{minipage}{0.328\textwidth}
		\centering
		\includegraphics[width =\textwidth]{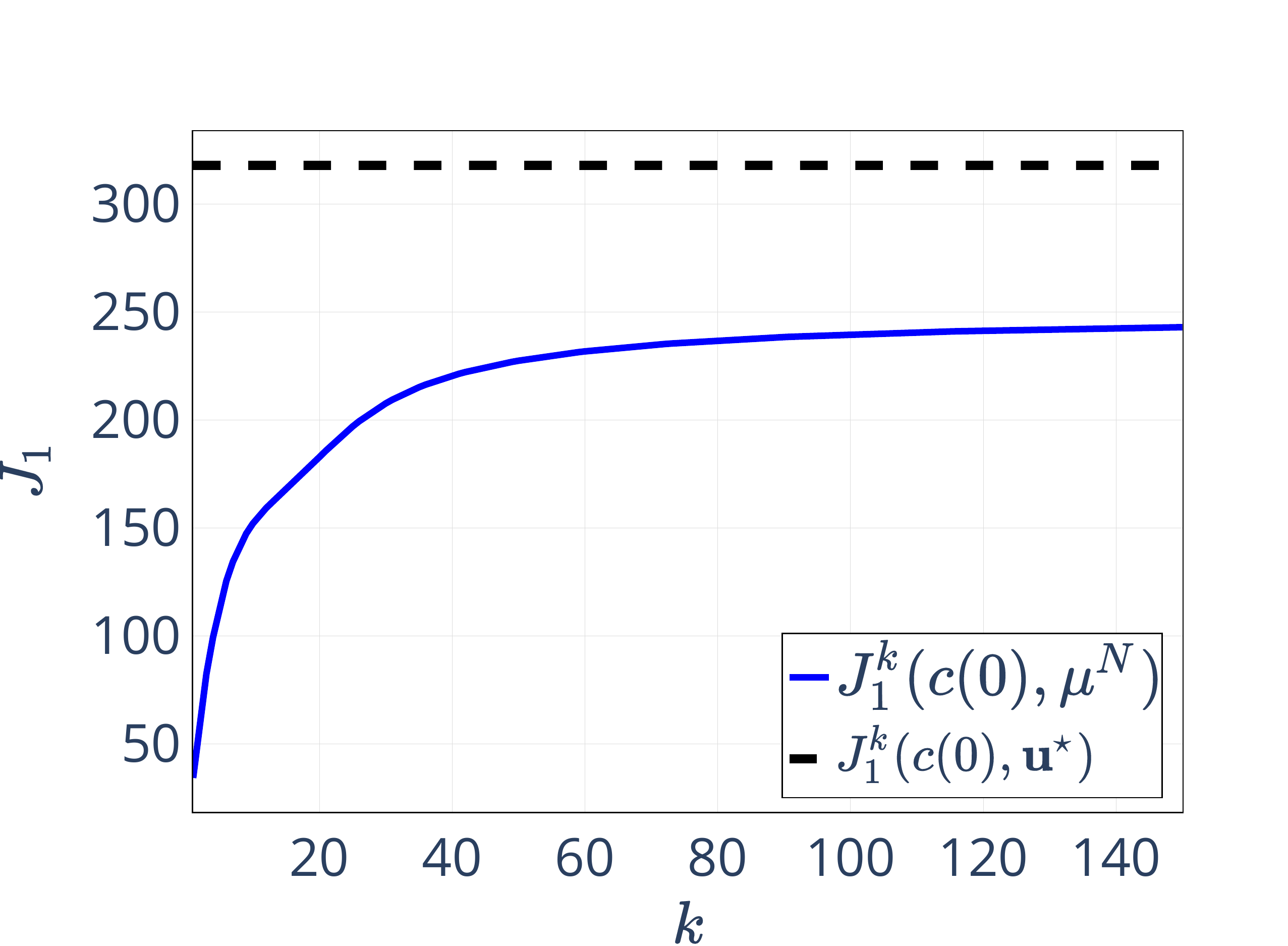}\\
		(a) $J_1^k(c_0,\mu^N)$
	\end{minipage}
	\begin{minipage}{0.328\textwidth}
		\centering
		\includegraphics[width =\textwidth]{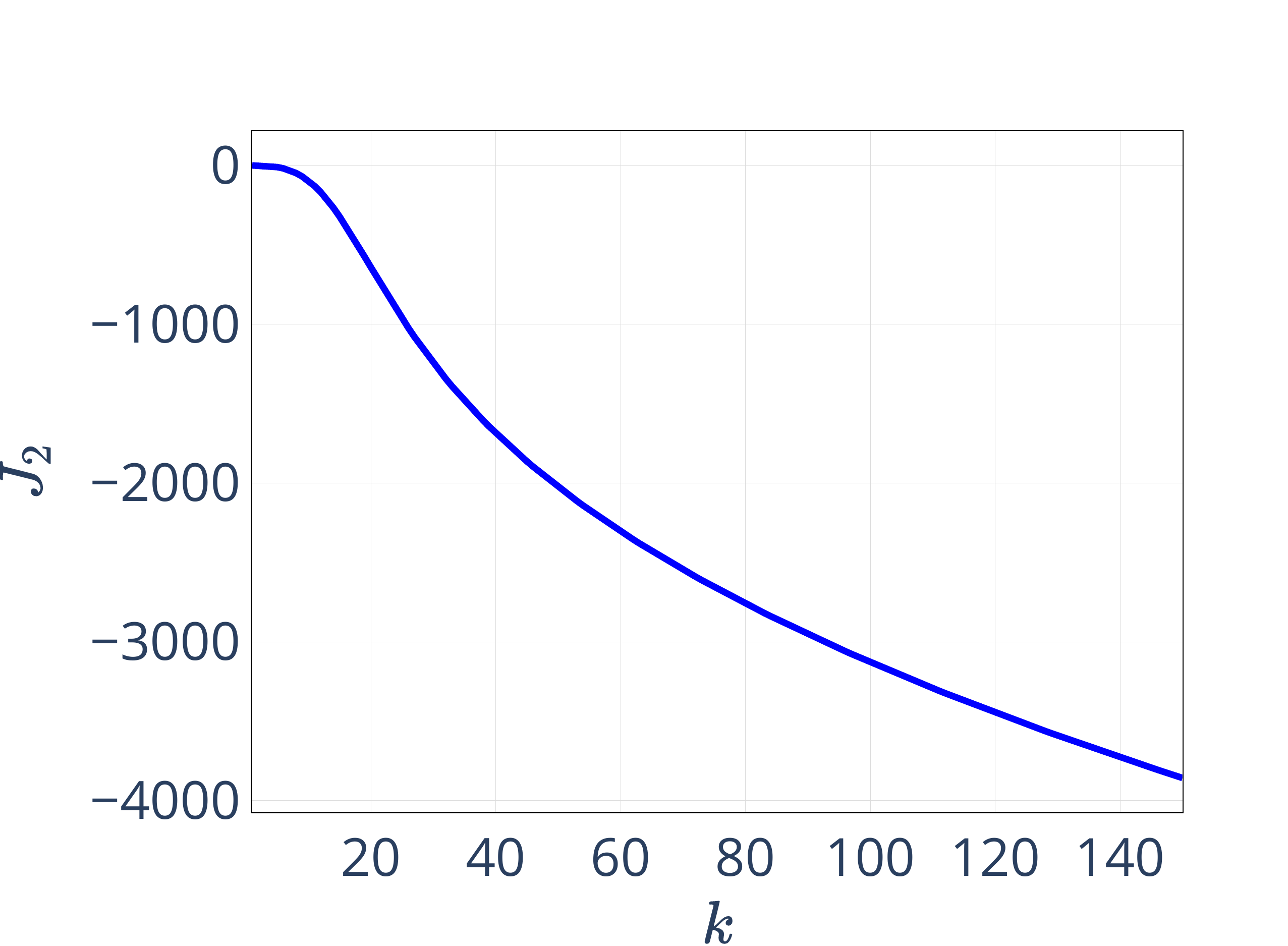}\\
		(b) $J_2^k(c_0,\mu^N)$
	\end{minipage}
	\begin{minipage}{0.328\textwidth}
		\centering
		\includegraphics[width =\textwidth]{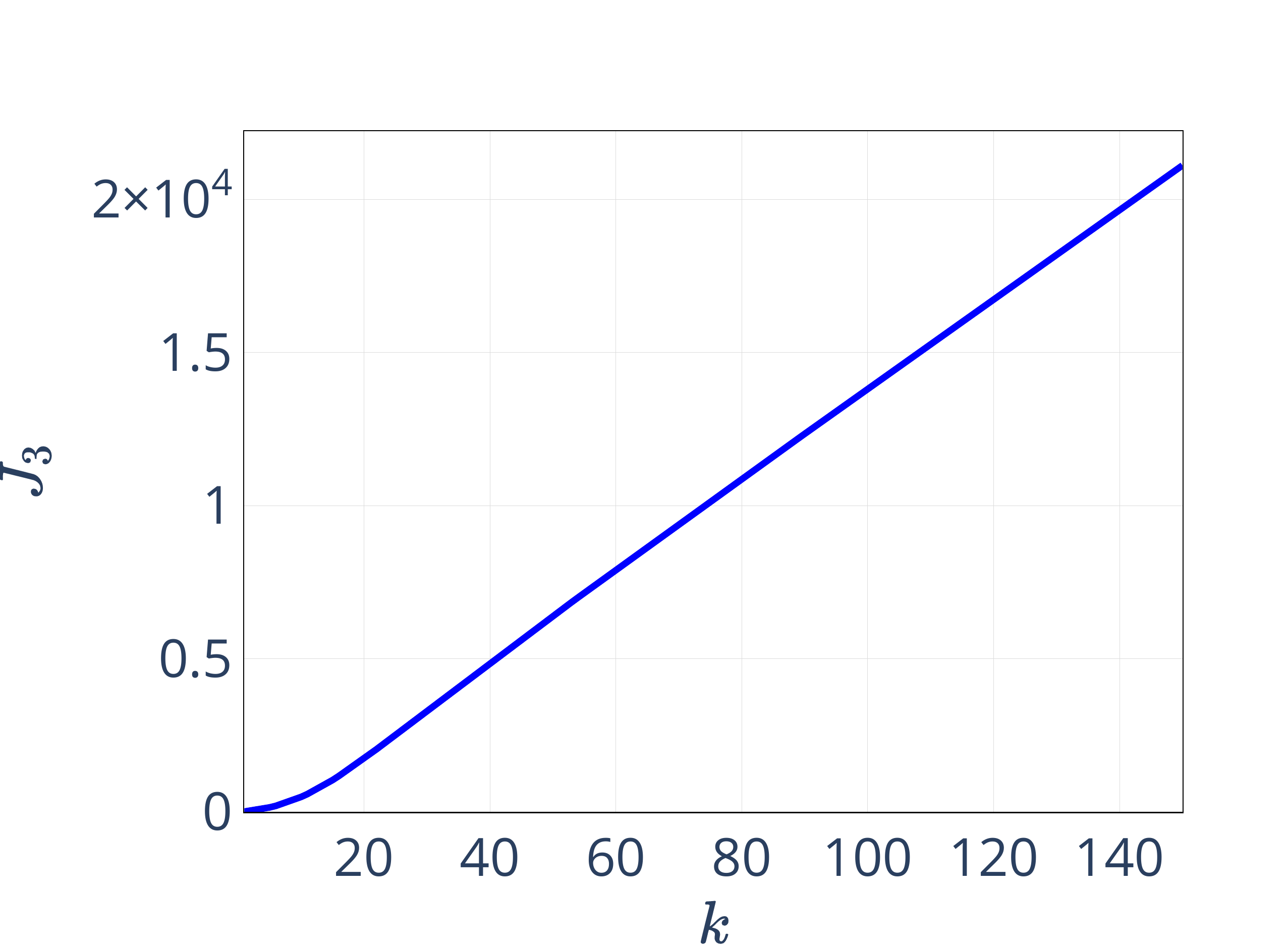}\\
		(c) $J_3^k(c_0,\mu^N)$
	\end{minipage}
	\caption{Performance of all cost criteria $J_i$}\label{fig: Cost3}
\end{center}
\end{figure}

Further, in Figure \ref{fig: Cost3} the performances of $J_i$, $i=1,2,3$, are shown. Again, the first cost function $J_1$ complies with the theoretical bound $J_1^N(c_0,\ub^{\star}_{c_0})$. For the second and third cost criterion we can observe that the performance behaves as expected. The third cost $J_3$ is strictly increasing since in each iteration the squared value of the cost in the equilibrium $u^{e}=12$ is added.

\end{bsp}

In the examples in this section we have illustrated our theoretical results numerically, particularly regarding the impact of the choice of the \emph{first} efficient solution $\posNxn$, chosen in step (0) of our algorithms. In the next section, we will analyse numerically whether the choice of the \emph{subsequent} efficient solutions $\posNxk$ in step~(1) of the algorithms, has an impact on the solution behavior and on the performance.

\section{Selection rules for subsequent efficient solutions}\label{sec:numerics2}

%In the previous section we have shown \LK{theoretically and visualized numerically the convergence of the closed-loop trajectory and the asymptotic stability of the (optimal) equilibrium $\eq$ for Algorithm \ref{alg: modMOMPC terminal}. Further, we introduced averaged and non-averaged performance results for $J_1$, and for all cost criteria $J_i$, $i=2,\dots,s$, for the second imposed Algorithm \ref{alg: modMOMPC2}. }The non-averaged performance estimates has in common that they depend again on the first POS $\posN$. We remark that these POSs are -- so far -- chosen completely freely, as long as the necessary conditions, as, e.g. inequality \eqref{eq: Jbound}, are fulfilled. Moreover, the averaged performance estimates depend on the feedback $\mu^N(x(k))=u^{*,N}_{x(k)}(0)$ and therefore also on the choice of the \LK{subsequent} POSs $\posNxk$ in every iteration. These choices depend on the previous choices because of inequality \eqref{ineq: mpcend1} in Algorithm \ref{alg: modMOMPC terminal} and inequality \eqref{ineq: mpc_constr} in Algorithm \ref{alg: modMOMPC2} respectively. But still the choice of the \LK{subsequent} POS $\posNxk$ remains a degree of freedom.
In this section we will investigate numerically the impact of different selection rules for the subsequent efficient solutions, i.e., those in step (1) of the algorithms. We will do this for Algorithm \ref{alg: modMOMPC2}, since we want to guarantee the theoretical performance estimates for all cost criteria. Closely related is the development of the nondominated sets during the iterations. In all simulations, we consider the same first efficient solution and, thus, determine the performance of the MO MPC algorithm only depending on the choice of the subsequent efficient solutions. In order to describe the expected effects, we first recall the resulting bounds on the nondominated set of Algorithm \ref{alg: modMOMPC2}. For this reason, we visualized the nondominated set of the isothermal reactor with two objectives, see Example \ref{ex: reactor2}, in the second iteration in Figure \ref{fig: reactorBound}.
The approximation of this nondominated set was calculated with  ASMO \cite{GithubASMO}.
We observe that the resulting nondominated set from which we choose the efficient solution is relatively small and excludes the extremal ends of the nondominated set.
%Further, we recall that the choice of a efficient solution can be described via a weighting of the cost criteria, as in the weighted sum approach for multiobjective optimization, see \cite{Eichfelder2008, Eichfelder2009, Miettinen1998, Ehrgott2005}. To this end, we consider the weighted sum 
%\[S(\textbf{w},J^N)= \sum_{i=1}^{s}w_i\Jni(x_0,\ub),\]
%where $\textbf{w}=(w_1,\dots,w_s)$ are the weights with $w_i\in[0,1]$ and $\sum_{i=1}^{s}w_i=1$. In each iteration $k =1, \dots, K$ the weights $\textbf{w}$ can either be newly selected or they remain fixed according to a certain rule, e.g. $w_1 =  1$ and $w_i = 0$, $i=2,\dots,s$. \LG{Minimising the weighted sum always yields efficient points, except possibly when some of the weights vanish and the corresponding solution lies at the boundary of the nondominated set. Since, however, our constraints cut off the boundary of the nondominated set, as shown in Figure \ref{fig: reactorBound}, we do not need to exclude the case of vanishing weigths and may also allow for $w_i\in\{0,1\}$.} 
\begin{figure}[h]
\begin{center}
\includegraphics[scale = 0.3]{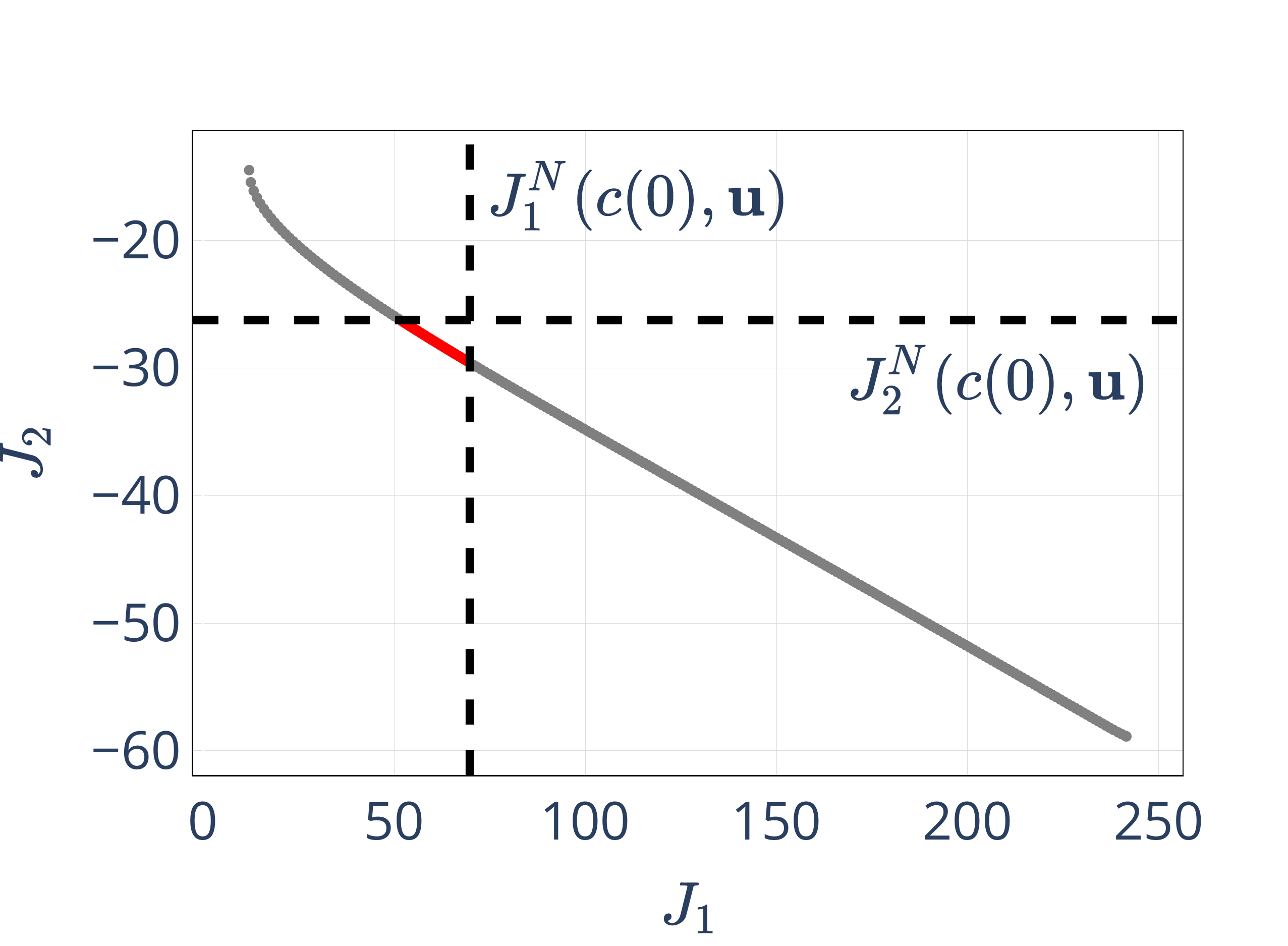}
\caption{Visualization of the resulting nondominated set}\label{fig: reactorBound}
\end{center}
\end{figure}

Nevertheless, there is still a degree of freedom in choosing the subsequent efficient solution. Generally, an efficient solution with values at the top left of the considered nondominated set will cause the first cost criterion to become smaller, and vice versa. Particularly, in our setting, where the first cost is always the stabilizing one, putting a large emphasis on the first cost criterion forces the trajectory to converge faster, since this causes a lower cost. In order to check whether this effect can be seen in practice, we will investigate different selection rules for choosing the subsequent efficient solutions. We examine the influence of the choice of the subsequent efficient solutions in the algorithms on the solution and performance behavior. To this end, we introduce the following selection rules for the subsequent efficient solutions:
\begin{itemize}
\item "ideal": the efficient solutions are computed as in \eqref{eq:defideal} %the pre-images of those nondominated  points in the nondominated set $\JN_\Pa(x(k,x_0))$ restricted by \eqref{ineq: mpc_constr} 
as pre-image of  nondominated points   with  minimal Euclidean distance to the ideal point $z^\star$. % introduced in Example \ref{ex: reactor1}.
%Thereby, we work with an approximation of the nondominated set obtained with  ASMO \cite{GithubASMO}. % to the \eqref{MO MPC terminal} in step (1) of Algorithm \ref{alg: modMOMPC2}  equipped with the additional bounds from \eqref{ineq: mpc_constr}.
This selection rule is illustrated in Figure \ref{fig: ideal}. %	t $\norm{\ub-\ub_{\text{ref}}}^2_2$ is minimal, where \LG{$u_{\text{ref}_i} =\argmin_{\ub\in\U^N(x_0)} \Jni(x_0,\ub)$.}
\item "min 1": the efficient solutions are chosen such that $J_1^N$ (with the additional bounds from \eqref{ineq: mpc_constr}) is minimal.
\item "min 2": the efficient solutions are chosen such that $J_2^N$  (with the additional bounds from \eqref{ineq: mpc_constr})  is minimal.
\end{itemize} 
\begin{figure}[h]
\begin{center}
\includegraphics[scale =0.9]{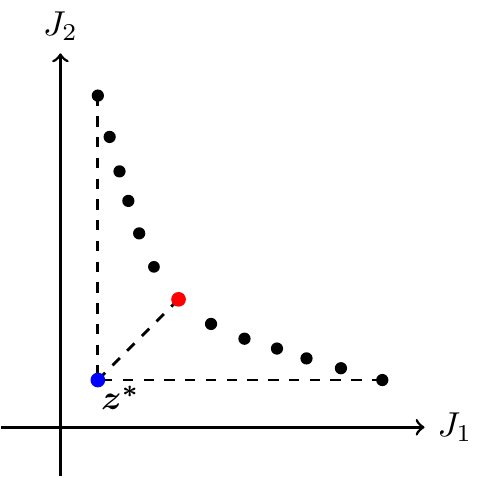}
\caption{Visualization of the selection rule "ideal"}\label{fig: ideal}
\end{center}
\end{figure}
We note that the last two selection criteria only guarantee to find  weakly efficient solutions of the underlying multiobjective optimization problem which are located at the extremal ends of the nondominated set. The set of weakly efficient solutions forms a superset of the set of efficient solutions and contains also those feasible solutions for which just no other feasible solution exists which strictly improves all objective functions at the same time. 
However, in our case the solutions are additionally constrained by the bounds \eqref{ineq: mpc_constr} in Algorithm \ref{alg: modMOMPC2}, which ``cut off'' these extremal points, cf.\ Figure \ref{fig: reactorBound}. Since according to our numerical experience this is the situation in all our numerical tests, we can rule out that our algorithm yields weakly efficient solutions which are not also efficient.

%Since we have already shown that the reactor fulfils the necessary assumptions we proceed with this example for illustrating these different selection rules.

\begin{bsp}[Reactor Part 3]\label{ex: reactor3}
We consider again the isothermal reactor from Example \ref{ex: reactor1} with the same constants and constraints. For the simulation we used Algorithm \ref{alg: modMOMPC2} with the different selection rules described above for choosing the subsequent efficient solution $\ub^{\star}_{c(k)}$. In all simulations we consider the MPC-horizon $N= 5$ and use the same first efficient solution $\ub_{c(0)}^{\star}$ which is chosen as in Example \ref{ex: reactor1}. %such that $J^5(c_0,\ub_{c_0}^\star)=(54.034,9.500)$.
In Figure \ref{fig: pareto} the first and the second nondominated set are shown. 
\begin{figure}[h]
\begin{center}
	\begin{minipage}{0.49\textwidth}
		\includegraphics[width =\textwidth]{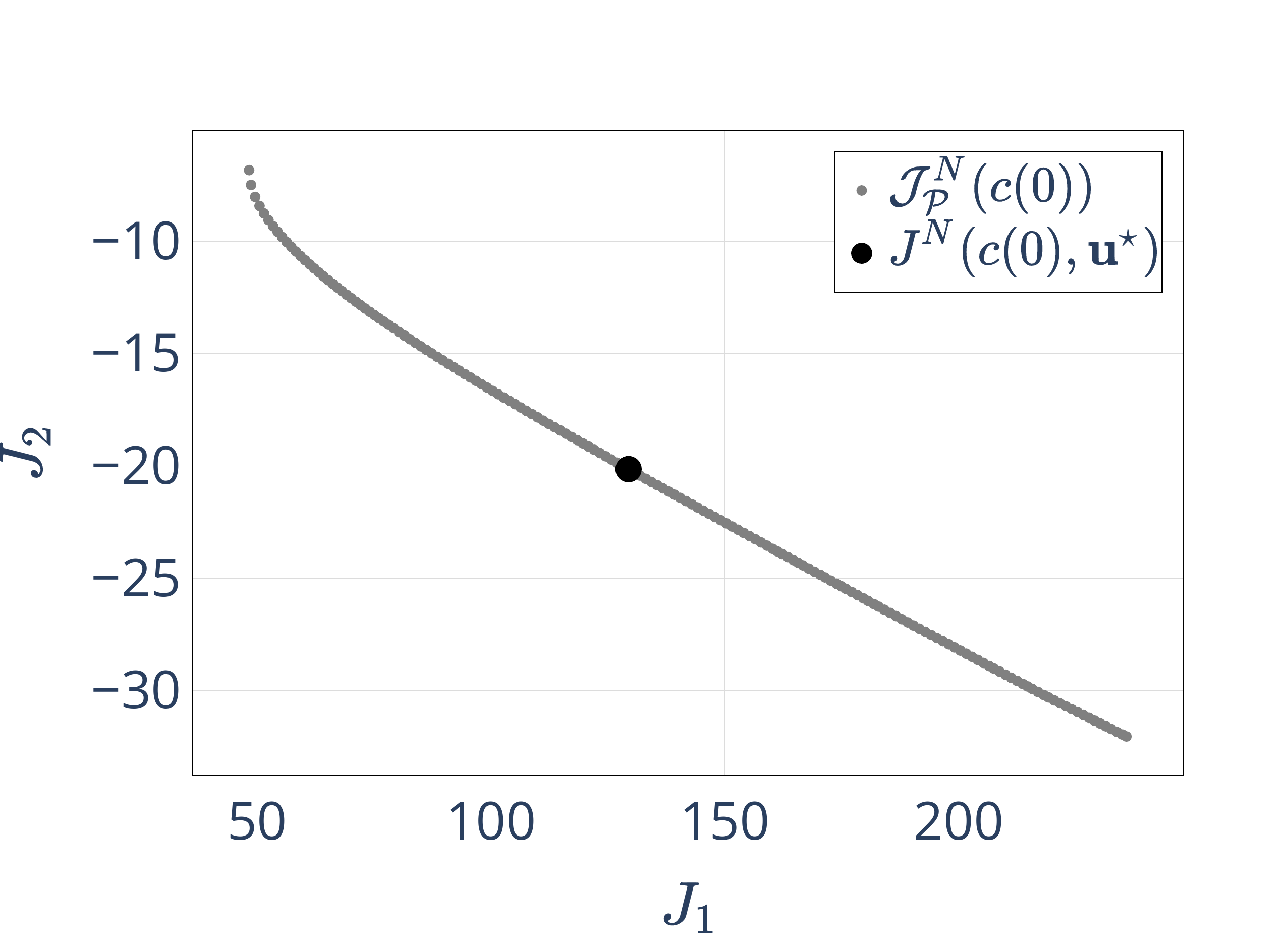}
	\end{minipage}
	\begin{minipage}{0.49\textwidth}
		\includegraphics[width =\textwidth]{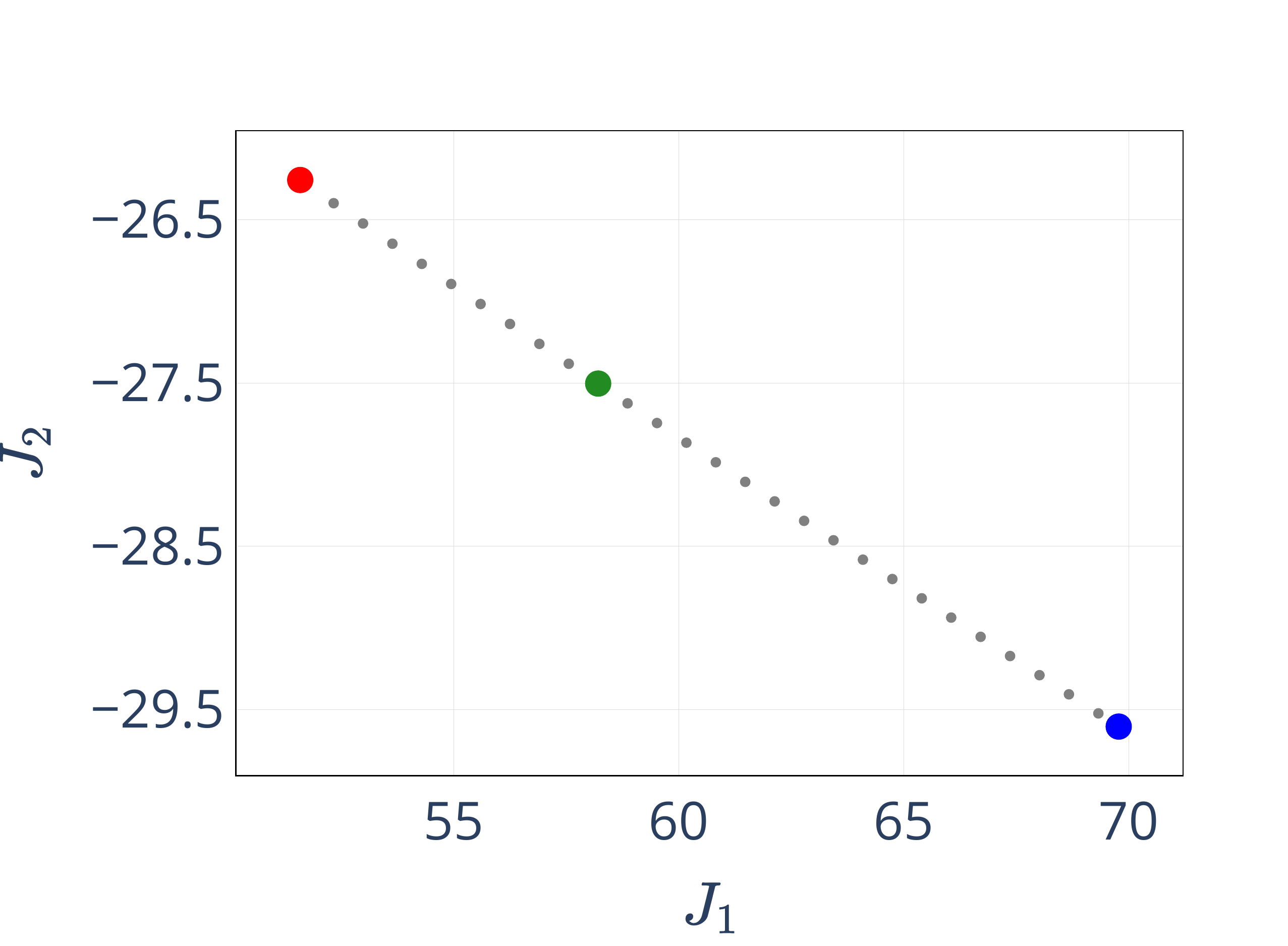}
	\end{minipage}
	\caption{First and second nondominated set with the chosen efficient solutions}\label{fig: pareto}
\end{center}
\end{figure}

The colored points are the efficient solutions chosen according to the respective selection rules. While in all simulations the same first efficient solution is chosen, we compare different selection rules for the subsequent efficient solutions, which are, however, fixed during the iterations. The different selection rules are visualized in Figure \ref{fig: pareto} on the right-hand side. Further, in Figure \ref{fig: parheur} we visualized the nondominated set and the corresponding chosen subsequent efficient solutions in iteration step $k=6$.
\begin{figure}[h]
\begin{center}
	\begin{minipage}{0.328\textwidth}
		\centering
		\includegraphics[width =\textwidth]{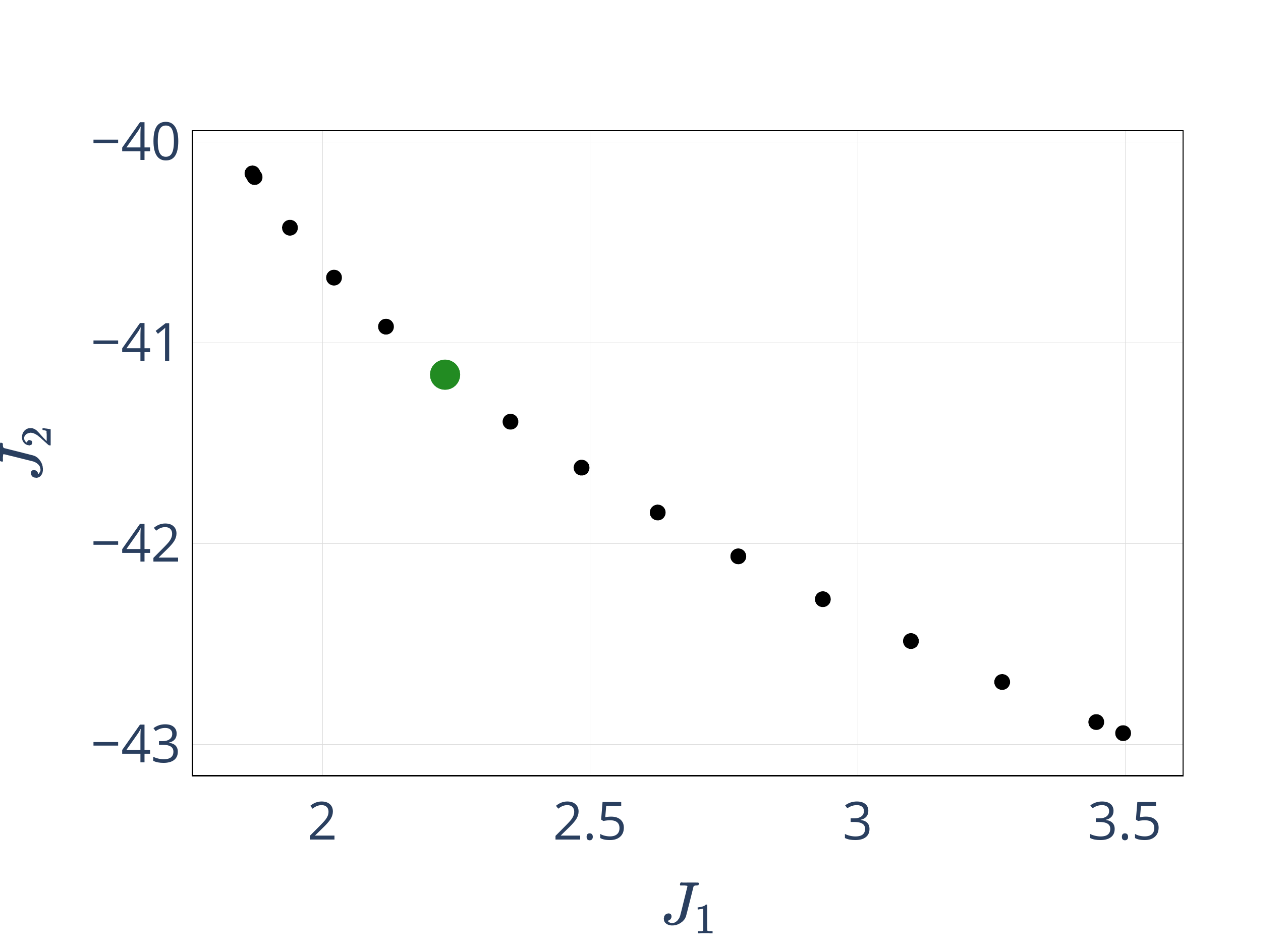}\\
		(a)	"ideal"
	\end{minipage}
	\begin{minipage}{0.328\textwidth}
		\centering
		\includegraphics[width =\textwidth]{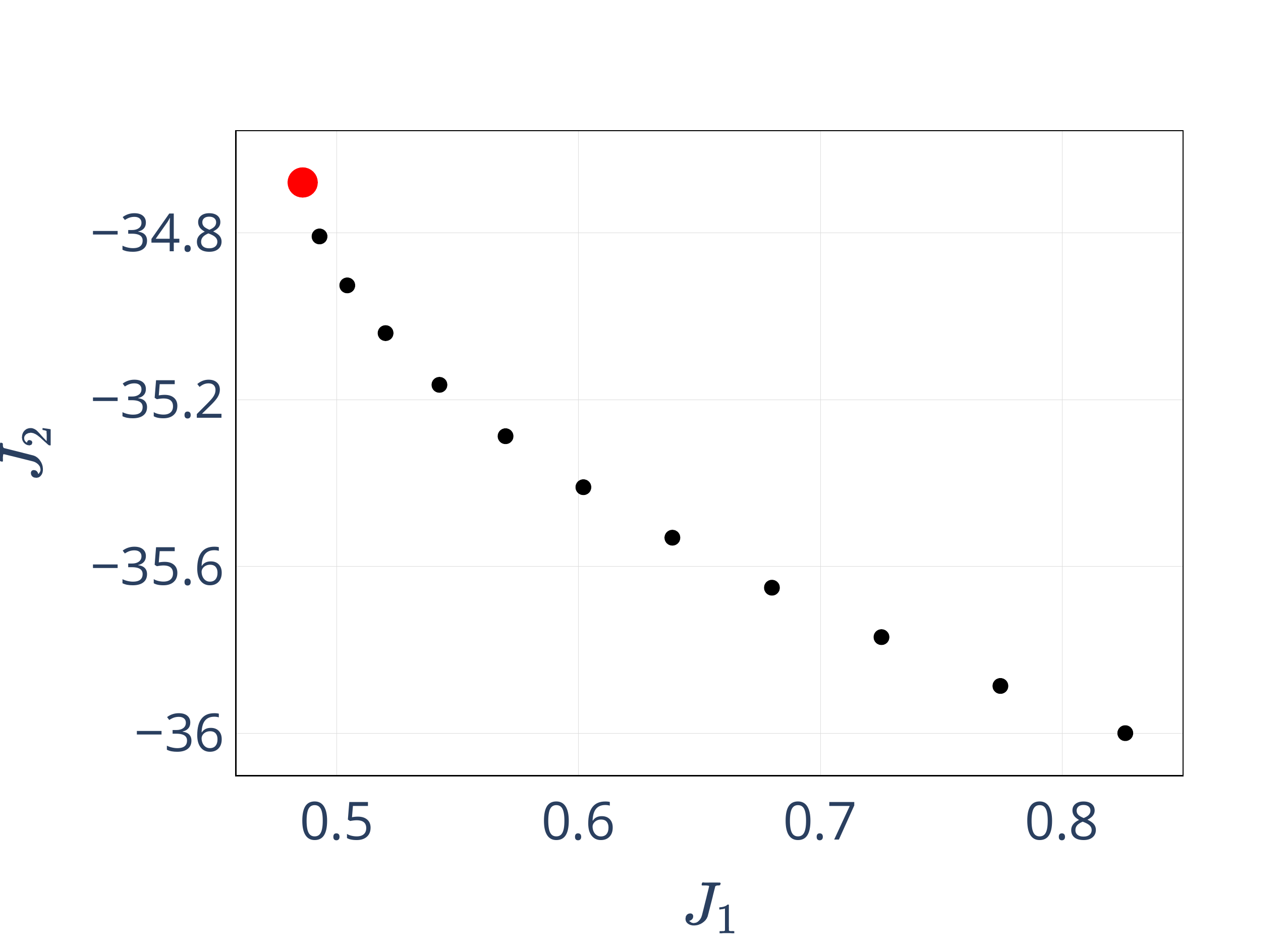}\\
		(b) "min 1"
	\end{minipage}
	\begin{minipage}{0.328\textwidth}
		\centering
		\includegraphics[width =\textwidth]{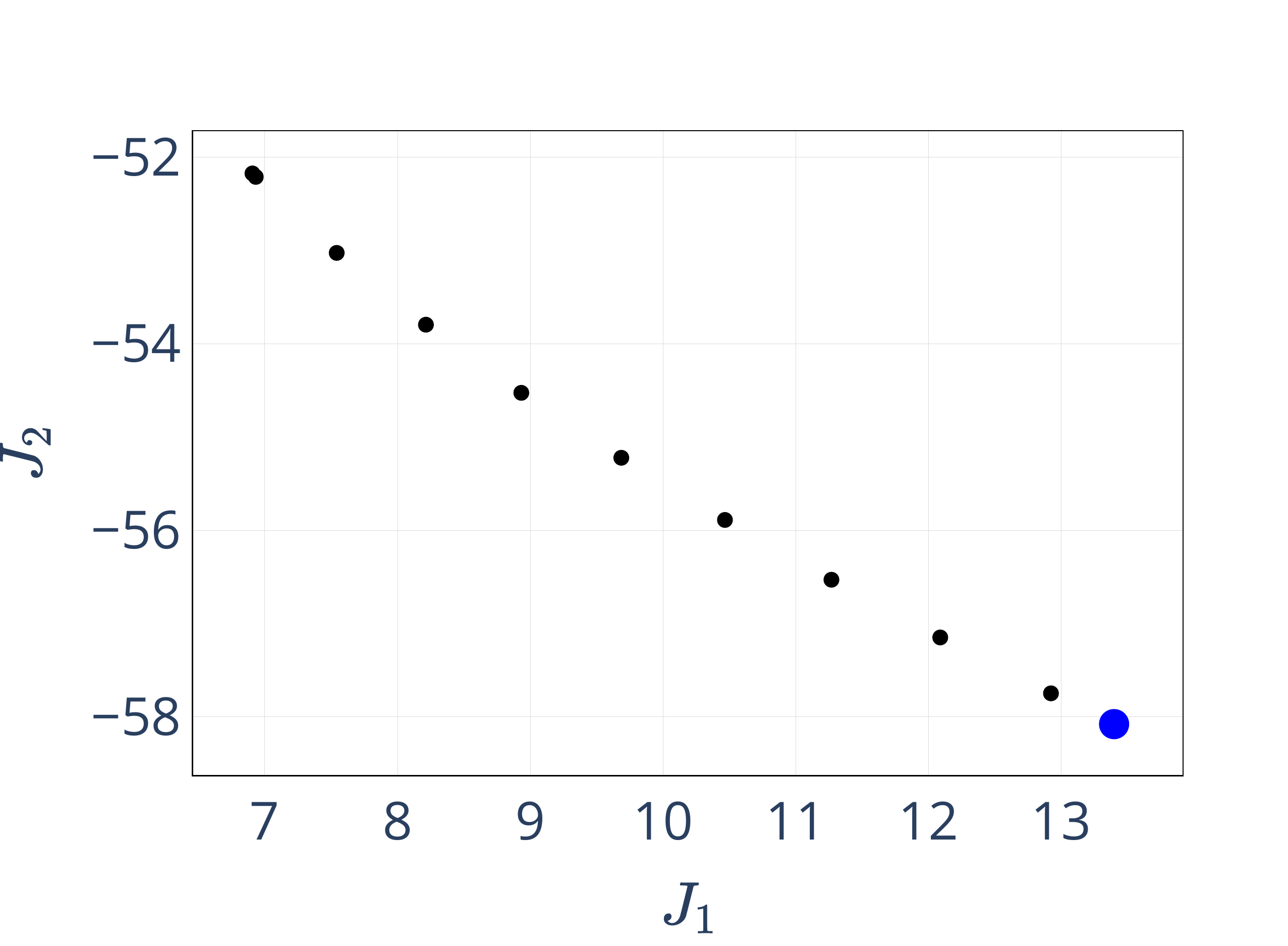}\\
		(c) "min 2"
	\end{minipage}
	\caption{Nondominated set in iteration step $k=6$ for the different selection rules}\label{fig: parheur}
\end{center}
\end{figure}

We remark that the magnitudes of the cost functionals and, thus, the size and the location of the nondominated sets are significantly different for the three selection rules. While the nondominated set for "min 2" (right) has a range from 7 to 13 and from -58 to -52, the nondominated set for "min 1" (mid) is substantially smaller with a range from 0.5 to 0.8 and from -36 to -34.6. The nondominated set for "ideal" (left) lies between those for "min 1" and "min 2". Hence, for each selection rule the subsequent efficient solutions are chosen not only according to different rules, but also from completely different nondominated sets. This indicates that the choice of the efficient solutions should have an impact on convergence rate and performance. Figure \ref{fig: traheu} illustrates the resulting closed-loop trajectories.
\begin{figure}[h]
\begin{center}
	\begin{minipage}{0.49\textwidth}
		\includegraphics[width =\textwidth]{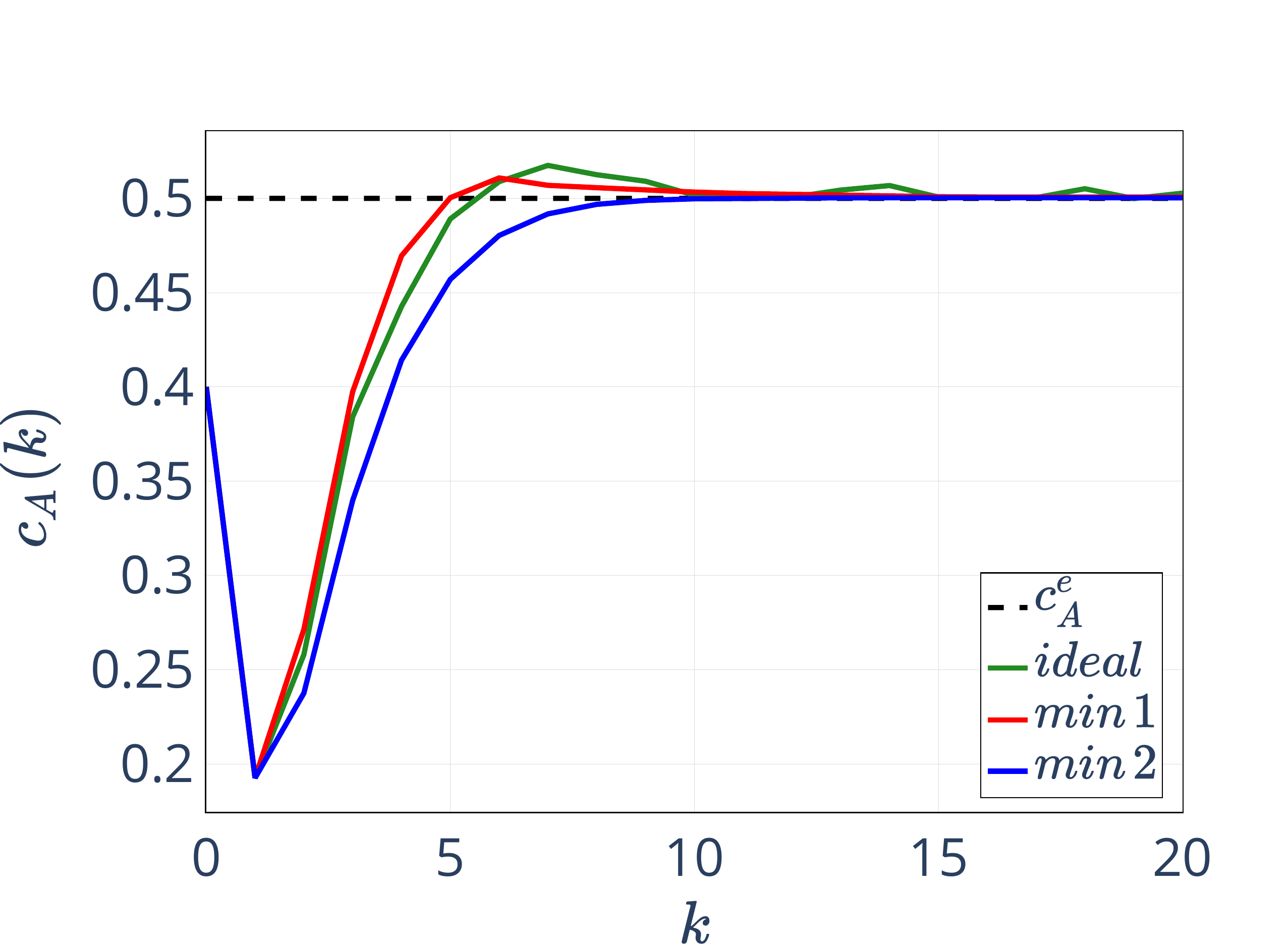}
	\end{minipage}
	\begin{minipage}{0.49\textwidth}
		\includegraphics[width =\textwidth]{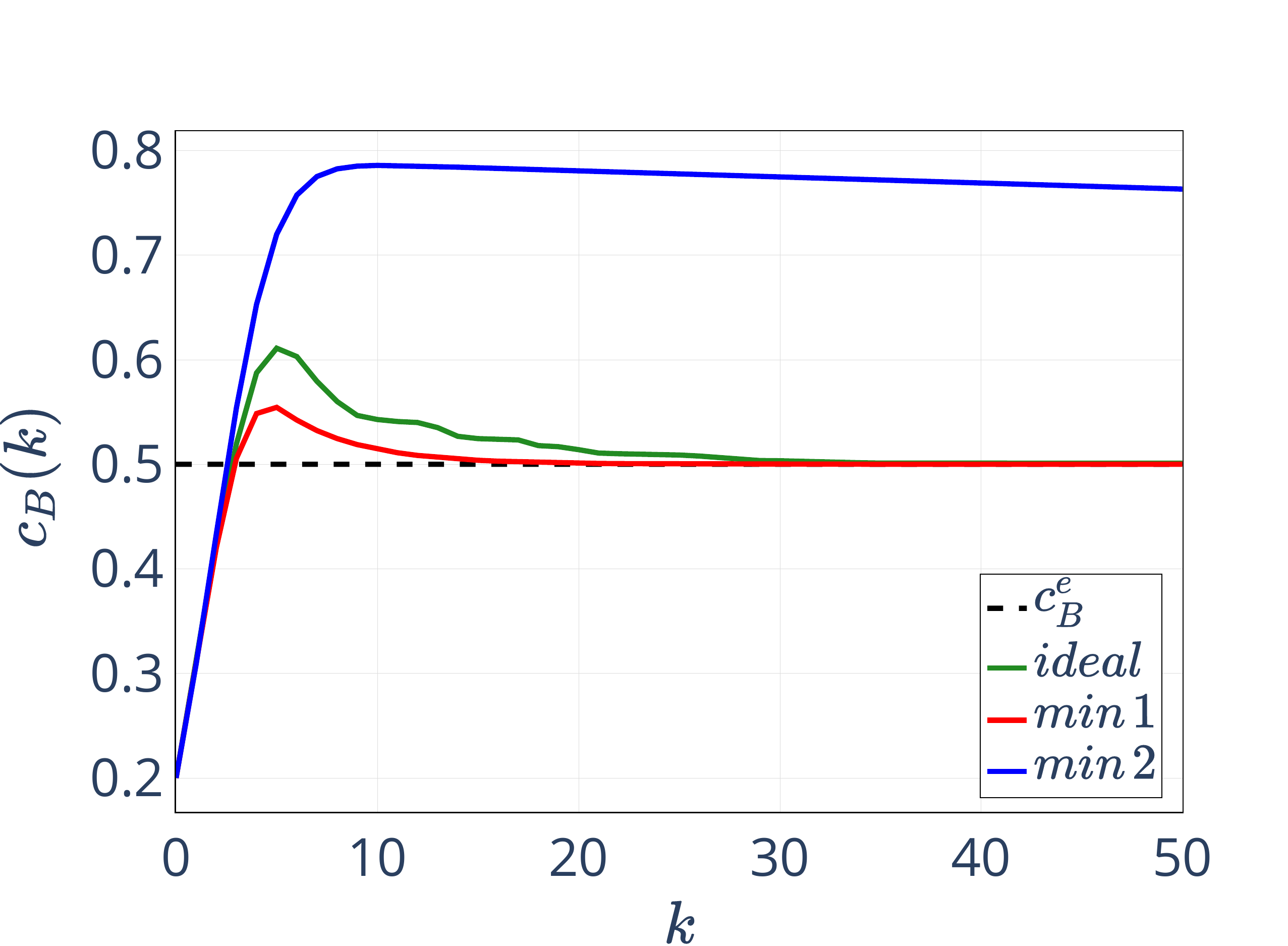}
	\end{minipage}
	\caption{Components of the closed-loop trajectory for the different heuristics}\label{fig: traheu}
\end{center}
\end{figure}

On the left-hand side in Figure \ref{fig: traheu} we observe that for the first component of the trajectory $c_A$ all selection rules deliver a similar behavior. In contrast, on the right-hand side, the behavior of the second component $c_B$ depends strongly on the selection rule. The rule "min 2" aims to maximize the economic yield. Therefore, the second component of the closed-loop has large values and converges only slowly to the equilibrium. While the trajectory of "min 1" reaches a small neighborhood of the equilibrium within 15 iteration steps, "min 2" needs about 2000 iterations to get similarly close to the equilibrium. In Figure \ref{fig: costheu}, we visualize the cost criteria $J_1^N$ and $J_2^N$.
\begin{figure}[h]
\begin{center}
	\begin{minipage}{0.49\textwidth}
		\includegraphics[width =\textwidth]{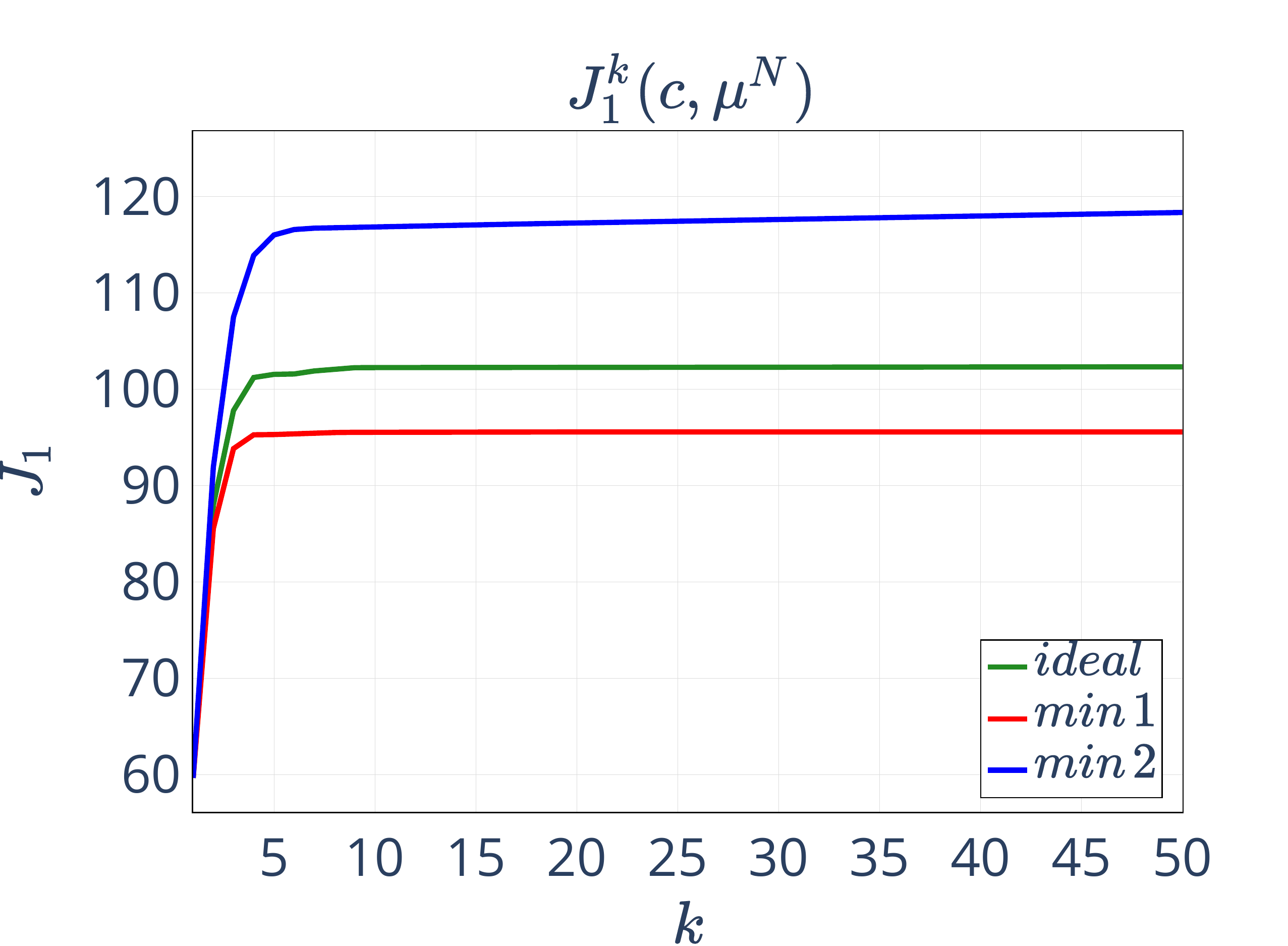}
	\end{minipage}
	\begin{minipage}{0.49\textwidth}
		\includegraphics[width =\textwidth]{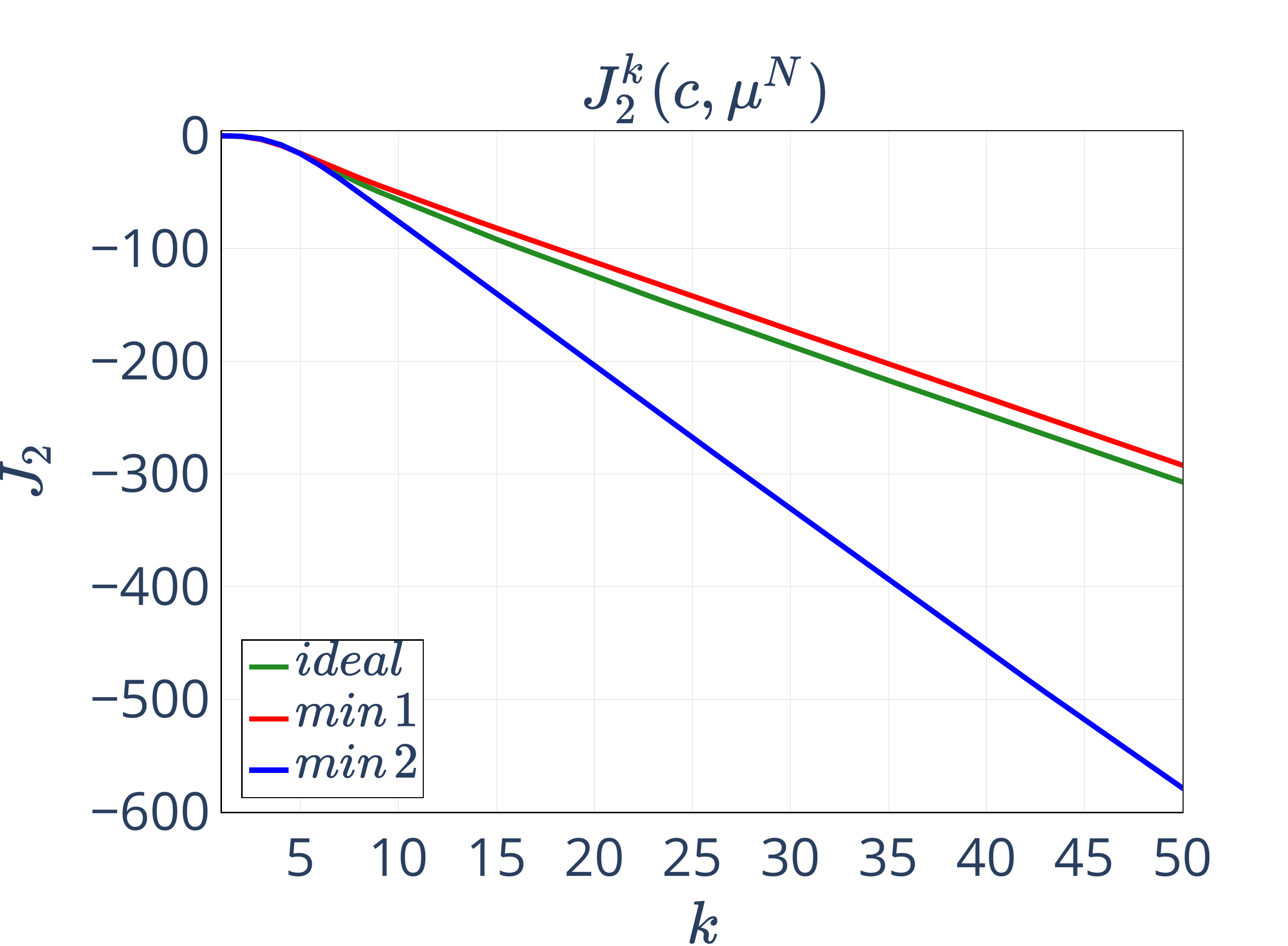}
	\end{minipage}
	\caption{Cost criteria $J_1$ and $J_2$ for the different heuristics}\label{fig: costheu}
\end{center}
\end{figure}

Here, we observe that "min 1" results in a significantly smaller $J_1$ than the other strategies, while $J_2$ is the largest, whereas "min 2" enforces exactly the opposite. In terms of the cost, the main feature of "ideal" becomes particularly clear. The selection rule "ideal" yields a compromise between both costs, which in this example turns out to be closer to "min 1" than to "min 2".
\end{bsp}

While the results in Example \ref{ex: reactor3} show precisely the behavior that one would expect from the different selection rules, a priori it was not clear that the quantitative differences are so pronounced. Indeed, as the following example shows, the effect of the different rules can also be almost negligible.

\begin{bsp}\label{ex: empc}
We consider an economic growth model, introduced in \cite{Brock1972}. The system dynamic is given by
\[x(k+1) = u(k),\quad k\in\N,\]
and the stage cost by
\[\ell_1(x,u)= -\ln(Ax^\alpha - u ),\]
with parameters $A =5$ and $\alpha=0.34$. We impose state and control constraints $\X=[0,10]$ and $\U=[0.1,5]$. As calculated in \cite{Damm2014}, the equilibrium for which the problem is strictly dissipative is given by $\equ=(\eq,\eq)\approx(2.23,2.23)$. 
We use the equilibrium to set the endpoint terminal constraint $\X_0=\{\eq\}$. Next, we introduce the second stage cost
\[\ell_2(x,u)= (x-\eq)^2+0.1(u-u^{e})^2,\]
which additionally stabilizes the equilibrium. Thus, the multiobjective optimal control problem  reads
\begin{align}
\min_{\ub\in\U^N(x_0)}J^N(x_0,\ub)&=\left(\sum_{k=0}^{N-1}\ell_1(x(k,x_0),u(k)), \sum_{k=0}^{N-1}\ell_2(x(k,x_0),u(k))\right)\nonumber\\
\text{s.t.}\quad x(k+1)&=u(k)\nonumber\\
x(0)&=5\nonumber\\
x(N,x_0)\in\X_0&=\{(2.23,2.23)\}\nonumber\\
\X&=[0,10],\; \U=[0.1,5]
\end{align}
Due to the strict dissipativity, the stage cost $\ell_1$ fulfills the required Assumption \ref{ass: modTerminal} and \ref{ass: cont1}. Since we have introduced endpoint terminal constraints Assumption \ref{ass: kappabound} holds with the same argument as in Example \ref{ex: reactor2}. Hence, this example fits in our theoretical setting. Now we check whether the choice of the subsequent efficient solution in Algorithm \ref{alg: modMOMPC2} has an impact on the solution behavior. To this end, we consider the MPC-horizon $N=10$ and the different selection rules as in Example \ref{ex: reactor3}. We chose the first efficient solution such that $J^{10}(x_0,\posN)=(-15.085,7.892)$.
\begin{figure}[h]
\begin{center}
	\begin{minipage}{0.49\textwidth}
		\includegraphics[width =\textwidth]{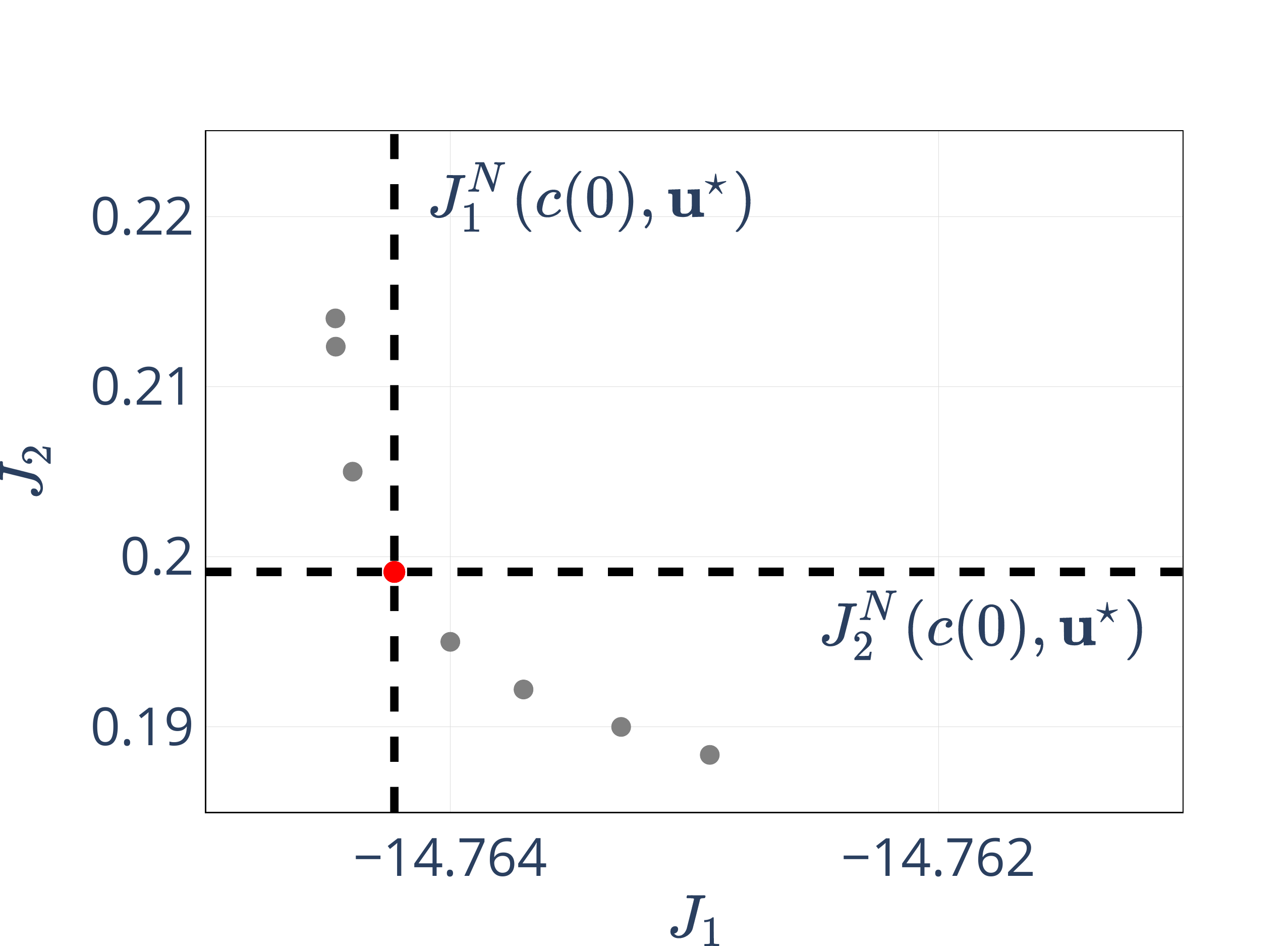}
		\caption{Resulting nondominated set} \label{fig: paretoempc}
	\end{minipage}
	\begin{minipage}{0.49\textwidth}
		\includegraphics[width =\textwidth]{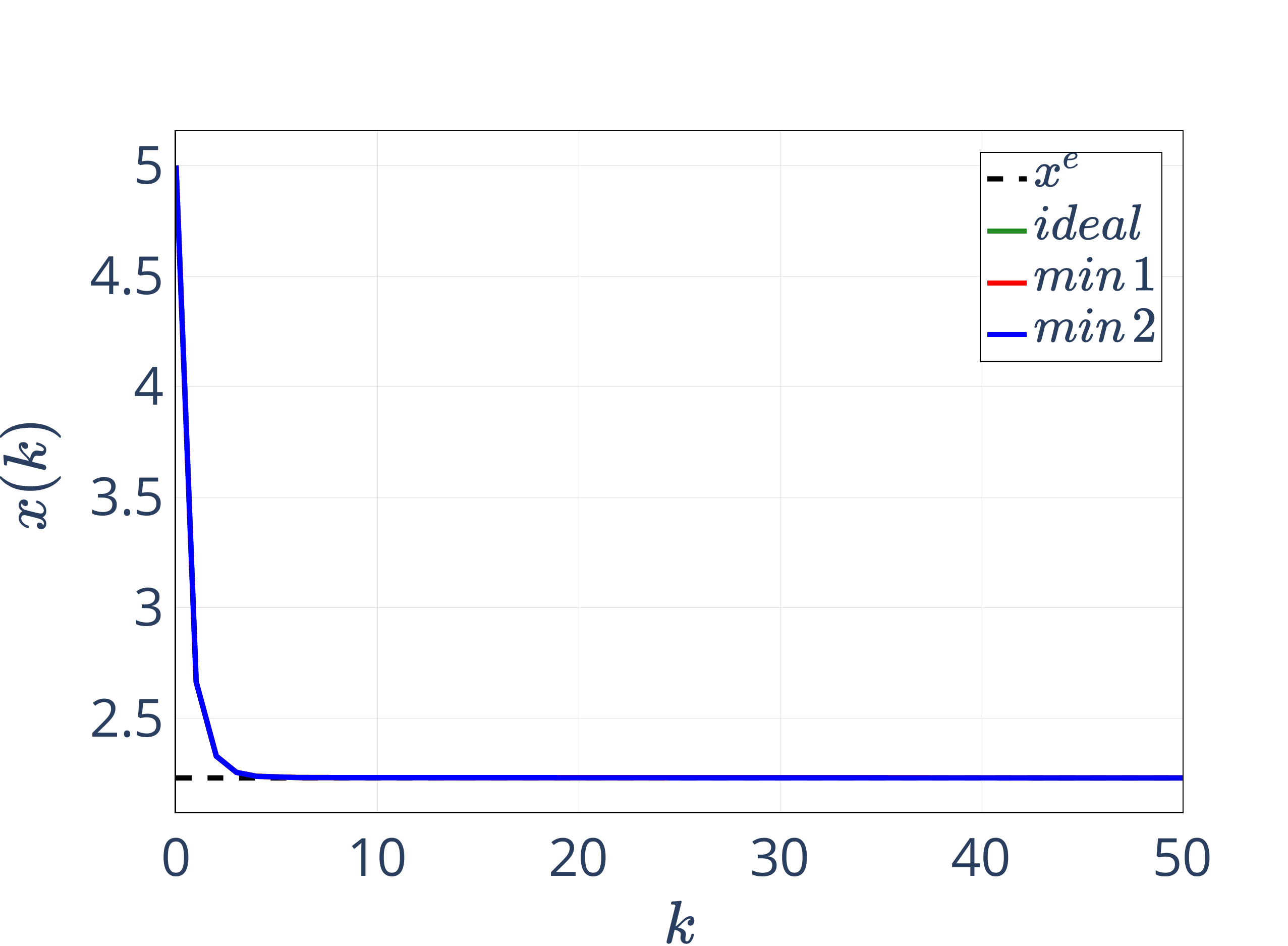}
		\caption{Closed-loop trajectories}\label{fig: empctra}
	\end{minipage}
\end{center}
\end{figure}

In Figure \ref{fig: paretoempc} we observe that in the second iteration only one single point is cut out of the nondominated set and, thus, there is no more degree of freedom in choosing the efficient solution. This suggests that the selection rules have no impact on the solution behavior. This is confirmed by Figure \ref{fig: empctra}, as there are no differences---except for numerical inaccuracies---in the trajectories resulting from the selection rules. The same phenomenon is reflected in the costs. Thus, we can conclude that for this example the choice of the subsequent efficient solutions has no influence on the behavior of the trajectory and the cost criteria.
%the trajectories do not differ from each other. For this reason, we added in Figure \ref{fig: empctra} on the right side the deviations. There, we illustrate the relative differences of the trajectories corresponding to "min 1" and "min 2", respectively, to the trajectory which arises through "ideal". The plot shows that all considered selection rules deliver virtually the same trajectory. 
%	\begin{figure}[h]
%		\begin{center}
	%			\begin{minipage}{0.49\textwidth}
		%				\includegraphics[width =\textwidth]{paper_mompc/figures/Cost1_empc_relError.pdf}
		%			\end{minipage}
	%			\begin{minipage}{0.49\textwidth}
		%				\includegraphics[width =\textwidth]{paper_mompc/figures/Cost2_empc_relError.pdf}
		%			\end{minipage}
	%			\caption{Deviation of the costs $J_1$ and $J_2$}\label{fig: empccost}
	%		\end{center}
%	\end{figure}
%The same phenomenon is reflected in the costs. 
%The relative deviation for both cost criteria $J_1$ and $J_2$ is visualized in Figure \ref{fig: empccost}. Here, the relative deviation is also in the range of $10^{-6}$. 
%Thus, we can conclude that for this example the choice of the subsequent efficient solutions has no influence on the behaviour of the trajectory and the cost criteria.
\end{bsp}

In summary, we can say that the implementation of different selection rules for the subsequent efficient solutions can make a significant difference for the resulting closed loop trajectories and costs, as seen in Example \ref{ex: reactor3}. In contrast, Example \ref{ex: empc} shows that this difference may also be negligible.

\section{Conclusion and Outlook}\label{sec: conclusion}

In this paper we have introduced a new multiobjective MPC algorithm for which we require strong assumptions, i.e., strict dissipativity and the existence of a compatible terminal cost, only for the first cost criterion. For this algorithm we have shown a performance estimate for the first cost function $J_1^N$ as well as averaged performance estimates for all cost function $J_i^N$, $i=1,\dots,s$. Under suitable technical assumption we have shown asymptotic stability of the closed-loop trajectory using time-varying Lyapunov functions. Further, for the algorithm introduced in \cite{Gruene2019a} and using our assumption we state a performance theorem for the other cost functions $J_i^N$, $i=2,\dots,s$, again without requiring strict dissipativity for these costs. In addition, we have numerically illustrated our theoretical results and, in the course of this, investigated the influence on the solution behavior of selection rules for the choice of the subsequent efficient solutions. For this influence we have shown that, depending on the concrete example, it may be significant or negligible.

references.bibFor future research it would be interesting to investigate theoretical performance estimates for different selection rules. In Section \ref{sec:numerics2} we have observed that for some selection rules the upper bound of the performance of the cost function is sharp, while for other selection rules it is not. Thus, the question arises whether we can refine our estimates by taking into account the selection rules from the second iteration on is still open. Another point is the efficiency of the resulting closed-loop trajectory on the infinite time horizon and the corresponding MPC feedback. The question whether we can state an optimality result comparable to those for the standard MPC case, see for instance \cite{Gruene2017a,Gruene2014}, still remains open.
\newline
\newline
\textbf{Data availability} The datasets generated and analyzed during the current study are available from the corresponding author on reasonable request.

{\bibliographystyle{plain}
\bibliography{references} 

\begin{thebibliography}{10}

\bibitem{GithubASMO}
{ASMO} - a solver for multiobjective optimization.
\newblock \url{https://github.com/GEichfelder/ASMO}.

\bibitem{AmRA11}
Rishi Amrit, James~B. Rawlings, and David Angeli.
\newblock Economic optimization using model predictive control with a terminal
  cost.
\newblock {\em Annual Reviews in Control}, 35:178--186, 2011.

\bibitem{AnAR09}
David Angeli, Rishi Amrit, and James~B. Rawlings.
\newblock Receding horizon cost optimization for overly constrained nonlinear
  plants.
\newblock In {\em Proceedings of the 48th IEEE Conference on Decision and
  Control -- CDC 2009}, pages 7972--7977, Shanghai, China, 2009.

\bibitem{AnAR12}
David Angeli, Rishi Amrit, and James~B. Rawlings.
\newblock On average performance and stability of economic model predictive
  control.
\newblock {\em IEEE Transactions on Automatic Control}, 57(7):1615--1626, 2012.

\bibitem{BFGV20}
Stefan Banholzer, Giulia Fabrini, Lars Gr\"une, and Stefan Volkwein.
\newblock Multiobjective model predictive control of a parabolic
  advection-diffusion-reaction equation.
\newblock {\em Mathematics}, 8(5), 2020.
\newblock Paper No. 777.

\bibitem{Brock1972}
William~A. Brock and Leonard~J. Mirman.
\newblock Optimal economic growth and uncertainty: The discounted case.
\newblock {\em Journal of Economic Theory}, 4(3):479--513, 1972.

\bibitem{Damm2014}
Tobias Damm, Lars Grüne, Marleen Stieler, and Karl Worthmann.
\newblock An exponential turnpike theorem for dissipative discrete time optimal
  control problems.
\newblock {\em {SIAM} Journal on Control and Optimization}, 52(3):1935--1957,
  2014.

\bibitem{Diehl2011}
Moritz Diehl, Rishi Amrit, and James~B. Rawlings.
\newblock A {L}yapunov function for economic optimizing model predictive
  control.
\newblock {\em {IEEE} Transactions on Automatic Control}, 56(3):703--707, mar
  2011.

\bibitem{Ehrgott2005}
Matthias Ehrgott.
\newblock {\em Multicriteria Optimization}.
\newblock Springer-Verlag, 2005.

\bibitem{Eichfelder2008}
Gabriele Eichfelder.
\newblock {\em Adaptive scalarization methods in multiobjective optimization}.
\newblock Vector Optimization. Springer, 2008.

\bibitem{Eichfelder2009}
Gabriele Eichfelder.
\newblock An adaptive scalarization method in multiobjective optimization.
\newblock {\em SIAM Journal on Optimization}, 19(4):1694--1718, 2009.

\bibitem{Eichfelder20}
Gabriele Eichfelder.
\newblock Twenty years of continuous multiobjective optimization in the
  twenty-first century.
\newblock {\em EURO Journal on Computational Optimization}, 9:100014, 2021.

\bibitem{Flasskamp2020}
Kathrin Fla{\ss}kamp, Sina Ober-Blöbaum, and Sebastian Peitz.
\newblock Symmetry in optimal control: A multiobjective model predictive
  control approach.
\newblock In {\em Advances in Dynamics, Optimization and Computation}, pages
  209--237. Springer International Publishing, 2020.

\bibitem{Grue13}
Lars Gr{\"u}ne.
\newblock Economic receding horizon control without terminal constraints.
\newblock {\em Automatica}, 49(3):725--734, 2013.

\bibitem{GruP15a}
Lars Gr\"une and Anastasia Panin.
\newblock On non-averaged performance of economic {M}{P}{C} with terminal
  conditions.
\newblock In {\em Proceedings of the 54th IEEE Conference on Decision and
  Control --- CDC 2015}, pages 4332--4337, Osaka, Japan, 2015.

\bibitem{Gruene2017a}
Lars Gr{\"u}ne and J{\"u}rgen Pannek.
\newblock {\em Nonlinear Model Predictive Control : Theory and Algorithms. 2nd
  Edition}.
\newblock Communications and Control Engineering. Springer, Cham, Switzerland,
  2017.

\bibitem{Gruene2014}
Lars Gr{\"u}ne and Marleen Stieler.
\newblock Asymptotic stability and transient optimality of economic {M}{P}{C}
  without terminal conditions.
\newblock {\em Journal of Process Control}, 24(8):1187--1196, 2014.

\bibitem{Gruene2019a}
Lars Grüne and Marleen Stieler.
\newblock Multiobjective model predictive control for stabilizing cost
  criteria.
\newblock {\em Discrete {\&} Continuous Dynamical Systems - B},
  24(8):3905--3928, 2019.

\bibitem{Kajgaard2013}
Mikkel~Urban Kajgaard, Jesper Mogensen, Anders Wittendorff, Attila~Todor
  Veress, and Benjamin Biegel.
\newblock Model predictive control of domestic heat pump.
\newblock In {\em Proceedings of the 2013 American Control Conference}, pages
  2013--2018. {IEEE}, 2013.

\bibitem{Logist2010}
Filip Logist, Boris Houska, Moritz Diehl, and Jan~Van Impe.
\newblock Fast {P}areto set generation for nonlinear optimal control problems
  with multiple objectives.
\newblock {\em Structural and Multidisciplinary Optimization}, 42(4):591--603,
  2010.

\bibitem{Miettinen1998}
Kaisa Miettinen.
\newblock {\em Nonlinear Multiobjective Optimization}.
\newblock Springer US, 1998.

\bibitem{RaMD17}
James~B. Rawlings, David~Q. Mayne, and Moritz~M. Diehl.
\newblock {\em Model Predictive Control: Theory, Computation and Design}.
\newblock Nob Hill Publishing, Madison, Wisconsin, 2017.

\bibitem{Sauerteig}
Philipp Sauerteig and Karl Worthmann.
\newblock Towards multiobjective optimization and control of smart grids.
\newblock {\em Optimal Control Applications and Methods}, 41:128--145, 2019.

\bibitem{Sawaragi1985}
Yoshikazu Sawaragi.
\newblock {\em Theory of multiobjective optimization}.
\newblock Academic Press, Orlando, 1985.

\bibitem{Schmitt2020}
Thomas Schmitt, Tobias Rodemann, and Jürgen Adamy.
\newblock Multi-objective model predictive control for microgrids.
\newblock {\em at - Automatisierungstechnik}, 68(8):687--702, 2020.

\bibitem{Stieler2018}
Marleen Stieler.
\newblock {\em Performance Estimates for Scalar and Multiobjective Model
  Predictive Control Schemes}.
\newblock PhD thesis, Universität Bayreuth, Bayreuth, 2018.

\bibitem{Zavala2015}
Victor~M. Zavala.
\newblock A multiobjective optimization perspective on the stability of
  economic {MPC}.
\newblock {\em {IFAC}-{PapersOnLine}}, 48(8):974--980, 2015.

\bibitem{Zavala2012}
Victor~M. Zavala and Antonio Flores-Tlacuahuac.
\newblock Stability of multiobjective predictive control: A utopia-tracking
  approach.
\newblock {\em Automatica}, 48(10):2627--2632, 2012.

\end{thebibliography}
}

\end{document}